%% file: smoothing_arXiv.tex
\documentclass[12pt]{article}
\input{smoothing_preamble}
\input{smoothing_commands}
\bibinput{smoothing_bib}

\title{Kernel smoothing on manifolds}
\author{Eunseong Bae\thanks{Department of Statistics, University of California, Davis. Email: \texttt{esbae@ucdavis.edu}}\,}
\author{Wolfgang Polonik\thanks{Department of Statistics, University of California, Davis. Email: \texttt{wpolonik@ucdavis.edu}}}
\affil{}
\date{}

\begin{document}
\maketitle

\begin{abstract}
	Under the assumption that data lie on a compact (unknown) manifold without boundary, we derive finite sample bounds for kernel smoothing and its (first and second) derivatives, and we establish asymptotic normality through Berry-Esseen type bounds.
	Special cases include kernel density estimation, kernel regression and the heat kernel signature. 
	Connections to the graph Laplacian are also discussed.
\end{abstract}


\section{Introduction}

Manifolds, which are geometric objects locally isomorphic to Euclidean spaces, play an important role in high-dimensional data analysis.
The so-called manifold hypothesis frequently employed in statistical machine learning postulates that high-dimensional data lie on (or close to) a low-dimensional manifold.
This assumption amounts to a complexity reduction of the underlying statistical model. 
Moreover, the local isomorphy to Euclidean spaces facilitates the application of diverse methodologies used in Euclidean settings to analyze the geometric structure.

Our work provides novel insights into kernel smoothing by establishing a {\em unified point of view} to various smoothing operations, providing {\em finite sample approximations} for smoothing estimators of not only the functions themselves, but in particular of the {\em derivatives of the functions} using the {\em uniform} norm for data sampled on Riemannian manifolds. This includes kernel density estimation, kernel regression, heat kernel smoothing, or kernel weighted graph Laplacians as special cases.

We assume that a function on a manifold is observed either with or without noise at data sampled from a manifold. Kernel smoothing then attempts to recover the function and also its derivatives over the entire manifold.  
We establish finite sample sup-norm concentration bounds and Berry-Esseen type bounds. 
These results then imply rates of uniform consistency and asymptotic normality of general kernel smoothing on a manifold. The fact that data are sampled from a manifold makes the derivations technically challenging.

Special cases have been considered in the literature.
A classical subject is the kernel density estimator (KDE) on spheres (or for directional data), including \cite{Hall1987Spherical, Bai1988Directional}.
The KDE on the $d$-dimensional torus was studied in \cite{DiMarzio2011torus}. 
Other works on KDEs on manifolds include \cite{jiang2017uniform} and \cite{henry2009KernelDensity}, who derive uniform consistency of the kernel density estimator. 
The latter showed strong uniform consistency and asymptotic normality. 
\cite{wu2022strong} derives uniform convergence rates of KDEs with general kernel functions.
\cite{bouzebda2023RatesStrong} derives the uniform convergence rate for KDEs and kernel regression with general kernel functions. Their rates are slightly slower than ours.
Pointwise consistency of kernel regression can also be found in \cite{khardani2022NonparametricRecursive} and \cite{cheng2013local}.
\cite{gine2006empirical} show uniform convergence and asymptotic normality of the kernel graph Laplacian under uniform sampling. \cite{berenfeld_density_2021} derive $L^2$-minimax bounds for kernel density estimation on unknown manifolds.

Prior work on smoothing on manifolds has mainly focused on the function itself, and only a few results treat its derivatives.
To our knowledge, \cite{zhang2021KernelSmoothing} is the only work in this direction.
Specifically, they show uniform convergence rates for the KDE and its derivatives on a general unit sphere.
They consider an entire class of kernel functions, including the Gaussian kernel, and their rates are the same as our results if specified to their setup. Heat kernel smoothing has not been considered before.

The structure of this paper is as follows.
Section \ref{smoothing:prelim} briefly introduces background knowledge on manifolds and notation that will be used throughout this paper.
Section \ref{unif} investigates uniform consistency of kernel smoothing.
Asymptotic normality of kernel smoothing is covered in Section \ref{Berry}.
Section \ref{deriv} discusses uniform consistency of derivatives of kernel smoothing.
In Section \ref{smoothing:application}, results from prior sections are applied to kernel density estimation and heat kernel signature smoothing.
Simulation studies and proofs are postponed to the appendix.

\section{Preliminaries} \label{smoothing:prelim}

A topological space $\M$ is called a $d$-dimensional manifold if each point $x \in \M$ admits a neighborhood homeomorphic to an open subset of $\R^d$ via a homeomorphism $\phi_x$.
If $\phi_x$ is $C^\infty$ for all $x \in \M$, then $\M$ is called a smooth $d$-dimensional manifold.
The image of the derivative of $\phi_x$ at $x$, denoted by $T_x \M$, is called the tangent space of $\M$ at $x$, and its elements are called tangent vectors.
A Riemannian metric on $\M$ assigns to each $x \in \M$ an inner product on $T_x \M$, and a smooth manifold equipped with a Riemannian metric is called a Riemannian manifold. 
The volume measure $\mathrm{vol}_\M$ is the measure on $\M$ induced by the Riemannian metric through the associated Riemannian volume form.

For a real-valued function $f$ defined on $\M$, we denote the gradient of $f$ at $x$ by $\nabla_\M f (x)$.
It is the unique tangent vector in $T_x \M$ associated with the differential map of $f$ through the Riemannian metric.
We also denote the Hessian of $f$ at $x$ by $\nabla_\M^2 f (x)$.
It is the bilinear form on $T_x \M \times T_x \M$ uniquely determined by $f$ and the Riemannian metric.
The Laplace-Beltrami operator applied to $f$, denoted by $\Delta_\M f$, is defined as $\Delta_\M f = -\div_\M \nabla_\M f$ where $\div_\M$ denotes the divergence operator on $\M$.

For a vector $v \in \R^k$, $\|v\|_{\R^k}$ denotes the Euclidean norm of $v$.
For a matrix $A$, $\|A\|_\F$ denotes the Frobenius norm of $A$. 
To define the norms of gradients or Hessians on $\M$, note that for any $x \in \M$, $T_x \M$ can be identified with $\R^d$,
so that we can regard $\Exp_x$ as a function from $\R^d$ to $\M \subset \R^D$ where $\Exp_x$ denotes the exponential map at $x$ with $\Exp_x (0) = x$.
Under this identification, for a real-valued function $f$ defined on $\M$, we can define the norm of $\nabla_\M f (x)$ as
the Euclidean norm of the gradient $\nabla_{\R^d} (f \circ \Exp_x) (0)$, that is,
$\|\nabla_\M f (x) \|_{T_x \M} = \| \nabla_{\R^d} (f \circ \Exp_x) (0) \|_{\R^d}$.
Similarly, we define the norm of $\nabla_\M^2 f (x)$ as the Frobineus norm of the Hessian matrix $\nabla_{\R^d}^2 (f \circ \Exp_x) (0)$, that is, $\|\nabla_\M^2 f (x) \|_{T_x \M \times T_x \M} = \| \nabla_{\R^d}^2 (f \circ \Exp_x) (0) \|_\F$.
For simplicity, we suppress the subscripts $T_x$ and $T_x\M \times T_x \M$ in the norms when no confusion arises.

For non-negative $k$, we write $f \in C^k(\M)$ if $f$ is a real-valued $C^k$-function defined on $\M$.
By convention, we write $f \in C^0(\M)$ to mean that $f$ is bounded and measurable. 
Additionally, for $\alpha \in (0,1]$, we write $f \in C^{k,\alpha}(\M)$ if $f \in C^k(\M)$ and its $k$th partial derivatives are uniformly $\alpha$-H\"{o}lder continuous.

\section{Uniform consistency of kernel smoothing on manifolds} \label{unif}

Let $X_1,\ldots,X_n$ be observations drawn from a Riemannian manifold $\M \subset \R^D$.
We begin with studying statistics of the form
$$\frac{1}{n} \sum_{i=1}^{n} K_\e \left( \|X_i - x \|_{\R^D} \right) f(X_i)$$
where $K_\e$ is the kernel function with the bandwidth $\e>0$ and $f$ is a given, {\em known} function, and we also consider normalized versions (see below). 
This includes a kernel density estimator (for $f \equiv 1$), and also a Nadaraya-Watson type regression estimator (after normalization - see below), respectively, as special cases.
Note that these weights use the distance in the ambient space, rather than the geodesic distance on $\M$.
This discrepancy is inevitable in a practical sense since the geodesic distance cannot be used if $\M$ is unknown,
and even for a known case, calculating all pairs of geodesic distances usually is computationally expensive.
However, in a small neighborhood, these two distances are very close to each other within acceptable bounds (cf. Proposition \ref{smoothing:aux:expansion} in the appendix).
That is, under the right assumptions, for sufficiently small bandwidth, kernel smoothing on manifolds using the ambient space metric can still work properly.

We first state the basic assumptions that are commonly used in manifold settings.

\begin{assumption} \label{smoothing:assume:manifold}
	$\M$ is a $d$-dimensional compact Riemannian submanifold without boundary isometrically embedded in $\R^D$.
\end{assumption}

\begin{assumption} \label{smoothing:assume:density}
	$\P$ is a probability measure on $\M$ with a smooth density $\rho$ with respect to the volume measure on (that is, $d\P = \rho \hspace{.25em} d\mathrm{vol}_{\M}$). The minimum of $\rho$, denoted by $\rho_\tmin$, is strictly positive. $X_1, \dots, X_n$ are i.i.d. samples from $\P$.
\end{assumption}

\begin{assumption} \label{smoothing:assume:kernel}
	$K_\e: \R \rightarrow \R$ is the Gaussian kernel of the form
	$K_\e \left( u\right) = (2 \pi\e^2)^{-d/2} \exp \big(-\frac{1}{2\e^2} u^2\big). $
\end{assumption}

Note that $K_\e$ is positive, symmetric, differentiable, and integrable over $\R$ with $$\int_{\R^d} K_\e(\|u\|_{\R^d})du = 1.$$ 
Also, all the moments of $K_\e$ exist.

{\em In our theoretical results, unless stated otherwise, we will always assume that Assumptions~\ref{smoothing:assume:manifold} - \ref{smoothing:assume:kernel} hold.}

In this work, we restrict attention to this Gaussian kernel because of its quickly decaying tails, which allows for strong control of the effect of replacing the intrinsic (geodesic) distance by the Euclidean distance.
With a different choice of kernel, similar results as stated in this paper can be obtained. 
However, convergence rates would be slower if the kernel lacks sufficient smoothness or has heavy tails.

Next, we introduce our quantities of interest:
For $f: \M \to \R$, $x \in \M$ and $\e>0$, the unnormalized kernel smoothing operator and its expected value are

\begin{align}
	\T_{n,\e} [f] (x)
	&= \frac{1}{n} \sum_{i=1}^{n} K_\e \big( \|X_i - x \|_{\R^D}\big) f(X_i), \\
	\T_{\e} [f] (x)
	&= \E \, \T_{n,\e} [f] (x) 
	= \int_{\M} K_\e \big( \|u - x \|_{\R^D} \big) f(u) d\P (u).
\end{align}
We also consider their normalized versions:
\begin{align*}
	\TT_{n,\e} [f] (x)
	&=  \frac{\T_{n,\e} [f] (x)}{\T_{n,\e} [1] (x)}
	=\frac{ {\sum_{i=1}^{n}} K_\e \big( \|X_i - x \|_{\R^D} \big) f(X_i) }{{\sum_{i=1}^{n}} K_\e \big( \|X_i - x \|_{\R^D} \big)}, \\
	\TT_{\e} [f] (x)
	&= \frac{\T_{\e} [f] (x)}{\T_{\e} [1] (x)}
	= \frac{ \int_{\M} K_\e \big( \|u - x \|_{\R^D} \big) f(u) d\P (u) }{ \int_{\M} K_\e \big( \|u - x \|_{\R^D}\big) d\P (u) }.
\end{align*}

Many of our results consist of finite sample bounds with constants depending on the underlying manifold $\M$ through various geometric quantities.
For convenience, we define the following notion of $\M$-dependent constants:

\begin{definition}[$\M$-dependence] 
	A positive constant is called $\M$-dependent if it depends on $\M$ through its dimension $d$, volume, diameter, injectivity radius, as well as the uniform bounds on the Riemann curvature tensor and on the second fundamental form, together with their covariant derivatives.
\end{definition}

To establish the uniform convergence of $\T_{n,\e}[f]$ (resp. $\TT_{n,\e}[f]$),
we decompose the analysis into the 
the convergence of $\T_{n,\e}[f]$ to $\T_{\e}[f]$ (resp. $\TT_{n,\e}[f]$ to $\TT_{\e}[f]$), the stochastic part,
and the convergence of $\T_{n,\e}[f]$ to $\rho f$ (resp. $\TT_{\e}[f]$ to $f$), the bias part.

\subsection{Controlling the bias}

Our first result is on the uniform convergence of the bias. The higher-order terms introduced below will be used to discuss bias-correction and its connection to graph Laplacians.
To this end, consider the $s$-weighted ($s > 0$) Laplace-Beltrami operator:
$$\Delta_{\M,s} f =- \frac{1}{\rho^s} \div_\M \big( \rho^s \nabla_\M f \big).$$
Note that the dependence on $\rho$ is not indicated in the notation. 
The standard Laplace-Beltrami operator $\Delta _\M$ corresponds to the uniform distribution on $\M$.

Expansions similar to the one presented in the following theorem can be found in the literature on empirical approximations of the Laplace-Beltrami operator (see, e.g., Lemma 8 of \cite{coifman2006diffusion}).

\begin{theorem} \label{unif:bias}
	Suppose $f \in C^{2,\alpha}(\M)$ with $\alpha \in (0,1]$.
	Then, there exist $\e_0,C_0>0$ such that if $\e < \e_0$, we have 
	\begin{equation} 
		\sup_{x \in \M} \bigg| \T_{\e} [f] (x) - \rho(x)f(x)\; +
		\frac{\e^2}{2}\Big[ c(x)\rho(x)f(x) + \Delta_\M (\rho f)(x)\Big] \bigg|
		\leq C_0 \e^{2+\alpha},
		\label{unif:bias-res1}
	\end{equation}
	\begin{equation}	
		\sup_{x \in \M} \left| \TT_{\e} [f] (x) - f(x) + \frac{\e^2}{2} \Delta_{\M,2} f (x)\,\right| 
		\le C_0 \e^{2+\alpha},
		\label{unif:bias-res2}
	\end{equation}
	where $c(x)$ depends only on the curvature of $\M$ at $x$.
	The constants $\e_0$ and $C_0$ are $\M$-dependent, and they also depend on the $C^{2,\alpha}$-norms of $f$ and $\rho$.
	For \eqref{unif:bias-res2}, $\e_0$ and $C_0$ additionally depend on $\rho_\tmin$.
	
	If we only assume $f \in C^2(\M)$, then the right-hand sides in (\ref{unif:bias-res1}) and (\ref{unif:bias-res2}) are $o(\e^2)$ as $\e \to 0$.
\end{theorem}

Note that the first order term of $\T_{\e} [f] (x)$ involves not only $f$ but also the density function $\rho$.
The normalization then removes the density and also the curvature-dependent term $c(x)$ in the second order term.
We also want to remark that the convergence rate $\e^2$ may in general not be improved even with a higher-order kernel.
This is because terms involving the curvature of $\M$ on the left-hand side of (\ref{unif:bias-res1}) can be removed only if $\M$ has zero curvature everywhere, that is, $\M$ itself should be a flat space, such as a one-dimensional manifold or a torus.

We also note that if the value $f(x)$ equals zero, then the above theorem tells us that $\frac{1}{\e^2}\TT_{\sqrt{2}\e} [f] (x)$ converges to  $-\Delta_{2,\M} f(x)$ as $\e \to 0,$ while $\T_{\sqrt{2}\e} [f](x) \to -\Delta_\M(\rho f)$. Here, the bandwidth is replaced with $\sqrt{2}\e$ to remove the factor $1/2$. We will come back to this when discussing the smoothing estimation of the Laplace-Beltrami operator.

\subsection{Controlling the stochastic part}

Next, we present the uniform convergence of the stochastic part, where we observe two different rates for the normalized and the unnormalized case.
For brevity, we use the notation $\e_{n,\delta}^{*} (C_0^*) = C_0^* \left(\frac{\log n \vee \log(1/\delta)}{n}\right)^{1/d}$ for $C_0^*>0$ and $0<\delta<1$ throughout the paper.

\begin{theorem} \label{unif:sto}
	Let	$f \in C^0(\M)$.
	There exist $\e_0, C_0>0$ such that if $\e < \e_0$,
	we have with probability at least $1-\delta$,
	\begin{equation} \label{unif:sto-res1}
		\sup_{x \in \M} | \T_{n,\e} [f] (x) - \T_{\e} [f] (x) |
		\leq C_0 \bigg( \sqrt{ \frac{\log(1/(\e \wedge \delta))}{n\e^{d}} } + \frac{\log (1/(\e \wedge \delta)) }{n \e^{d}} \bigg).
	\end{equation}
	The constants $\e_0$ and $C_0$ are $\M$-dependent, and they also depend on the uniform bounds of $f$ and $\rho$. 
	
	Moreover, if $f \in C^2(\M)$, then there exists $C_0^*>0$ such that, if $\e_{n,\delta}^{*} (C_0^*) < \e < \e_0$,
	then we have with probability at least $1-\delta$, 
	\begin{equation} \label{unif:sto-res2}
		\sup_{x \in \M} \left| \TT_{n,\e} [f] (x) - \TT_{\e} [f] (x) \right|
		\leq C_0 \left( \sqrt{ \frac{\log(1/(\e \wedge \delta))}{n\e^{d-2}} } 
		+ \frac{\log (1/(\e \wedge \delta)) }{n \e^d} \right).
	\end{equation} 
	The constant $C_0^*$ is $\M$-dependent, and it also depends on $\rho_\tmin$ and the $C^2$-norms of $f$ and $\rho$.
	The constants $\e_0$ and $C_0$ also follow these additional dependencies.
\end{theorem}

Note that the upper bound depends on the dimension $d$ of $\M$, rather than the dimension of the ambient space $\R^D$. 

Compared to the unnormalized version, a lower bound for $\e$ is assumed for the normalized case.  This lower bound can be interpreted as requiring $n$ to be large enough.
This is needed to ensure that the denominator $\T_{n,\e} [1] (x)$ is bounded away from 0 w.h.p., and hence $\TT_{n,\e} [f] (x)$ is properly bounded.

As already indicated, the convergence rate is typically faster for the normalized case. Indeed, if $\delta \approx \e$ and $\frac{\log(1/\e)}{n \e^d}$ is sufficiently small, then the dominant term in the upper bound is $\sqrt{\frac{\log(1/\e)}{n \e^{d-2}}}$, which is faster than $\sqrt{\frac{\log(1/\e)}{n \e^d}}$ from the unnormalized case. The intuition behind this is that $\TT_{n,\e} [f] (x) - \TT_{\e} [f] (x)$ is shift-invariant, meaning that  $\TT_{n,\e} [f] (x) - \TT_{\e} [f] (x) = \TT_{n,\e} [f + c] (x) - \TT_{\e} [f+ c] (x)$ for any $c \in \R$ (this does not hold in the unnormalized case). Choosing $c = f(x)$ (here $x$ is fixed(!)), we see that we actually are smoothing $\tilde f(\cdot) = f(\cdot) - f(x)$, and since the kernel localizes about $x$, so that $f(\cdot) - f(x) = O(\e)$, we gain a factor of $\e$.
(For details, see supplemental material, discussion given before Lemma~\ref{unif:sto:denom} ff.).

\subsection{Uniform consistency}

By simply combining the above two theorems, one obtains the rates for uniform consistency of the two smoothing operators.
For reference, we state this result as a corollary:

\begin{corollary} \label{unif:cons}
	Let $f \in C^2(\M)$.
	There exist $\e_0, C_0 > 0$ such that if $\e < \e_0$, we have with probability at least $1-\delta$, 
	\begin{equation} \label{unif:cons-res}
		\sup_{x \in \M} \left| \T_{n, \e} [f] (x) - \rho(x) f(x) \right|
		\leq C_0 \bigg( \e^{2} + \sqrt{ \frac{\log(1/(\e \wedge \delta))}{n\e^{d}} }  
		+ \frac{\log(1/(\e \wedge \delta))}{n \e^{d}} \bigg).
	\end{equation}
	Moreover, there exist $C_0^* > 0$ such that if $\e_{n,\delta}^{*} (C_0^*) < \e < \e_0$,
	then (\ref{unif:cons-res}) with $\sqrt{ \frac{\log(1/(\e \wedge \delta))}{n\e^{d}} }$ replaced by $\sqrt{ \frac{\log(1/(\e \wedge \delta))}{n\e^{d-2}} }$ also holds for $\sup_{x \in \M}|\TT_{n,\e} [f] (x) - f(x)|$.
	The dependence of the constants $\e_0,C_0,C_0^*$ on $\M, f, \rho$ follows directly from Theorems~\ref{unif:bias} and \ref{unif:sto}.
\end{corollary}

\section{Asymptotic normality and Berry-Esseen bounds} \label{Berry}

We use the following notation: 
$(z_j)_{j=1}^m$ is an $m$-dimensional vector with $j$th component $z_j$,
and $\diag(z_j)_{j=1}^m$ or $\diag(z_1,...z_m)$ is the $m \times m$ diagonal matrix with $j$th diagonal component  $z_j$.
For $\mu \in \R^m$ and $\Sigma \in \R^{m \times m}$, $\Sigma \ge 0$, 
$\cN_m \left(\mu, \Sigma \right)$ denotes the $m$-dimensional normal with mean $\mu$ and covariance matrix 
$\Sigma$, and $\Phi_{m,\mu,\Sigma}$ denotes the corresponding probability measure.
$\cvx(\R^m)$ is the collection of all convex subsets of $\R^m$.

First, consider the unnormalized case.
Below we state several results on multivariate asymptotic normality for kernel smoothing vectors 
\begin{equation*}
	\sqrt{n \e^d} \big( \T_{n,\e} [f] (x_j) - \T_{\e} [f] (x_j) \big)_{j=1}^m,
\end{equation*}
with $x_1\ldots, x_m \in \M.$ 
To this end, we derive the Berry-Esseen type bounds for a generalization of the KS-distance to the multivariate case, given by
\begin{align*}
	\distcvx(Z, \cN_\Sigma) := \sup_{A \in \cvx(\R^m)} | \P(Z \in A) - \Phi_{m,0,\Sigma}(A) |,
\end{align*}
where $Z \in \R^m$ is a random vector, and $\cN_\Sigma \sim \cN_m(0,\Sigma)$.
Such bounds then imply multivariate asymptotic normality of kernel smoothing.

\begin{theorem} \label{Berry:unnor}	
	Let $f\in C^2(\M)$.	
	For distinct points $x_1, \dots, x_m \in \M$ with $f(x_j) \ne 0, j=1,\dots,m$, let
	\begin{align*}
		\bfZ_{n,\e,f} &= \sqrt{n \e^d} \left( \T_{n,\e} [f] (x_j) - \T_{\e} [f] (x_j) \right)_{j=1}^m, \\
		\tbfZ_{n,\e,f} &= \sqrt{n \e^d} \left( \T_{n,\e} [f] (x_j) - \rho(x_j) f(x_j) \right)_{j=1}^m, \\
		\Sigma_{f} &= \diag \bigg( \frac{\rho(x_j) f^2(x_j)}{(4\pi)^{d/2}} \bigg)_{j=1}^m.
	\end{align*}
	There exist $\e_0,C_0,C_1,C_2 > 0$ such that if $\e < \e_0$, we have
	\begin{align}
		\distcvx(\bfZ_{n,\e,f}, \cN_{\Sigma_f}) &\leq \frac{C_1}{\sqrt{n\e^d}} + C_2 \e^{2 \wedge d}, 
		\label{Berry:unnor-res1} \\
		\distcvx(\tbfZ_{n,\e,f}, \cN_{\Sigma_f}) &\leq \frac{C_1}{\sqrt{n\e^d}} + C_2 \e^{2 \wedge d}+C_0 \sqrt{n\e^{d+4}}.
		\label{Berry:unnor-res2}
	\end{align}
	The constants $\e_0,C_0,C_1,C_2$ are $\M$-dependent, and they also depend on the $C^2$-norms of $f$ and $\rho$, $m$, and the values $f(x_j)$, $\rho(x_j), j = 1,\ldots,m$.
\end{theorem}

Observe that the limit variance of the $j$th component equals zero when $f(x_j)=0.$ This suggests a faster rate of convergence holds in this case, and we will return to this point when discussing the relation to the estimation of the Laplace-Beltrami operator below.

The following corollary is immediate:

\begin{corollary}
	In the setting of the above theorem:
	\begin{enumerate}[label=(\alph*)]
		\item If $\e \rightarrow 0$ and $n \e^d \rightarrow \infty$, we have
		\begin{equation*}
			\sqrt{n \e^d} \big( \T_{n,\e} [f] (x_j) - \T_{\e} [f] (x_j) \big)_{j=1}^m
			\to \cN_m \left( 0, \Sigma_{f} \right)
		\end{equation*}
		in distribution.
		
		\item If, additionally, $n \e^{d+4} \rightarrow 0$, then the assertion of part (a) also holds when 
		$\T_{\e} [f] (x_j)$ is replaced by $\rho(x_j) f(x_j)$, $j = 1,\ldots,m$. 
	\end{enumerate}
\end{corollary}

These results are of course in line with standard results for the KDE on $\R^d$ (see, e.g., \cite{nishiyama2011ImpossibilityWeak}).

Below is a Berry-Esseen bound for the normalized case:

\begin{theorem} \label{Berry:nor}
	Let $f\in C^2(\M)$.	For distinct $x_1, \dots, x_m \in \M$ with $\|\nabla_\M f(x_j)\| \ne 0, j=1,\dots,m$, let
	\begin{align*}
		\overline{\bfZ}_{n,\e,f} &= \sqrt{n \e^{d-2}} \left( \TT_{n,\e} [f] (x_j) - \TT_{\e} [f] (x_j) \right)_{j=1}^m, \\
		\widetilde{\overline{\bfZ}}_{n,\e,f} &= \sqrt{n \e^{d-2}} \left( \TT_{n,\e} [f] (x_j) - f(x_j) \right)_{j=1}^m, \\
		\overline\Sigma_{f} &= \diag \left( \frac{\|\nabla_\M f(x_j)\|^2}{2(4\pi)^{d/2} \rho(x_j)} \right)_{j=1}^m.
	\end{align*}
	Then, there exist $\e_0,C_0$-$C_3 > 0$ such that if $\e < \e_0$ and $n\e^d > C_3$, we have
	\begin{align}
		\distcvx(\overline\bfZ_{n,\e,f}, \cN_{\overline\Sigma_f}) &\leq \frac{C_1 \log(n\e^d)}{\sqrt{n\e^d}} + C_2 \e, 
		\label{Berry:nor-res1} \\
		\distcvx(\widetilde{\overline\bfZ}_{n,\e,f}, \cN_{\overline\Sigma_f}) &\leq \frac{C_1 \log(n\e^d)}{\sqrt{n\e^d}} + C_2 \e + C_0 \sqrt{n\e^{d+2}}.
		\label{Berry:nor-res2}
	\end{align}
	The constants $\e_0,C_0$ - $C_3$ are $\M$-dependent, and they also depend on $\rho_\tmin$, the $C^2$-norms of $f$ and $\rho$, $m$, and the values $\|\nabla_\M f(x_j)\|$, $\rho(x_j)$, $j = 1,\ldots,m$.
\end{theorem}

We see that the scaling factor $\sqrt{n\e^{d-2}}$ is used to reflect the smaller variance in the normalized case.
The Berry–Esseen convergence rates include logarithmic terms to properly control the denominator in the normalized case.
Also note that the covariance matrix depends on $f$ only through its gradient.
When the gradient is zero, the scaling factor changes, and the asymptotic covariance matrix depends on $f$ through its second-order derivative:

\begin{theorem}	\label{Berry:norcrit}
	Let $f \in C^3(\M)$, and
	let
	\begin{align*}
		\Sigma^\prime_{f} &= \diag \left( \frac{ 2\|\nabla_\M^2 f(x_j)\|^2 + | \Delta_\M f(x_j)|^2 }{16(4\pi)^{d/2}\rho(x_j)} \right)_{j=1}^m.
	\end{align*}
	For distinct critical points $x_1, \dots, x_m \in \M$ of $f$ with $\|\nabla^2_\M f(x_j)\| \ne 0, j=1,\dots,m$, there exist $\e_0,C_0,C_1,C_2,C_3 > 0$ such that if $\e < \e_0$ and $n\e^d > C_3$, then we have with $\overline\bfZ_{n,\e,f}$ and $\tilde{\overline\bfZ}_{n,\e,f}$ from Theorem~\ref{Berry:unnor} 
	\begin{align}
		\distcvx(\e^{-1}\overline\bfZ_{n,\e,f}, \cN_{\Sigma^\prime_f}) &\leq \frac{C_1 \log(n\e^d)}{\sqrt{n\e^d}} + C_2 \e, 
		\label{Berry:norcrit-res1} \\
		\distcvx(\e^{-1}\tilde{\overline\bfZ}_{n,\e,f}, \cN_{\Sigma^\prime_f}) &\leq \frac{C_1 \log(n\e^d)}{\sqrt{n\e^d}} + C_2 \e + C_0 \sqrt{n\e^{d}}.
		\label{Berry:norcrit-res2}
	\end{align}
	The constants $\e_0,C_0$ - $C_3$ are $\M$-dependent, and they also depend on $\rho_\tmin$, the $C^3$-norms of $f$ and $\rho$, and the values $\|\nabla^2_\M f(x_j)\|$, $|\Delta_\M f(x_j)|$, $\rho(x_j)$, $j = 1,\ldots,m$.
\end{theorem}

Note that the right-hand side in \eqref{Berry:norcrit-res2} cannot converge to zero.
This indicates that for the asymptotic normality of $\TT_{n,\e} [f] (x)$ at a critical point $x \in \M$,
the bias term cannot be removed without further assumptions.

\subsection{Relation to the estimation of the Laplace-Beltrami operator} \label{Berry:LB-operator}

There is a close relation of normalized kernel smoothing to the estimation of the Laplace-Beltrami operator.
A few experts might be aware of these relations, but it still seems worthwhile to discuss them here.

The standard smoothing estimator of the Laplace-Beltrami operator is as follows.
First recall the random walk graph Laplacian. 
With ${\mathbb K}_{ij} = K_{\sqrt{2}\e} (\|X_i - X_j)\|),$ a random geometric graph is defined by the affinity matrix $ {\mathbb K} = ({\mathbb K}_{ij})_{i,j=1,\dots,n}.$
Then the corresponding random walk graph Laplacian is
\begin{equation*}
	\Delta_{n,\e} = \frac{1}{\e^2} \big(I_n-D_{\mathbb K}^{-1} {\mathbb K}\big)
\end{equation*}
with degree matrix $D_{\mathbb K} = {\rm diag}(D_1,\ldots,D_n)$ where $D_{i} = \sum_{j=1}^n {\mathbb K}_{ij}$, and $I_n$ the $n \times n$ identity matrix. 
The factor $\frac{1}{\e^2}$ is a convenient normalization (e.g. see \cite{coifman2006diffusion, hein2007graph}), and see (\ref{unif:bias-res2}) for a motivation.
The corresponding operator version (acting on functions $f$) is
\begin{equation}
	\Delta_{n,\M,2}f (x) = \frac{\frac{1}{n\e^{2}}\sum\limits _{j=1}^nK_{\sqrt{2}\e} (\|X_j - x\|_{\R^D}) 
	(f(x) - f(X_j))}{\frac{1}{n}\sum\limits _{j=1}^nK_{\sqrt{2}\e} (\|X_j - x\|_{\R^D})}.
\end{equation}

Clearly, $\Delta_{n,\M,2}f (x) = \frac{1}{\e^2}\big(f(x)-\TT_{n,{\sqrt{2}\e}} [f] (x)\big)$, and using Theorem~\ref{unif:bias}, it is straightforward to see that for $f \in C^{2,\alpha}(\M)$,
\begin{align*}
	\sqrt{n \e^{d-2}} \big(\TT_{n,{\sqrt{2}\e}} [f] (x) - \TT_{{\sqrt{2}\e}} [f] (x)\big) 
	=-\sqrt{n \e^{d+2}}\big(\Delta_{n,\M,2}f(x) - \Delta_{\M,2}f(x)\big) + O\big(\sqrt{n\e^{d+2+\alpha }}\big).
\end{align*}

Thus, if $n\e^{d+2+\alpha } \to 0$, then convergence in distribution of the normalized kernel smoothing is the same as that of the normalized smoothing estimator of the Laplace-Beltrami operator. 

\begin{corollary} \label{Berry:Laplace}
	In the setting of Theorem~\ref{Berry:nor}, for $f \in C^{2,\alpha}(\M)$ and
	\begin{align*}
		{\mathbf L}_{n,\e,f} &= \sqrt{n \e^{d+2}} \big(\Delta_{n,\M,2}f (x_j) - \Delta_{\M,2}f (x_j) \big)_{j=1}^m,
	\end{align*}
	we have
	\begin{equation}
		\distcvx({\mathbf L}_{n,\e,f}, \cN_{\overline\Sigma_f}) \leq \frac{C_1 \log(n\e^d)}{\sqrt{n\e^d}} + C_2 \e + C_0 \sqrt{n\e^{d+2+\alpha}}.
	\end{equation}
\end{corollary}

Strengthening the smoothness assumptions on $f$ to $C^{2,\alpha}$ gives smaller error rates.
Corollary~\ref{Berry:Laplace} generalizes a similar asymptotic normality result that can be found in \cite{gine2006empirical} from the case of a uniform sample to a sample from $\rho.$ 

\section{Derivatives of kernel smoothing on manifolds} \label{deriv}

Next, we investigate the uniform convergence of the derivatives of kernel smoothing on manifolds.
Such properties have applications in several areas, including topological data analysis \citep{chazal2017robust} and the mean shift algorithm \citep{zhang2021KernelSmoothing}.
We only consider gradients and Hessians, but expect that similar results can be obtained for higher-order derivatives.
For a compact presentation we use the notation $\nabla_\M^k, k = 1,2$, with $k=1$ corresponding to the gradient and $k=2$ to the Hessian.

\begin{theorem}	\label{deriv:unnor}
	For $k = 1,2$, we have:
	\begin{enumerate}[label=(\alph*)]
		\item
		For $f \in C^{2+k}(\M)$,
		there exist $\e_0, C_0 >0$ such that if $\e < \e_0$, we have 
		\begin{equation}
			\sup_{x \in \M} \big\| \nabla^k_{\M} \T_{\e} [f] (x) - \nabla^k_{\M} (\rho f) (x) \big\|
			\leq C_0 \e^{2},
		\end{equation}
		where $\e_0$ and $C_0$ are $\M$-dependent and they also depend on the $C^{2+k}$-norms of $f$ and $\rho$.
		\item
		For $f \in C^k (\M)$,
		there exist $\e_0, C_0 >0$ such that if $\e < \e_0$, we have with probability at least $1-\delta$, 
		\begin{equation}
			\sup_{x \in \M} \big\| \nabla^k_{\M} \T_{n,\e} [f] (x) - \nabla^k_{\M} \T_{\e} [f] (x) \big\|
			\leq C_0 \bigg( \sqrt{ \frac{\log(1/(\e \wedge \delta))}{n\e^{d+2k}} }
			+ \frac{\log(1/(\e \wedge \delta))}{n \e^{d+2k}} \bigg),
		\end{equation}
		where $\e_0$ and $C_0$ are $\M$-dependent and they also depend on the $C^k$-norms of $f$ and $\rho$.
	\end{enumerate}
\end{theorem}

\begin{theorem} \label{deriv:nor}
	For $k = 1,2$, we have:
	\begin{enumerate}[label=(\alph*)]
		\item 
		For $f \in C^{2+k}(\M)$,
		there exist $\e_0, C_0 >0$ such that if $\e < \e_0$, we have 
		\begin{equation}
			\sup_{x \in \M} \left\| \nabla^k_{\M} \TT_{\e} [f] (x) - \nabla^k_{\M} f(x) \right\|
			\leq C_0 \e^{2}
		\end{equation}
		where $\e_0$ and $C_0$ are $\M$-dependent, and also depend on $\rho_\tmin$ and the $C^{2+k}$-norms of $f$ and $\rho$.
		\item
		For $f \in C^{2+k}(\M)$,
		there exist $\e_0, C_0, C_0^* >0$ such that if $\e_{n,\delta}^{*} (C_0^*) < \e < \e_0^*$,
		we have with probability at least $1-\delta$, 
		\begin{equation}
			\sup_{x \in \M} \big\| \nabla^k_{\M} \TT_{n,\e} [f] (x) - \nabla^k_{\M} \TT_{\e} [f] (x) \big\|
			\leq C_0 \left( \sqrt{ \frac{\log(1/(\e\wedge \delta))}{n\e^{d +2(k-1)}} }  
			+ \frac{\log(1/(\e\wedge\delta))}{n \e^{d+2k}}\right)
		\end{equation}
		where $\e_0$, $C_0$ and $C^*_0$ are $\M$-dependent, and also depend on $\rho_\tmin$ and the $C^{2+k}$-norms of $f$ and $\rho$.
	\end{enumerate}
\end{theorem}

\section{Applications}
\label{smoothing:application}

\subsection{Kernel density estimation on manifolds}

Let the kernel density estimator with kernel $K_\e$ be 
\begin{align*}
	\rho_{n, \e} (x) &= \frac{1}{n} \sum_{i=1}^{n} K_\e (\|X_i - x\|_{\R^D}), \;\;x \in \M.
\end{align*}
We immediately see that $\rho_{n,\e}(x) = \T_{n,\e}[1](x),$ i.e., the KDE equals the unnormalized kernel smoothing of the function $f \equiv 1.$ Thus, $\T_\e[1](x) = \E \left[ \rho_{n, \e} (x) \right] = \int_\M K_\e ( \|y - x\|_{\R^D} ) d \P(y).$

So, our above results on uniform finite sample bounds for kernel smoothing and its derivatives, as well as the Berry-Esseen bounds immediately apply to the KDE.

\subsection{Kernel regression on manifolds} \label{reg}

In this part, we investigate Nadaraya-Watson type kernel regression when the explanatory variable lies on a manifold.
For simplicity, we consider the commonly used setting.

\begin{assumption} \label{reg:assume:density2}
	$\P_{X, Y}$ is a probability measure on $\M \times \R$ with marginal distribution $\P_X$ admitting a smooth density $\rho_X$ with respect to the volume measure on $\M$.
	The minimum of $\rho_X$, denoted by $\rho_{X,\tmin}$, is strictly positive.
	The pairs $(X_1,Y_1), \dots, (X_n, Y_n)$ are i.i.d. samples from $\P_{X,Y}$.
	$|Y_i| \leq C_Y$ almost surely for some $C_Y >0$.
\end{assumption}

The goal of the kernel regression is to estimate $\varphi(x) := \E(Y|X=x)$.
The Nadaraya-Watson estimator and its population version are:
\begin{align}
	\varphi_{n,\e} (x) &= \frac{ \sum_{i=1}^{n} Y_i K_\e \big( \|X_i -x\|_{\R^D} \big)}{\sum_{i=1}^{n} K_\e \big( \|X_i -x\|_{\R^D} \big)}, \\
	\varphi_{\e} (x) &= \frac{\int_{\M \times \R} y K_\e \big\|u -x\|_{\R^D} \big) d\P_{X,Y} (u,y) }{ \int_{\M \times \R} K_\e \big\|u -x\|_{\R^D}\big) d\P_{X,Y} (u,y) }.
\end{align}

Note that we do not impose any specific form of $\varphi (x)$.
For instance, one may have $\varphi (x) =  g(x)$ for some function $g$ when $Y= g(X) + \xi$ with the error term $\xi$ satisfying $\E[\xi|X]=0$.

Applying similar arguments as above, the following results can be obtained. 
As in Section~\ref{deriv}, we use $\nabla_\M^k$ with $k=0,1,2$, where $\nabla_\M^0f = f$, $\nabla_\M^1f$ is the gradient, and $\nabla_\M^2f$ the Hessian of $f$. For $k=0$, $\|\cdot\|$ means $|\cdot|$.

\begin{corollary} \label{reg:unif}
	Under Assumption \ref{reg:assume:density2},
	for each $k=0,1,2$, the following hold:
	\begin{enumerate}[label=(\alph*)]
		\item There exist $\e_0, C_0 >0$ such that if $\e < \e_0$ we have
		\begin{equation}
			\sup_{x \in \M} \left\| \nabla_\M^k \varphi_{\e} (x) - \nabla_\M^k \varphi(x) \right\| \leq C_0 \e^2.
		\end{equation}
		The constants $\e_0$ and $C_0$ are $\M$-dependent, and also depend on $\rho_{X,\tmin}$, and the $C^{2+k}$-norms of $\varphi$ and $\rho_X$.
		\item There exist $\e_0,C_0,C_0^*>0$ such that if $\e_{n,\delta}^{*} (C_0^*)<\e<\e_0$,
		then we have with probability at least $1-\delta$,
		\begin{equation}
			\sup_{x \in \M} \left\| \nabla_\M^k \varphi_{n,\e} (x) - \nabla_\M^k \varphi_{\e} (x) \right\|
			\leq C_0 \bigg( \sqrt{ \frac{\log(1/(\e\wedge\delta))}{n\e^{d+k}} }  
			+ \frac{\log(1/(\e\wedge\delta))}{n \e^{d+k}} \bigg).
		\end{equation}
		The constants $\e_0$ and $C_0$ are $\M$-dependent, and depend on $C_Y$, $\rho_{X,\tmin}$, and the $C^{2+k}$-norms of $\varphi$ and $\rho_X$.
	\end{enumerate}
\end{corollary}

\begin{corollary} \label{reg:Berry}	
	Let Assumption \ref{reg:assume:density2} hold and let 
	$x_1, \dots, x_m \in \M$ be distinct with $\var(Y|X=x_j) \ne 0, j=1,\dots,m$. Let
	\begin{align*}
		\bfZ_{n,\e,\varphi} &= \sqrt{n \e^d} \left(  \varphi_{n,\e} (x_j) - \varphi_{\e} (x_j) \right)_{j=1}^m, \\
		\tbfZ_{n,\e,\varphi} &= \sqrt{n \e^d} \left(  \varphi_{n,\e} (x_j) - \varphi (x_j) \right)_{j=1}^m, \\
		\Sigma_{\varphi} &=  \diag \left( \frac{\var(Y|X=x_j)}{ (4\pi)^{d/2} \rho_X (x_j)} \right)_{j=1}^m.
	\end{align*}
	Then, there exist $\e_0,C_0$ - $C_3 > 0$ such that if $\e < \e_0$ and $n\e^d > C_3$, we have
	\begin{align}
		\distcvx(\bfZ_{n,\e,\varphi}, \cN_{\Sigma_\varphi}) &\leq \frac{C_1 \log(n\e^d)}{\sqrt{n\e^d}} + C_2 \e, \\
		\distcvx(\tbfZ_{n,\e,\varphi}, \cN_{\Sigma_\varphi}) &\leq \frac{C_1 \log(n\e^d)}{\sqrt{n\e^d}} + C_2 \e +C_0 \sqrt{n\e^{d+4}}.
	\end{align}
	The constants $\e_0,C_0$ - $C_3$ are $\M$-dependent, and they also depend on $\rho_{X, \tmin}$, the $C^2$-norms of $\E[Y^2\,|\,X=\cdot]$, $\varphi, \rho_X, m$, and the values $\var(Y|X=x_j)$, $\rho_X(x_j), j = 1,\ldots,m$.
\end{corollary}

Note that even though $\varphi_{n,\e}$ has a form of normalized kernel smoothing, the convergence rate is not improved relative to usual normalized cases.
This is due to the presence of a possible error term $\xi$ in $Y$, 
which is not offset by the normalization through the denominator $\rho_X$.

\subsection{Estimation of heat kernel signature on manifolds} \label{HKS}

The heat kernel signature (HKS) is a multi-scale function defined on a Riemannian manifold encoding certain geometric information about the manifold \citep{sun2009concise}.
It has been used for shape classification, where the goal is to classify figures which may have different poses and some noise.
Theoretically, the HKS is obtained by considering the diagonal components of the corresponding heat kernel.
Alternatively, the HKS can be expressed by the eigenvalues and the eigenfunctions of the Laplace-Beltrami operator of the manifold.
Formally, for diffusion time $\tau>0$, the HKS is given by
\begin{equation*}
	H_\tau (x) = \sum_{j=0}^\infty e^{-\lambda_j \tau} \phi^2_j (x),
\end{equation*}
where $\lambda_j \ge 0$ denotes the $j$th eigenvalue, sorted in non-decreasing order, of the Laplace-Beltrami operator $\Delta_\M$, and $\phi_j$ is the corresponding $L^2$-orthonormal eigenfunction.

To estimate the HKS, \cite{dunson2021spectral} suggest a discrete version of an estimator of the heat kernel evaluated at the data points sampled from the manifold through the graph Laplacian, and showed that the estimator converges to the corresponding heat kernel in probability uniformly over the data points.
Therefore, a discrete version of the estimator of the HKS can be obtained by restricting the estimator of the heat kernel.
Here, we propose to smoothly extend this discrete version to the entire domain using kernel smoothing, and then to investigate the asymptotic behavior of this extended version.

To define the estimator, consider the $n \times n$ affinity matrix $W$ with a reweighted kernel \citep{coifman2006diffusion} 
\begin{equation*}
	W = (W_{ij})_{i,j=1,\dots,n} = \bigg( \frac{K_{\sqrt{2}\eta} \|X_i -X_j\|_{\R^D}}{q_\eta (X_i) q_\eta (X_j)} \bigg)_{i,j=1,\dots,n},
\end{equation*}
where $q_\eta (x) = \sum_{i=1}^{n} K_{\sqrt{2}\eta} (\|X_i-x\|_{\R^D})$. The kernel-normalized graph Laplacian matrix $\Delta_{n,\eta}$ then is 
\begin{equation*}
	\Delta_{n,\eta} = \frac{1}{\eta^2}\big(I_n-D_W^{-1} W\big)
\end{equation*}
with degree matrix $D_W = {\rm diag}(D_1,\ldots,D_n),$ where $D_{i} = \sum_{j=1}^n W_{ij}$, and $I_n$ denotes the $n \times n$ identity matrix. 
Note that due to the re-weighting this is different from Section~\ref{Berry:LB-operator}.
Also, we use a different bandwidth $\eta$; we explain this at the end of subsection.

Let $\mu_{i,n,\eta}$ be the $i$-th eigenvalue of $\Delta_{n,\eta}$ with ordering $0= \mu_{0,n,\eta} \leq \mu_{1,n,\eta} \leq \dots \leq \mu_{n,n,\eta}$
and the associated eigenvector $\tilde{\nu}_{i,n,\eta}$ normalized in the Euclidean norm. 
To match the eigenvectors with the eigenfunctions of the Laplace-Beltrami operator,
$\tilde{\nu}_{i,n,\eta}$ needs to be normalized as
$ \nu_{i,n,\eta} = \tilde{\nu}_{i,n,\eta} / ||\tilde{\nu}_{i,n,\eta}||_{w}, $
where
${||\tilde{\nu}_{i,n,\eta}||^2_{w} } =  \sum_{i=1}^n w_i \tilde{\nu}_{i,n,\eta}^2 (i), \quad w_i = \frac{|S^{d}| \eta^d}{d} \frac{1}{{\mathbb{N}(i)}}. $
Here, $|S^{d}|$ denotes the volume of the unit sphere in $d$-dimensional Euclidean space, and $\mathbb{N}(i) = |B_\eta^{\R^D}(x_i) \cap \set{x_1, \dots, x_n}|$.
Now, using $N$ eigenpairs of $\Delta_{n,\eta}$, given $\tau >0$, the estimator of the HKS evaluated at $X_i$, denoted by $\tilde{H}_{n,N,\eta,\tau} (X_i)$, is defined by the $i$th diagonal component of the matrix $\sum_{j=0}^{N-1} e^{-\tau \mu_{j,n,\eta}} \nu_{j,n,\eta} \nu_{j,n,\eta} ^\top$.

To state the uniform convergence of $\tilde{H}_{n,N,\eta,\tau}$, several assumptions are needed:

\begin{assumption} \label{HKS:assume}
	The parameters $n$, $\eta$, $N$ and $\tau$ satisfy the following conditions:
	\begin{enumerate}[label=(\alph*)]
		\item Assume all eigenvalues $\set{\lambda_j: j=0,1,\dots}$ of $\Delta_\M$ are simple.
		For some constants $\Omega_1$, $\Omega_2$ and $\Omega_3 >1$:
		\begin{equation*}
			\eta \leq \Omega_1 \min \Bigg( \Bigg( \frac{\Gamma_N \wedge 1}{\Omega_2 + \lambda_N^{\frac{d}{2}+5}} \Bigg)^2 , 
			\frac{1}{\Big(\Omega_3 + \lambda_N^{(5d+7)/4}\Big)^{2}},\;
			\frac{1}{N^4 \lambda_N^{\frac{d-1}{2}}} \Bigg),
		\end{equation*}
		where $\Gamma_N := \min_{1\leq j \leq N} \inf_{k \not= j} |\lambda_j - \lambda_k|$. 
		\item $n$ is large enough so that
		$ \big( \frac{\log n}{n}\big) ^{\frac{1}{4d+13}} \leq \eta.$
		\item $\tau$ is not too small so that 
		$\frac{C \log N }{N^{2/d}} \leq \tau,$
		where $C$ is $\M$-dependent.
	\end{enumerate}
\end{assumption}

Now, we state the uniform convergence of $\tilde{H}_{n,N,\eta,\tau}$ toward the true HKS $H_\tau$.
The following proposition is the modified version of Theorem 3 from \cite{dunson2021spectral} for the HKS on the data points. 
\begin{proposition} \label{HKS:Dunson}
	There exist $\M$-dependent constants $\Omega_1$, $\Omega_2$ and $\Omega_3 >1$ also depending on $\rho_\tmin$ and the $C^2$-norm of $\rho$, such that if Assumption~\ref{HKS:assume} holds, we have with probability at least $1-n^{-2}$,
	\begin{equation*}
		\sup_{i \in [n]} \left| \tilde{H}_{n,N,\eta,\tau} (X_i) - H_\tau (X_i) \right| \le \frac{\Omega_4}{N} (\tau+1) e^{-\Omega_5 \tau} + \Omega_6 \eta.
	\end{equation*}
	The constants $\Omega_4$-$\Omega_6$ are $\M$-dependent.
	Additionally, $\Omega_4$ and $\Omega_6$ also depend on $\rho_\tmin$ and the $C^2$-norm of $\rho$.
\end{proposition}

We introduce an extension of $\tilde{H}_{n,N,\eta,\tau}$ through kernel smoothing:
\begin{definition}
	For $\e >0$ and $x \in \M$, the smooth extension $\widehat{H}_{n,N,\eta,\tau,\e}$ of $\tilde{H}_{n,N,\eta,\tau}$ is defined by normalized kernel smoothing:
	\begin{align*}
		\widehat{H}_{n,N,\eta,\tau,\e} (x)
		&= \frac{\sum_{i=1}^{n} K_\e \big( \|X_i - x \|_{\R^D}\big) \tilde{H}_{n,N,\eta,\tau} (X_i) }{\sum_{i=1}^{n} K_\e \big( \|X_i - x \|_{\R^D}\big)}.
	\end{align*}
\end{definition}
Applying the uniform convergence and asymptotic normality of kernel smoothing to $\widehat{H}_{n,N,\eta,\tau,\e}$, we obtain the following results.

\begin{corollary} \label{HKS:unif}
	Let Assumption~\ref{HKS:assume} hold. Then, for $k=0,1,2$, there exist $\e_0, C_0,C_0^*>0$ such that if $\e_{n,\delta}^{*} (C_0^*) < \e < \e_0$,
	then, with probability at least $1-\delta-\frac{1}{n^{2}}$,
	\begin{multline}
		\sup_{x \in \M} \left\| \nabla_\M^k \widehat{H}_{n,N,\eta,\tau,\e} (x) - \nabla_\M^k H_\tau (x) \right\|
		\leq C_0 \bigg( \frac{\Omega_4}{N \e^{2k}} (\tau+1) e^{-\Omega_5 \tau} + \Omega_6 \frac{\eta}{\e^{2k}} +\e^{2} 
		\\+ \sqrt{ \frac{\log(1/(\e\wedge\delta))}{n\e^{d+2(k-1)}} } + \frac{\log(1/(\e\wedge\delta))}{n \e^{d+2k}} \bigg).
	\end{multline}
	The constants $\e_0$, $C_0$ and $C^*_0$ are $\M$-dependent, and they also depend on $\rho_\tmin$ and the $C^{2+k}$-norms of $H_\tau$ and $\rho$.
\end{corollary}

\begin{corollary} \label{HKS:Berry}
	For distinct $x_1, \dots, x_m \in \M$ with $\|\nabla_\M H_\tau(x_j)\| \ne 0, j=1,\dots,m$, let
	\begin{align*}
		\tbfZ_{n,N,\eta,\tau,\e} &= \sqrt{n \e^{d-2}} \left( \widehat{H}_{n,N,\eta,\tau,\e} (x_j) - H_\tau (x_j) \right)_{j=1}^m, \\
		\Sigma_{\tau} &= \, \diag \left( \frac{\|\nabla_\M H_\tau(x_j)\|^2}{2(4\pi)^{d/2}\rho(x_j)} \right)_{j=1}^m.
	\end{align*}
	Then, if Assumption~\ref{HKS:assume} holds, there exist $\e_0,C_0$ - $C_3 > 0$ such that if $\e < \e_0$ and $n\e^d > C_3$, we have
	\begin{align}
		\distcvx(\tbfZ_{n,N,\eta,\tau,\e}, \cN_{\Sigma_\tau})
		\frac{C_1 \log(n\e^d)}{\sqrt{n\e^d}} + C_2 \e + C_0 \sqrt{n\e^{d+2}} 
		+ C_0 \sqrt{n\e^{d-2}} \left[ \frac{\Omega_4}{N} (\tau+1) e^{-\Omega_5 \tau} + \Omega_6 \eta \right] + \frac{1}{n^2},
	\end{align}
	where $\Omega_4$ - $\Omega_6$ are from Proposition \ref{HKS:Dunson}.
	The constants $\e_0,C_0$ - $C_3$ are $\M$-dependent, and they also depend on $\rho_\tmin$, the $C^2$-norms of $H_\tau$ and $\rho$, $m$, and the values $\|\nabla_\M H_\tau(x_j)\|$, $\rho(x_j)$, $j = 1,\ldots,m$.
\end{corollary}

Note that for large $\tau$, the optimal convergence rate in Proposition~\ref{HKS:Dunson} is $O(\eta)=O(n^{-\frac{1}{4d+13}})$.
In this case, to ensure convergence in Corollary~\ref{HKS:unif}, we require $\e=n^{-\gamma}$ with $0<\gamma<\frac{1}{4(4d+13)}$, whereas for Corollary~\ref{HKS:Berry}, we require $\frac{1}{d+2}\vee\frac{4d+11}{2(d-2)(4d+13)}<\gamma<\frac{1}{d}$ and $d\ge4$.
We also see that if $\eta=\e$, we cannot guarantee any convergence.

\section{Discussion}

Above, kernel smoothing has been formulated assuming the dimension $d$ of the manifold is known.
If $d$ is unknown, it must be estimated; let $\hat{d}_n$ denote such an estimator.
In this case, as long as $\hat{d}_n$ is consistent, our results continue to hold.
For instance, let $\widehat{\T}_{n,\e} [f]$ denote the unnormalized kernel smoothing obtained by replacing $d$ in the kernel with $\hat{d}_n$. For any measurable set $A \subset \R$, it holds that
$$ \left|\P(\widehat{\T}_{n,\e} [f] (x)\in A)- \P(\T_{n,\e} [f] (x)\in A)\right| \leq 2 \P(\hat{d}_n \ne d). $$
Thus, the error caused by using $\hat{d}_n$ is negligible whenever $\hat{d}_n$ is consistent at an appropriate rate.
There is an extensive literature on constructing such estimators with explicit rate; see, e.g., \cite{block2022intrinsic} or \cite{farahmand2007manifoldadaptive}.

Also, the asymptotic properties of kernel smoothing we have shown depend on the compactness and boundedness of the manifold.
For a non-compact manifold, to control the convergence rate uniformly, additional assumptions appear to be needed, such as uniform bound of curvature of the manifold.
For manifolds with boundary, boundary corrections can be applied, assuming the manifold to be known.
This is considered in \cite{berry2017density}, where pointwise convergence of boundary-corrected KDE on a manifold is derived.
We expect that similar corrections can be applied to the more general setting of kernel smoothing; we leave it as future work.

\bibliography{smoothing_bib}

\appendix

\section{Simulation study}
In this section, we show simulation results to validate our results.
Primarily, we will mainly focus on the asymptotic distribution of kernel smoothing.

First, we consider a circle in $\R^2$ centered at the origin with the radius 5, which is a 1-dimensional compact manifold without boundary.
$n$ i.i.d data points are sampled from this circle with uniform distribution.
The bandwidth is chosen as $\e = n^{-1/(d+1)}$ so that $n\e^d$ is large but $n \e^{d+2}$ is small for convergence.
Now, to illustrate Theorems~\ref{Berry:unnor} and \ref{Berry:nor},
we set $f(x_1, x_2) = e^{\sin(x_1)} + x_2$, and evaluate
\begin{equation}
	\sqrt{n \e^d} \left( \T_{n,\e} [f] (x) - \rho(x) f(x) \right), \label{formula: simul1}
\end{equation}
and its normalized version
\begin{equation}
	\sqrt{n \e^{d-2}} \left( \TT_{n,\e} [f] (x) - f(x) \right) \label{formula: simul2}
\end{equation}
at $x = (5,0)$ and $x=(0,5)$.
We perform this 300 times to compare with the corresponding asymptotic distribution.
The results are visualized in Figure \ref{fig: circle-unnormal} and \ref{fig: circle-normal}.
These show that the asymptotic behavior holds as $n$ becomes large.

Next, to investigate a more complex case, we consider the surface of a 2-dimensional torus embedded in $\R^3$ with the major radius $R=1/2$ and the minor radius $r=1/3$.
$n$ i.i.d. data 

points are sampled from the bivariate von Mises sine model in the following way: The torus is parameterized by
$ (    (R+r \cos\theta_2) \cos \theta_1, (R+r \cos\theta_2) \sin \theta_1, r \sin\theta_2 ) $
and the density is imposed to $(\theta_1, \theta_2) \in S^1 \times S^1$ by
\begin{equation*}
	\rho(\theta_1,\theta_2 | \mu_1, \mu_2, \kappa_1, \kappa_2) \sim  e^{\kappa_1 \cos \left(\theta_1 -\mu_1\right)+\kappa_2 \cos \left(\theta_2 -\mu2\right)
		+\kappa_3 \sin \left(\theta_1 - \mu_1\right) \sin \left(\theta_2 - \mu_2\right)}. 
\end{equation*}
We set $\mu_1 = \mu_2 = 0$, $\kappa_1 = \kappa_2 = 1$, $\e=n^{-1/(d+1)}$ and $f(x_1,x_2,x_3) =  \sin(x_1-x_2) + e^{-\cos(x_1+x_2)} + x_3^2 $, and
evaluate \eqref{formula: simul1} and \eqref{formula: simul2} at $x=(1/6,0,0)$ and $x=(0,-1/2,1/3)$.
Again, we repeat 300 times and obtain empirical distributions.
Like the circle case, we can observe the asymptotic behavior from Figures \ref{fig: torus-unnormal} and \ref{fig: torus-normal}.

\begin{figure}[htbp]
	\centering
	\begin{subfigure}{.33\linewidth}
		\centering
		\includegraphics[height = 3cm,width=.8\linewidth]{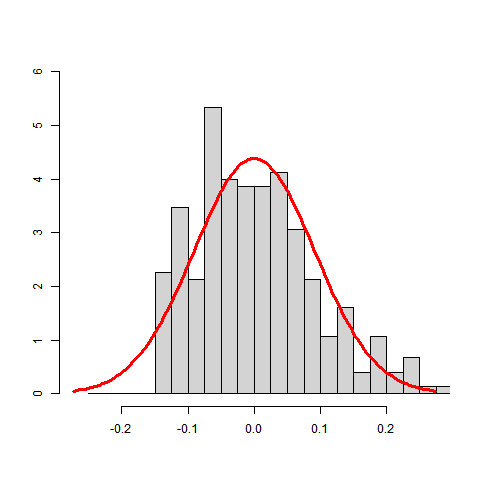}
		\vspace{-1em}
		\caption*{$n=500$}
	\end{subfigure}%
	\begin{subfigure}{.33\linewidth}
		\centering
		\includegraphics[height = 3cm,width=.8\linewidth]{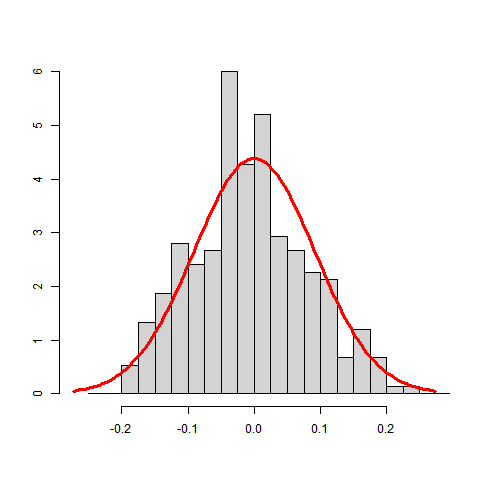}
		\vspace{-1em}
		\caption*{$n=2000$}
	\end{subfigure}%
	\begin{subfigure}{.33\linewidth}
		\centering
		\includegraphics[height = 3cm,width=.8\linewidth]{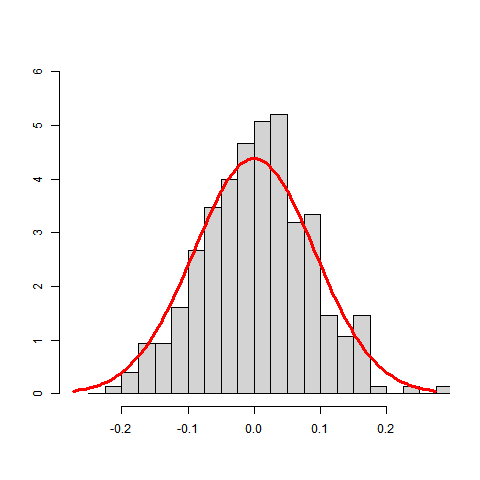}
		\vspace{-1em}
		\caption*{$n=10000$}
	\end{subfigure}%
	\vspace{-0.25em}
	\newline
	\begin{subfigure}{.33\linewidth}
		\centering
		\includegraphics[height = 3cm,width=.8\linewidth]{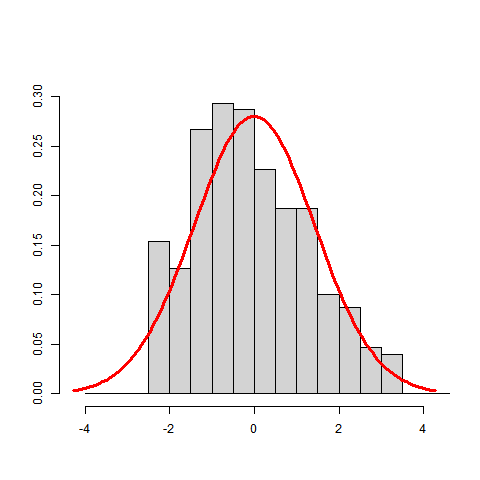}
		\vspace{-1em}
		\caption*{$n=500$}
	\end{subfigure}%
	\begin{subfigure}{.33\linewidth}
		\centering
		\includegraphics[height = 3cm,width=.8\linewidth]{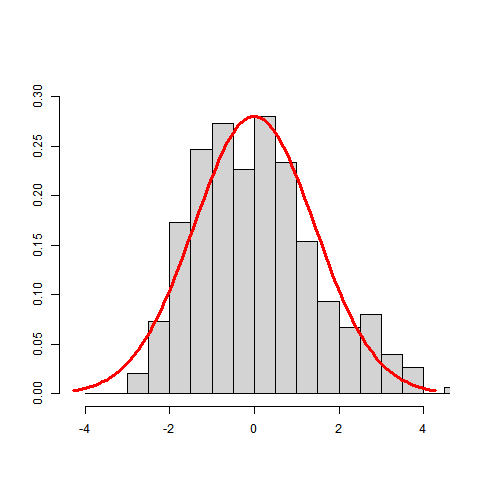}
		\vspace{-1em}
		\caption*{$n=2000$}
	\end{subfigure}%
	\begin{subfigure}{.33\linewidth}
		\centering
		\includegraphics[height = 3cm,width=.8\linewidth]{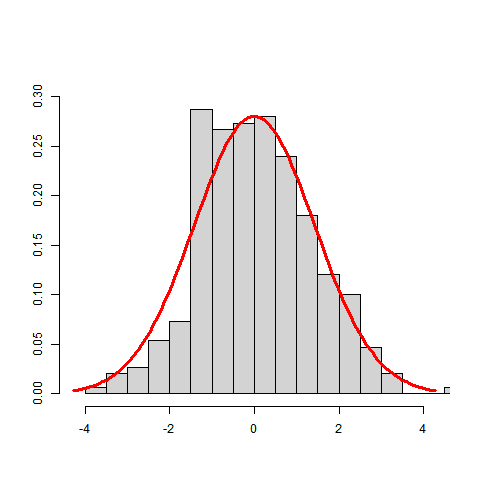}
		\vspace{-1em}
		\caption*{$n=10000$}
	\end{subfigure}%
	\caption{
		Empirical distribution of \eqref{formula: simul1} over 300 simulated datasets from the circle case and the corresponding theoretic asymptotic density (red line)
		at two points (5,0) (top) and (0,5) (bottom).
	}
	\label{fig: circle-unnormal}
\end{figure}

\begin{figure}[htbp]
	\centering
	\begin{subfigure}{.33\linewidth}
		\centering
		\includegraphics[height = 3cm,width=.8\linewidth]{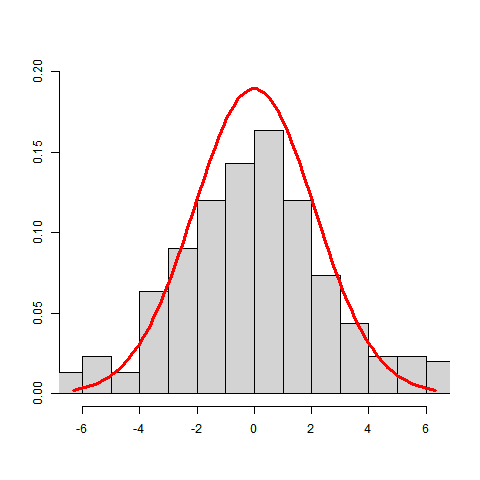}
		\vspace{-1em}
		\caption*{$n=500$}
	\end{subfigure}%
	\begin{subfigure}{.33\linewidth}
		\centering
		\includegraphics[height = 3cm,width=.8\linewidth]{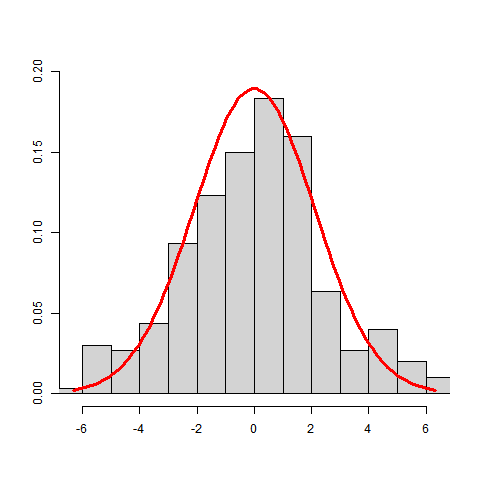}
		\vspace{-1em}
		\caption*{$n=2000$}
	\end{subfigure}%
	\begin{subfigure}{.33\linewidth}
		\centering
		\includegraphics[height = 3cm,width=.8\linewidth]{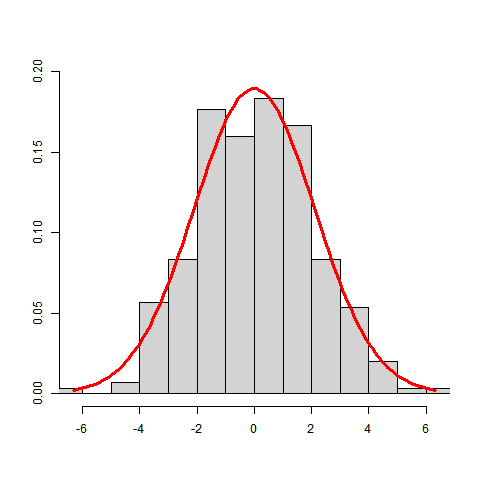}
		\vspace{-1em}
		\caption*{$n=10000$}
	\end{subfigure}%
	\vspace{-0.25em}
	\newline
	\begin{subfigure}{.33\linewidth}
		\centering
		\includegraphics[height = 3cm,width=.8\linewidth]{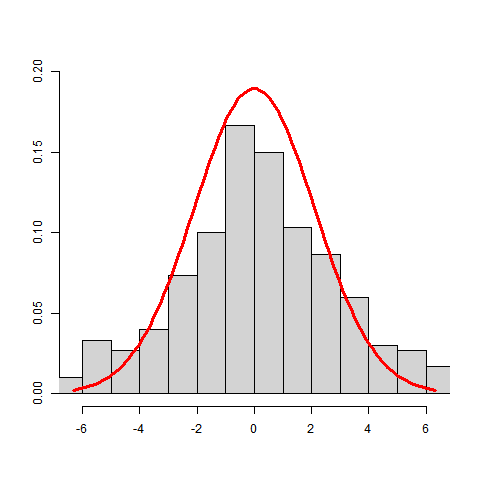}
		\vspace{-1em}
		\caption*{$n=500$}
	\end{subfigure}%
	\begin{subfigure}{.33\linewidth}
		\centering
		\includegraphics[height = 3cm,width=.8\linewidth]{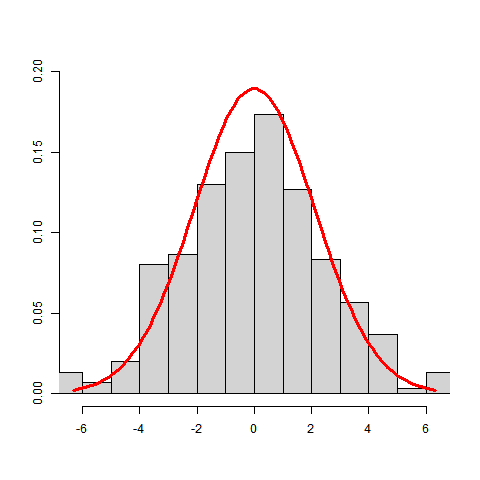}
		\vspace{-1em}
		\caption*{$n=2000$}
	\end{subfigure}%
	\begin{subfigure}{.33\linewidth}
		\centering
		\includegraphics[height = 3cm,width=.8\linewidth]{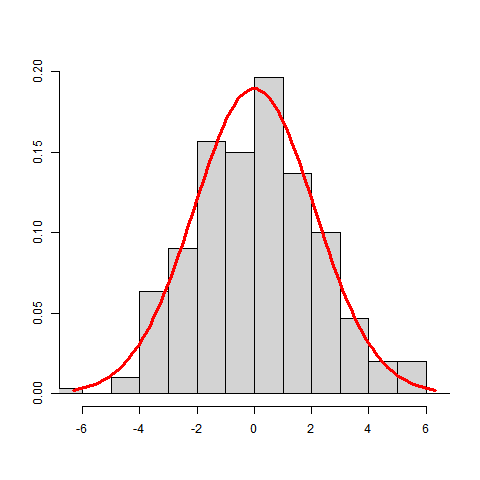}
		\vspace{-1em}
		\caption*{$n=10000$}
	\end{subfigure}%
	\caption{
		Empirical distribution of (\ref{formula: simul2}) over 300 simulated datasets from the circle case and the corresponding theoretic asymptotic density (red line)
		at two points (5,0) (top) and (0,5) (bottom).
	}
	\label{fig: circle-normal}
\end{figure}

\begin{figure}[htbp]
	\centering
	\begin{subfigure}{.33\linewidth}
		\centering
		\includegraphics[height = 3cm,width=.8\linewidth]{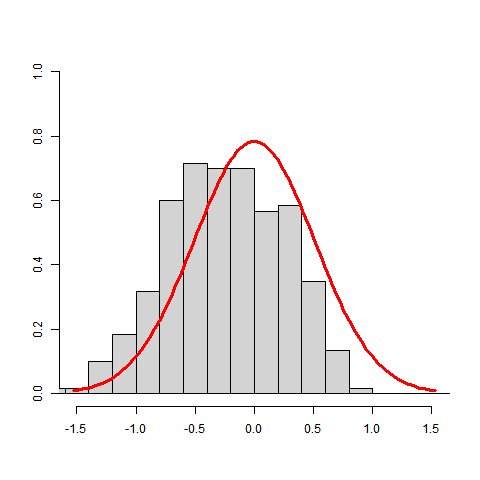}
		\vspace{-1em}
		\caption*{$n=500$}
	\end{subfigure}%
	\begin{subfigure}{.33\linewidth}
		\centering
		\includegraphics[height = 3cm,width=.8\linewidth]{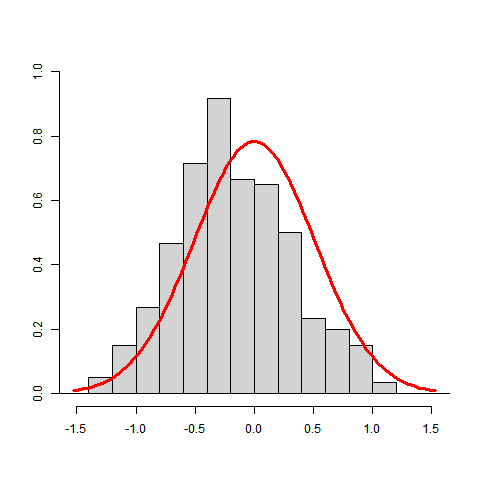}
		\vspace{-1em}
		\caption*{$n=2000$}
	\end{subfigure}%
	\begin{subfigure}{.33\linewidth}
		\centering
		\includegraphics[height = 3cm,width=.8\linewidth]{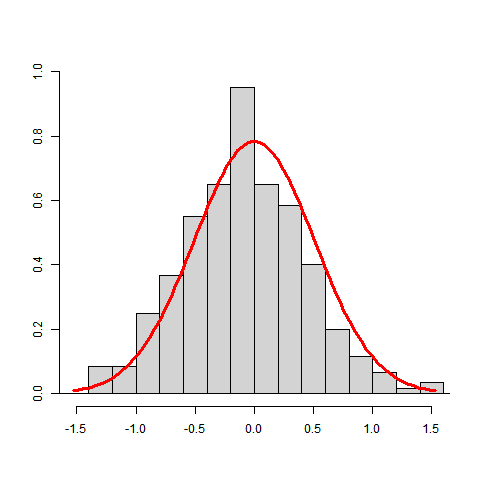}
		\vspace{-1em}
		\caption*{$n=10000$}
	\end{subfigure}%
	\vspace{-0.25em}
	\newline
	\begin{subfigure}{.33\linewidth}
		\centering
		\includegraphics[height = 3cm,width=.8\linewidth]{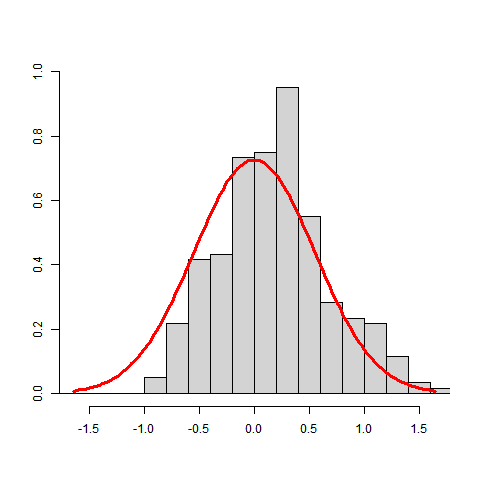}
		\vspace{-1em}
		\caption*{$n=500$}
	\end{subfigure}%
	\begin{subfigure}{.33\linewidth}
		\centering
		\includegraphics[height = 3cm,width=.8\linewidth]{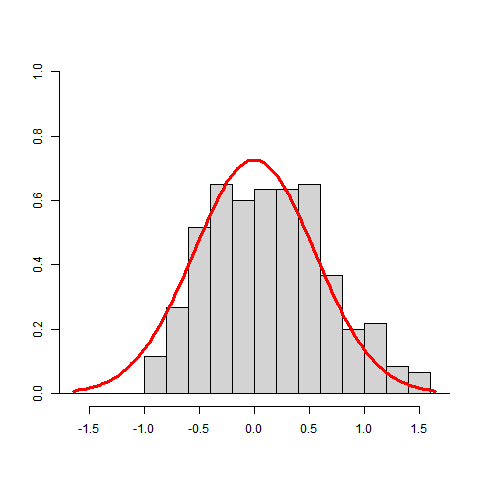}
		\vspace{-1em}
		\caption*{$n=2000$}
	\end{subfigure}%
	\begin{subfigure}{.33\linewidth}
		\centering
		\includegraphics[height = 3cm,width=.8\linewidth]{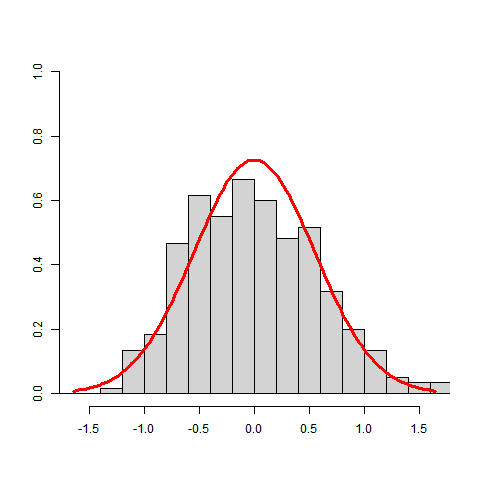}
		\vspace{-1em}
		\caption*{$n=10000$}
	\end{subfigure}%
	\caption{
		Empirical distribution of \eqref{formula: simul1} over 300 simulated datasets from the torus case and the corresponding theoretic asymptotic density (red line)
		at two points (1/6,0,0) (top) and (0,-1/2,1/3) (bottom).
	}
	\label{fig: torus-unnormal}
\end{figure}

\begin{figure}[htbp]
	\centering
	\begin{subfigure}{.33\linewidth}
		\centering
		\includegraphics[height = 3cm,width=.8\linewidth]{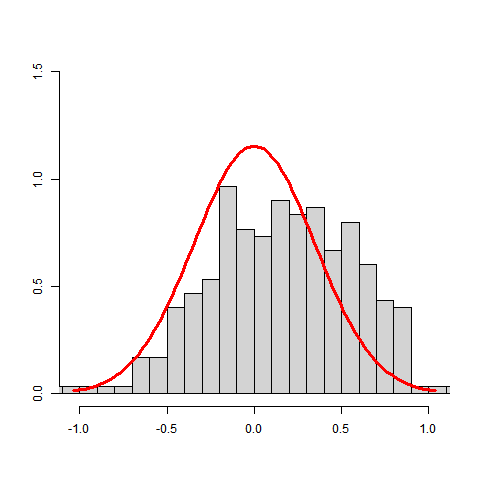}
		\vspace{-1em}
		\caption*{$n=500$}
	\end{subfigure}%
	\begin{subfigure}{.33\linewidth}
		\centering
		\includegraphics[height = 3cm,width=.8\linewidth]{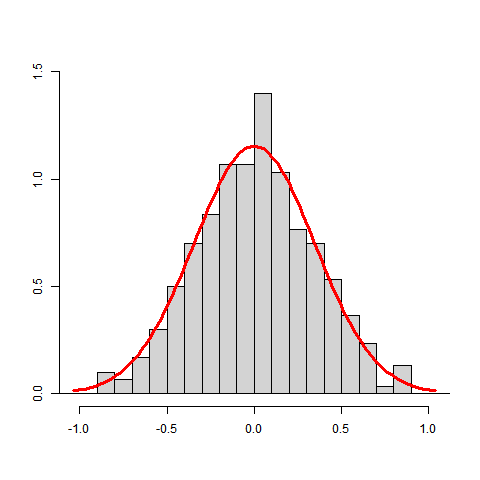}
		\vspace{-1em}
		\caption*{$n=2000$}
	\end{subfigure}%
	\begin{subfigure}{.33\linewidth}
		\centering
		\includegraphics[height = 3cm,width=.8\linewidth]{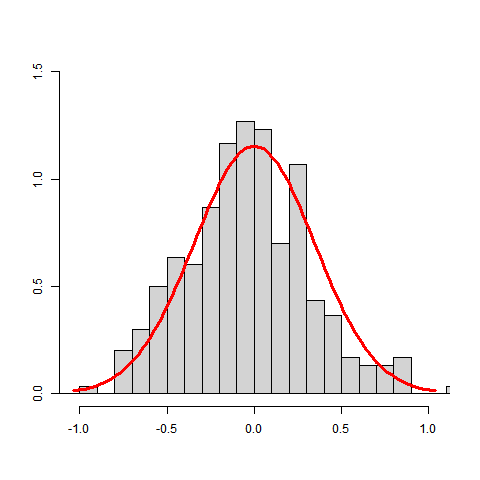}
		\vspace{-1em}
		\caption*{$n=10000$}
	\end{subfigure}%
	\vspace{-0.25em}
	\newline
	\begin{subfigure}{.33\linewidth}
		\centering
		\includegraphics[height = 3cm,width=.8\linewidth]{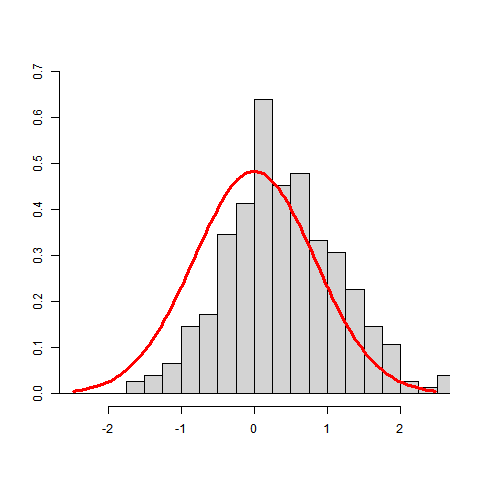}
		\vspace{-1em}
		\caption*{$n=500$}
	\end{subfigure}%
	\begin{subfigure}{.33\linewidth}
		\centering
		\includegraphics[height = 3cm,width=.8\linewidth]{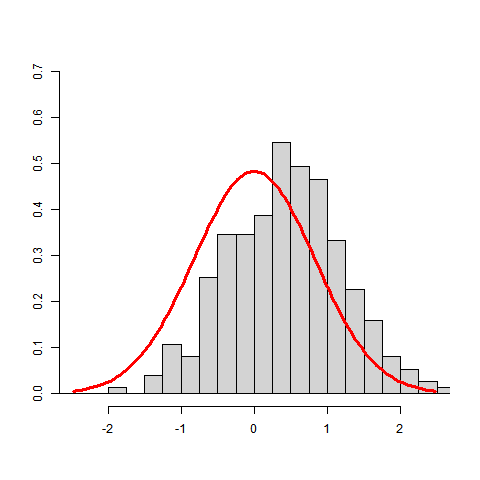}
		\vspace{-1em}
		\caption*{$n=2000$}
	\end{subfigure}%
	\begin{subfigure}{.33\linewidth}
		\centering
		\includegraphics[height = 3cm,width=.8\linewidth]{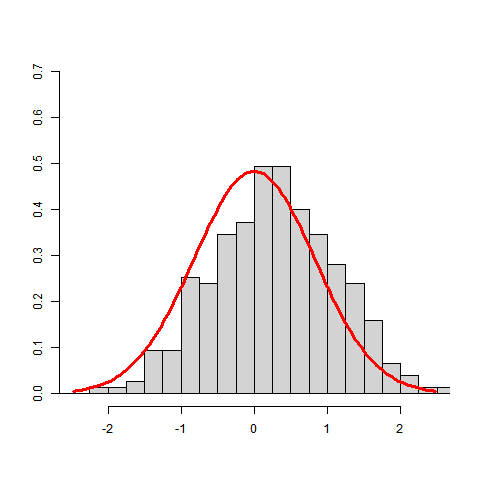}
		\vspace{-1em}
		\caption*{$n=10000$}
	\end{subfigure}%
	\caption{
		Empirical distribution of (\ref{formula: simul2}) over 300 simulated datasets from the torus case and the corresponding theoretic asymptotic density (red line)
		at two points (1/6,0,0) (top) and (0,-1/2,1/3) (bottom).
	}
	\label{fig: torus-normal}
\end{figure}

\newpage
\section{Proofs}
\subsection{Auxiliary results}

We present key technical results along with some notation that will be used throughout the proofs.

For each $x \in \M$, $\Exp^{-1}_x$ determines a local coordinate system near $x$ called normal coordinates.
With the identification of $T_x \M$ with $\R^d$, for $U \subset \R^d$ a neighborhood of 0 with suitable size,
one can express the integral of a function $f$ on $\Exp_x (U)$ with respect to the volume measure $\mathrm{vol}_\M$ as
\begin{equation}
	\int_{\Exp_x (U)} f(u) \dvol(u) = \int_{U} \tilde{f}_x (z) G_x(z) dz,
\end{equation}
where $\tilde{f}_x = f \circ \Exp_x$ is the pullback of $f$ by $\Exp_x$, and $G_x$ is the volume form at $x$ with respect to normal coordinates.
We use this approach to express integrals on $\M$ explicitly, and so we need the expansions of $G_x$ and $\Exp_x$ in terms of normal coordinates.

The result is summarized in the following proposition, and the proof can be found in Proposition 2.2 of \cite{gine2006empirical} and Section 3 of \cite{monera2014taylor}.
Throughout the following, the notation $R_k(z)$ denotes a homogeneous term of degree $k$ in $z \in \R^d$.
Depending on the context, $R_k(z)$ may represent a scalar, a vector, or a matrix whose components are homogeneous polynomials of degree $k$.

\begin{proposition} \label{smoothing:aux:expansion}
	Suppose Assumption \ref{smoothing:assume:manifold} holds. Then, there exists a positive constant $r$ depending on the injectivity radius of $\M$ such that for any $x \in \M$ and $z \in \R^d$ with $\|z\|_{\R^d} < r$, the following hold:
	\begin{enumerate}[label=(\alph*)]
		\item
		$G_x(z)$ admits the Taylor expansion
		\begin{equation}
			G_x (z) = 1 - \frac{1}{6} z^\top \Ric_x z + R_3 (z) + R_4(z) + \dots,
		\end{equation}
		where $\Ric_{x} \in \R^{d \times d}$ denotes the matrix representation of the Ricci curvature tensor of $\M$ at $x$ in normal coordinates.
		
		\item
		$\Exp_{x} (z)$, when viewed as a map from $\R^d$ to $\R^D$, admits the Taylor expansion
		\begin{equation}
			\Exp_{x} (z) = x + J(x) z + \frac{1}{2} \sff_{x} (z, z) + \frac{1}{6} \alpha_{x,3} (z) + R_4(z) + \dots,
		\end{equation}
		where $J(x) \in \R^{D \times d}$ denotes the Jacobian matrix of $\Exp_x$ at the origin (in particular, the columns of $J(x)$, $\set{J_1(x), \dots, J_d(x)}$, are orthonormal in $\R^D$),
		$\sff_{x}: \R^d \times \R^d \rightarrow (T_x \M)^\perp \subset \R^D$ denotes the second fundamental form of $\M$ at $x$,
		and $\alpha_{x,3} (z) \in \R^D$ is a homogeneous cubic term in $z$ satisfying
		\begin{equation*}
			J (x)^\top \alpha_{x,3} (z) = -\left( \sff_x(e_1,z)^\top \sff_x(z,z), \ldots, \sff_x(e_d,z)^\top
			\sff_x(z,z) \right)^\top  \in \R^{d}
		\end{equation*}
		with $\set{e_1, \dots, e_d}$ denoting the standard basis of $\R^d$.
		
		\item 
		The norm of the difference $\Exp_x(z)-x$ satisfies
		\begin{equation}
			\|\Exp_{x} (z) - x \|_{\R^D}^2 - \|z\|_{\R^d}^2  
			= -\frac{1}{12} \sff_x(z,z)^\top \sff_x(z,z) + R_5 (z) + \dots,
		\end{equation}
		and
		\begin{equation}
			\frac{1}{2} \|z\|_{\R^d} \leq \|\Exp_{x} (z) - x \|_{\R^D} \leq \|z\|_{\R^d}.
		\end{equation}
	\end{enumerate}
	Here, all implied coefficients in the $R_k$ terms are uniformly bounded by an $\M$-dependent constant for all $x \in \M$ and $z\in\R^d$ with $\|z\|_{\R^d} < r$.
\end{proposition}

Recall the notation $K_\e \left( u\right) = (2 \pi\e^2)^{-d/2} \exp \big(-\frac{1}{2\e^2} u^2\big),$ and note that $\int_{\R^d} K_\e(\|u\|) du = 1.$
For $\e = 1,$ we simply write $K(u)$.
Throughout the proofs, we will primarily work with $K$.

\begin{remark}
	Several arguments in our proofs reduce to evaluating convolution integrals over $\R^d$ by using Taylor expansions, that is,
	\begin{align*}
		&\quad \int_\M K_\e \left( \|u - x\|_{\R^D} \right) f(u) d \P(u)
		\\&= \int_{\R^d} K(\|z\|_{\R^d}) \Big( 1 + \frac{\e^2}{24} \sff_x (z,z)^\top \sff_x (z,z) + \e^3 R_5(z) + \e^4 R_6 (z) + \dots \Big)
		\\& \quad \times \Big( (\widetilde{\rho f})_x (0) + \e z^\top \nabla_{\R^d} (\widetilde{\rho f})_x (0)
		+ \frac{\e^2}{2} z^\top \left[ \nabla_{\R^d}^2 (\widetilde{\rho f})_x (0) \right] z 
		+ \e^3 R_3(z) + \ldots \Big)
		\\& \quad \times \Big( 1 - \frac{\e^2}{6} z^\top \Ric_x z + \e^3 R_3 (z) + \ldots \Big) dz + O(\e^\kappa),
	\end{align*}
	for arbitrary $\kappa > 0$.
	The right-hand side can then be simplified by either of the following:
	\begin{itemize}
		\item removing all odd-order terms in $z$ due to the symmetry of the kernel;
		\item evaluating even-order terms using the moments of the kernel.
	\end{itemize}
	
	We will present the details once in the following proof, and, to avoid repetitions, omit the details in other proofs.
\end{remark}

The following lemma serves to extend integration involving $K$ from an open ball to the whole space.

\begin{lemma} \label{smoothing:aux:tail}
	Suppose Assumption \ref{smoothing:assume:kernel} holds.
	For any non-negative integer $m$, there exists a positive constant $C(m,d)$ depending only on $m$ and $d$ 
	such that for any positive real number $\kappa$, we have
	\begin{align}
		\int_{ \set{z\in\R^d: \|z\|_{\R^d} \geq \kappa} } \|z\|_{\R^d}^m K(\|z\|_{\R^d}) dz 
		& \leq
		C(d,m) \kappa^{d+m-2} e^{-\frac{1}{2}\kappa^2}.
	\end{align}
\end{lemma}

\begin{proof}[Proof of Lemma \ref{smoothing:aux:tail}]
	Using polar coordinates, we have
	\begin{align*}
		\int_{ \set{z\in\R^d: \|z\|_{\R^d} \geq \kappa }} \|z\|_{\R^d}^m K(\|z\|_{\R^d}) dz 
		&= C(d) \int_{\kappa}^\infty t^{d+m-1} e^{-\frac{1}{2}t^2} dt,
	\end{align*}
	where $C(d)$ is a positive constant depending only on $d$.
	Thus we immediately obtain the asserted bound from the following fact.
	
	{\bf Fact:} {\em For each non-negative integer $k$ and $\kappa > 0$, there exists a constant $C(k)$, such that}
	$$ \int_\kappa^\infty t^k e^{-\frac{1}{2}t^2}\,dt \le C(k) \, \kappa^{k-1}e^{-\frac{1}{2}\kappa^2}. $$
	
	\noindent
	The proof is standard by using repeated integration by parts. Details are omitted. 
\end{proof}

The following lemma states some simple results for the moments of $K$.

\begin{lemma} \label{smoothing:aux:quadratic}
	For any $d \times d$ symmetric matrix $A$, the following hold with $\tr(A)$ denoting the trace of $A$:
	\begin{align*}
		\int_{\R^d} K(\|z\|_{\R^d}) \left(z^\top A z\right) dz	&= \tr (A), \\
		\int_{\R^d} K(\|z\|_{\R^d}) \left(z^\top A z\right)^2 dz &= 2 \|A\|_\F^2 + \left(\tr(A)\right)^2, \\
		\int_{\R^d} K(\|z\|_{\R^d}) \left(z z^\top - I_d\right) \left(z^\top A z \right) dz &= 2A,\\
		\int_{\R^d} K^2(\|z\|_{\R^d}) dz &= \frac{1}{(4\pi)^{d/2}}, \\
		\int_{\R^d} K^2(\|z\|_{\R^d}) \left(z^\top A z\right) dz &= \frac{\tr (A)}{2(4\pi)^{d/2}}, \\
		\int_{\R^d} K^2(\|z\|_{\R^d}) \left(z^\top A z\right)^2 dz &= \frac{2 \|A\|_\F^2 + \left(\tr(A)\right)^2}{4(4\pi)^{d/2}},
	\end{align*}.
\end{lemma}

\subsection{Proof of Theorem \ref{unif:bias}}

\begin{proof}[Proof of \eqref{unif:bias-res1}]
	Let $\e_0 \le \min(1,r)$ with $r$ from  Proposition~\ref{smoothing:aux:expansion}.
	For $0<a<1$ and $\e < \e_0$, set 
	\begin{equation*}
		B_{\e,a} = \{z \in \R^d : \|z\|_{\R^d} < \e^{-a} \}, \quad \e B_{\e,a} = \{z \in \R^d : \|z\|_{\R^d} < \e^{1-a} \}.
	\end{equation*}
	We split $\T_{\e} [f] (x)$ into two parts as
	\begin{equation} 
		\T_{\e} [f] (x) = \cA_1 + \cA_2,
		\label{unif:bias:eq1}
	\end{equation}
	where
	\begin{align*}
		\cA_1 &= \frac{1}{\e^d} \int_{\Exp_x( \e B_{\e,a})} K \left( \frac{\|u - x\|_{\R^D}}{\e} \right) \rho(u) f(u) \dvol (u),\\
		\cA_2 &= \frac{1}{\e^d} \int_{\M \backslash \Exp_x( \e B_{\e,a})} K \left( \frac{\|u - x\|_{\R^D}}{\e} \right) \rho(u) f(u) \dvol (u),
	\end{align*}
	and show each part to be bounded as asserted.
	
	\begin{myenumerate}
		\myitem{Bounding $\cA_1$}
		Applying the change of variables using normal coordinates, we have
		\begin{align*}
			\cA_1 &= \frac{1}{\e^d} \int_{\Exp_x( \e B_{\e,a})} K \left( \frac{\|u - x\|_{\R^D}}{\e} \right)  \rho (u) f(u) \dvol (u) \nonumber \\ 
			& = \frac{1}{\e^d} \int_{\e B_{\e,a}} K \left( \frac{\| \Exp_x(z) - x\|_{\R^D}}{\e} \right)  \rho( \Exp_x(z) ) f( \Exp_x(z) ) G_x (z) dz \nonumber \\
			&= \int_{B_{\e,a}} K \left( \frac{\| \Exp_x(\e z) - x\|_{\R^D}}{\e} \right)  (\widetilde{\rho f})_x (\e z) G_x (\e z) dz.
		\end{align*}
		Noting that $\|\e z\|_{\R^d} < r$ for $z \in B_{\e,a}$ and all small $a$,
		we have from Proposition~\ref{smoothing:aux:expansion} that $G_x (\e z) = 1 - \frac{\e^2}{6} z^\top \Ric_x z + \e^3 O(\|z\|^3_{\R^d}),$ and
		\begin{align*}
			K \left( \frac{ \|\Exp_{x}(\e z) - x \|_{\R^D}}{\e} \right)
			&= K(\|z\|_{\R^d}) \cdot \exp \left[-\frac{1}{2} \left( \frac{ \|\Exp_{x}(\e z) - x \|_{\R^D}^2}{\e^2} -\|z\|_{\R^d}^2 \right) \right]
			\\ &= K(\|z\|_{\R^d}) \left( 1 + \frac{\e^2}{24} \sff_x (z,z)^\top \sff_x (z,z)
			+ \e^3 O(\|z\|^5_{\R^d}) \right),
		\end{align*}
		where the implied constants in $O(\cdot)$ are $\M$-dependent.
		In the last line, we used the Taylor expansion of $g(w) = e^{-\frac{1}{2}w}$ around $0$. Moreover, the Taylor expansion of $(\widetilde{\rho f})_x$ around $0$ gives
		\begin{equation} 
			(\widetilde{\rho f})_x (\e z)
			= (\widetilde{\rho f})_x (0) + \e z^\top \nabla_{\R^d} (\widetilde{\rho f})_x (0)
			+ \frac{\e^2}{2} z^\top \left[ \nabla_{\R^d}^2 (\widetilde{\rho f})_x (0) \right] z 
			+ \e^{2+\alpha} O(\|z\|_{\R^d}^{2+\alpha}),
			\label{unif:bias:eq2}
		\end{equation}
		where the implied constant in $O(\cdot)$ depends on the $C^{2,\alpha}$-norms of $\rho$ and $f$.
		
		Combining the above, $\cA_1$ can be rewritten as
		\begin{equation*} 
			\cA_1 = \cA_{11} + \cA_{12} + \cA_{13} + \cA_{14} + \cA_{15}.
		\end{equation*}
		where
		\begin{align*}
			\cA_{11} &= (\widetilde{\rho f})_x (0) \int_{B_{\e,a}} K(\|z\|_{\R^d}) dz
			= \rho(x)f(x) \int_{B_{\e,a}} K(\|z\|_{\R^d}) dz, \\
			\cA_{12} &= \e \int_{B_{\e,a}} K(\|z\|_{\R^d}) z^{\top} \nabla_{\R^d} (\widetilde{\rho f})_x(0) dz, \\
			\cA_{13} &= \e^2 \rho(x)f(x)  \int_{B_{\e,a}} K(\|z\|_{\R^d}) \left(\frac{1}{24}
			\sff_x (z,z)^\top \sff_x (z,z) - \frac{1}{6} z^\top \Ric_x z \right)dz, \\
			\cA_{14} &= \frac{\e^2}{2} \int_{B_{\e,a}} z^\top \left[ \nabla_{\R^d}^2 (\widetilde{\rho f}_x) (0) \right] z  \, dz , \\
			\cA_{15} &=\e^{2+\alpha} \int_{B_{\e,a}} K(\|z\|_{\R^d}) O (\|z\|_{\R^d}^{2+\alpha})dz.
		\end{align*}
		
		For $\cA_{11}$, by letting $a$ sufficiently small and using Lemma~\ref{smoothing:aux:tail}, we have
		\begin{align*} 
			|\cA_{11} - \rho(x)f(x) |
			&=\left| (\widetilde{\rho f})_x (0) \int_{B_{\e,a}} K(\|z\|_{\R^d}) dz - \rho(x) f(x) \right|
			\\&= |\rho(x) f(x)| \left| \int_{B_{\e,a}} K(\|z\|_{\R^d}) dz - \int_{\R^d} K(\|z\|_{\R^d}) dz \right|
			\\&= |\rho(x) f(x)| \int_{\{z: \|z\|_{\R^d} \ge \e^{-a}\}} K(\|z\|_{\R^d}) dz
			\\&\le |\rho(x) f(x)| \cdot C(d)\, \e^{-a(d-2)} e^{-\frac{1}{2}\e^{-2a}}
			\\& = O\big(\e^{\kappa}\big)
		\end{align*}
		where $\kappa > 0$ is arbitrary. 
		By symmetry of the kernel, we see $\cA_{12}=0$.
		For $\cA_{13}$ and $\cA_{14}$, similar to $\cA_{11}$, we have
		\begin{align*}
			\left|\cA_{13} + \frac{\e^2}{2} c(x) \rho(x) f(x) \right| &= O(\e^\kappa),\\
			\left|\cA_{14} - \frac{\e^2}{2} \int_{\R^d} z^\top \left[ \nabla_{\R^d}^2 (\widetilde{\rho f}_x)
			(0) \right] z  \, dz \right| &= O(\e^\kappa),
		\end{align*}
		where
		\begin{equation*}
			c(x) = \int_{\R^d} K(\|z\|_{\R^d}) \left(\frac{1}{3} z^\top \Ric_x z -\frac{1}{12} \sff_x (z,z)^\top
			\sff_x (z,z) \right)dz. 
		\end{equation*}
		In particular, using Lemma \ref{smoothing:aux:quadratic},
		\begin{align*}
			\frac{\e^2}{2} \int_{\R^d} K(\|z\|_{\R^d}) z^\top \left[ \nabla_{\R^d}^2 (\widetilde{\rho f}_x) (0)
			\right] z  \, dz
			&= \frac{\e^2}{2} \tr \left[\nabla_{\R^d}^2 (\widetilde{\rho f}_x) (0) \right]
			\\&= - \frac{\e^2}{2} \Delta_\M f(x).
		\end{align*}
		Lastly, noting all moments of $K$ is bounded, we have $\cA_{15} = O(\e^{2+\alpha})$. Therefore, by choosing $a$ and $\kappa$ properly, we conclude
		\begin{equation}
			\sup_{x \in \M} \bigg| \cA_1 - \rho(x)f(x) + \frac{\e^2}{2}\Big[ c(x)\rho(x)f(x)
			+ \Delta_\M (\rho f)(x)\Big] \bigg|
			\leq C_1 \e^{2+\alpha}.
			\label{unif:bias:eq3}
		\end{equation}
		We see that $C_1$ is $\M$-dependent, and it also depends on the $C^{2,\alpha}$-norms of f and $\rho$.
		
		\myitem{Bounding $\cA_2$}	
		Note that $\Exp_x( \e B_{\e,a})$ can be rewritten as
		\begin{equation*}
			\Exp_x( \e B_{\e,a}) = \{ u \in \M: \dist_\M (u,x) < \e^{1-a} \},
		\end{equation*}
		where $\dist_\M$ denotes the geodesic distance on $\M$.
		On the other hand, by Proposition~\ref{smoothing:aux:expansion}, for $u \in \M$ with $\|u-x\|_{\R^D} < \frac{\e^{1-a}}{2}$, we have for all sufficiently small $\e$,
		$$\dist_\M (u,x) \leq 2 \| u - x \|_{\R^D} < \e^{1-a}.$$
		This implies 
		$$ \left\{u \in \M: \|u-x\|_{\R^D} < \frac{\e^{1-a}}{2} \right\} \subset \Exp_x( \e B_{\e,a}), $$
		and thus,
		$$ \M \backslash \Exp_x( \sqrt{\e} B_{\e,a}) \subset \left\{u \in \M: \|u-x\|_{\R^D} \geq \frac{\e^{1-a}}{2} \right\}. $$
		Using this and the exponential decay of $K$, for sufficiently small $a$, we have
		\begin{align*}
			\frac{1}{\e^d} \int_{\M \backslash \Exp_x( \sqrt{\e} B_{\e,a})} 
			K \left( \frac{\|u - x\|_{\R^D}}{\e} \right) \dvol (u)
			&\leq \mathrm{vol}(\M) \left( \sup_{u \in \R^D: \|u-x\|_{\R^D} \geq \frac{\e^{1-a}}{2}} \frac{1}{\e^d} K \left( \frac{\|u - x\|}{\e} \right) \right)
			\\&= \mathrm{vol}(\M) \frac{e^{-\frac{1}{8} \e^{-2a}}}{ (2\pi)^{d/2} \e^d} 
			\\&\le C'(\kappa) \e^{\kappa},
		\end{align*}
		for $\kappa > 0$ arbitrary and a constant $C'(\kappa)$.
		Therefore, by choosing $a$ and $\kappa$ properly, we obtain
		\begin{equation}
			\sup_{x \in \M} |\cA_2| \leq C_2 \e^{2+\alpha},
			\label{unif:bias:eq4}
		\end{equation}
		where $C_2$ is $\M$-dependent, and it also depends on the uniform bounds of $\rho$ and $f$.
	\end{myenumerate}
	
	Finally, we obtain \eqref{unif:bias-res1} by combining \eqref{unif:bias:eq1}, \eqref{unif:bias:eq3} and \eqref{unif:bias:eq4}, and by letting $C_0 = C_1 + C_2$.
	
	If $f \in C^2$, then the last term in \eqref{unif:bias:eq2} becomes $o(\e^2)$.
	Consequently, the right-hand side in \eqref{unif:bias:eq2} is $o(\e^2)$, and the assertion follows. This concludes the proof of \eqref{unif:bias-res1}.
\end{proof}

\begin{proof}[Proof of \eqref{unif:bias-res2}]
	We begin by proving the following lemma:

	\begin{lemma} \label{unif:bias:denom}
		Under Assumptions \ref{smoothing:assume:manifold}, \ref{smoothing:assume:density}, \ref{smoothing:assume:kernel},
		there exist constants $\e_0,C_0 > 0$ such that with $c(x)$ from \eqref{unif:bias-res1}, if $\e < \e_0$, 
		\begin{equation*}
			\frac{1}{\T_\e [1] (x)} 
			= \frac{1}{\rho(x)} + \frac{\e^2}{2\rho^2(x)} \left[ c(x) \rho(x) - \Delta_\M \rho (x) \right] + O(\e^{2+\alpha}).
		\end{equation*}
		The constants $\e_0$ and $C_0$ are $\M$-dependent, and they also depend on the $C^{2,\alpha}$-norm and the minimum of $\rho$.
	\end{lemma}

	\begin{proof}[Proof of Lemma \ref{unif:bias:denom}]
		From the identity
		\begin{equation*}
			\frac{1}{z+w} = \frac{1}{w} - \frac{z}{w^2} + \frac{z^2}{w^2(z+w)},
		\end{equation*}
		substituting $z=\T_\e [1] (x) - \rho(x)$ and $w=\rho(x)$ yields
		\begin{equation}
			\frac{1}{\T_\e [1] (x)}
			= \frac{1}{\rho(x)} - \frac{\T_\e [1] (x)-\rho (x)}{\rho^2(x)}
			+ \frac{(\T_\e [1] (x)-\rho (x))^2}{\rho^2(x) \T_\e[1] (x)}.
			\label{unif:bias:den-eq1}
		\end{equation}
		
		By $\eqref{unif:bias-res1}$ with $f \equiv 1$, there exist positive constants $\e_1$ and $C_1$ such that for $\e < \e_1$, we have
		\begin{equation} 
			\sup_{x \in \M} \bigg| \T_{\e} [1] (x) - \rho(x)
			+ \frac{\e^2}{2}\Big[ c(x)\rho(x) + \Delta_\M \rho (x)\Big] \bigg|
			\leq C_1 \e^{2+\alpha}.
			\label{unif:bias:den-eq2}
		\end{equation}
		This gives
		\begin{equation}
			\sup_{x \in \M} \left| - \frac{\T_\e [1] (x)-\rho (x)}{\rho^2(x)}
			- \frac{\e^2}{\rho^2(x)}  [c(x)\rho(x) - \Delta_\M \rho (x)] \right| 
			\leq \frac{C_1 \e^{2+\alpha}}{\rho^2_\tmin},	
			\label{unif:bias:den-eq3}
		\end{equation}
		where $\rho_\tmin$ denotes the minimum of $\rho$. On the other hand, by letting $\e_1$ sufficiently small, we have from \eqref{unif:bias:den-eq2},
		\begin{align}
			\sup_{x \in \M} \left| \T_{\e} [1] (x) - \rho(x) \right| \leq \min \left\{ C_2 \e^2,
			\frac{\rho_\tmin}{2}  \right\},
		\end{align}
		where $C_2 >0$ is a constant.
		Then, it follows that
		\begin{equation*}
			\sup_{x \in \M} (\T_\e [1] (x)-\rho (x))^2 \leq C_2^2 \e^4,
			\quad
			\sup_{x \in \M} \frac{1}{\T_\e[1](x)} \leq \frac{2}{\rho_\tmin},
		\end{equation*}
		and hence,
		\begin{equation}
			\sup_{x \in \M} \left| \frac{(\T_\e [1] (x)-\rho (x))^2}{\rho^2(x) \T_\e[1] (x)} \right| 
			\leq \frac{C_2^2 \e^4}{\rho_\tmin^3}.
			\label{unif:bias:den-eq4}
		\end{equation}
		
		Therefore, we complete the proof by combining \eqref{unif:bias:den-eq1}, \eqref{unif:bias:den-eq3} and \eqref{unif:bias:den-eq4}, and by adjusting $\e_0$ and $C_0$ appropriately (which may depend on $\e_1$, $C_1$ and $\rho_\tmin$).
	\end{proof}

	We now turn to the proof of \eqref{unif:bias-res2}. From \eqref{unif:bias-res1} and Lemma \ref{unif:bias:denom}, for all sufficiently small $\e$, we have
	\begin{align*}
		\T_\e [f] (x) &= \rho(x) f(x) -\frac{\e^2}{2}\left[ c(x) \rho(x) f(x) + \Delta_\M (\rho f ) (x) \right] + O(\e^{2+\alpha}), \\
		\frac{1}{\T_\e [1] (x)} &= \frac{1}{\rho(x)} + \frac{\e^2}{2\rho^2(x)} \left[ c(x) \rho(x) - \Delta_\M \rho (x) \right] + O(\e^{2+\alpha}).
	\end{align*}
	Combining these yields
	\begin{align}
		\TT_{\e} [f] (x)
		&= f(x) -\frac{\e^2}{2\rho(x)}\left[ c(x) \rho(x) f(x) - \Delta_\M (\rho f ) (x) \right]
		+ \frac{\e^2 f(x)}{2\rho(x)} \left[ c(x) \rho(x) - \Delta_\M \rho (x) \right] + O(\e^{2+\alpha})
		\notag\\&= f(x) +\frac{\e^2}{2}\cdot \frac{\Delta_\M (\rho f ) (x) - f(x) \Delta_\M \rho (x) }{\rho(x)} + O(\e^{2+\alpha}).
		\label{unif:bias-res2:eq1}
	\end{align}
	
	Here, the product rule for the Laplace-Beltrami operator gives
	\begin{align*}
		\Delta_\M (\rho f ) - f \Delta_\M \rho
		&= \rho \Delta_\M f + f \Delta_\M \rho -2 \langle \nabla_\M \rho, \nabla_\M f \rangle_{T_x \M} - f \Delta_\M \rho
		\\&= \rho \Delta_\M f -2 \langle \nabla_\M \rho, \nabla_\M f \rangle_{T\M},
	\end{align*}
	where $\langle \cdot, \cdot \rangle_{T \M}$ denotes the Riemannian inner product.
	On the other hand, we can rewrite the weighted Laplace-Beltrami operator as
	\begin{align*}
		\Delta_{\M,2} f 
		= -\frac{1}{\rho^2}\div_\M (\rho^2 \nabla_\M f) 
		= \frac{\rho^2 \Delta_\M f -2 \rho \langle \nabla_\M \rho, \nabla_\M f \rangle_{T \M} }{\rho^2}.
	\end{align*}
	This yields
	\begin{equation}
		\frac{\Delta_\M (\rho f ) - f \Delta_\M \rho }{\rho}
		= \frac{\rho^2 \Delta_\M f -2 \rho \langle \nabla_\M \rho, \nabla_\M f \rangle_{T\M}}{\rho^2}
		= \Delta_{\M,2} f.
		\label{unif:bias-res2:eq2}
	\end{equation}
	Therefore, substituting \eqref{unif:bias-res2:eq2} into \eqref{unif:bias-res2:eq1} completes the proof.
\end{proof}

\subsection{Proof of Theorem \ref{unif:sto}}

Hereafter, for a bounded real-valued function $f$, we denote by $\|f\|_\infty$ the sup-norm of $f$.
The sup-norms of gradients and Hessians is the maximum sup-norm of their components.

First observe that with $\msF_\e = \left\{ g_x(\cdot) = K \left( \frac{ \|\cdot - x\|_{\R^D} }{\e} \right) f(\cdot) : x \in \M\right\},$ we can write
\begin{equation}
	\sup_{x \in \M}\big|\T_{n,\e} [f] (x) - \T_{\e} [f] (x)  \big|
	= \frac{1}{\e^{d}} \sup_{g \in \msF_\e} 
	\left| \frac{1}{n} \sum_{i=1}^n g(X_i) - \E_\P g(X) \right|,
	\label{unif:sto-eq1}
\end{equation}
This point of view enables us to use empirical process theory.
To this end, we use some facts about Vapnik–Chervonenkis (VC) type classes:

\begin{definition}[Uniformly bounded VC type class]
	A class $\mathscr{G}$ of measurable functions on $\M$ is called \emph{uniformly bounded VC type class} on $\M$ 
	if there exist positive constants $A,B$ and $\nu$ satisfying the following:
	\begin{enumerate}[label=(\alph*)]
		\item $\displaystyle \sup_{g \in \mathscr{G}} ||g||_\infty \leq B.$
		\item For any probability measure $\Q$ on $\M$ and any $\eta \in (0, B)$, the covering number $\mathscr{N} ( \mathscr{G}, L^2 (\Q), \eta )$, the minimal number of $L^2(\Q)$-balls needed to cover ${\mathscr G}$, satisfies 
		$$ \mathscr{N} ( \mathscr{G}, L^2 (\Q), \eta ) \leq \left( \frac{AB}{\eta} \right)^\nu \text{.}$$
	\end{enumerate}
	Below we refer to the constants $(\nu, A, B)$ as the {\em characteristics} (or {\em VC characteristics}) of $\mathscr G$.
\end{definition}

\begin{proposition} \label{smoothing:aux:VC}
	Suppose Assumptions \ref{smoothing:assume:manifold} and \ref{smoothing:assume:density} hold.
	Let $\mathscr{G} = (g_x)_{x \in \M}$ be a uniformly bounded VC type class on $\M$ with characteristics $(\nu,A,B).$ 
	Let $\sigma>0$ with $\sup_{g \in \msF} \mathbb{E}_\P g^2 (X) \leq \sigma^2$.
	Then, there exists a universal constant $C > 0$ not depending on any parameters such that with probability at least $1-\delta$,
	\begin{multline}
		\sup _{g \in \mathscr{G}}\left|\frac{1}{n} \sum_{i=1}^n g\left(X_i\right)-\E_\P g(X)\right|
		\leq C \Bigg(
		\sqrt{\frac{\nu \sigma^2}{n} \log \left(\frac{2 A B}{\sigma}\right)}
		+\sqrt{\frac{\sigma^2 \log (1/\delta)}{n}}
		\\ +\frac{\nu B}{n} \log \left(\frac{2 A B}{\sigma}\right)
		+ \frac{B \log(1/\delta)}{n} \Bigg).    
	\end{multline}
\end{proposition}

This proposition is a slightly modified version of Theorem 30 in the supplement material of \cite{kim2019uniform}.
The only difference is that in our case, the functions in $\msF_\e$ are defined on $\M$, while \cite{kim2019uniform} considers functions defined on $\R^d$.
The proof is identical. 

The following lemma is needed to control the covering number and the $L^2$-type bounds for functions in $\msF_\e$.

\begin{lemma}[Lemma 4.2 of \cite{loubes2008kernel}] \label{smoothing:aux:Mcovering}
	Under Assumption \ref{smoothing:assume:manifold},
	there exist $\M$-dependent constants $\eta_\M$ and $A_\M$ such that
	if $\eta < \eta_\M$, we have
	\begin{equation}
		\mathscr{N} \left( \M, \dist_\M, \eta \right) \leq \left( \frac{A_\M}{\eta} \right)^{d}.
	\end{equation}	
\end{lemma}

Equipped with these, we first prove \eqref{unif:sto-res1}.

\begin{proof}[Proof of \eqref{unif:sto-res1}]
	By definition of $\msF_\e$,
	\begin{equation} 
		\sup_{g \in \msF_\e} \|g\|_\infty 
		= \sup_{x \in \M} \left\| K\left(\frac{ \| \cdot - x \|_{\R^D}}{\e}\right) f(\cdot) \right\|_\infty
		\leq B: =  \|f\|_\infty.
		\label{unif:sto-res1-eq1}
	\end{equation}
	
	Next, by following the same ideas as in the proof of Theorem~\ref{unif:bias}, 
	we can choose $\e_1$, $C_1 > 0$ such that for $\e < \e_1$, we have
	\begin{equation}
		\sup_{g \in \msF_\e} \E_\P g^2(X)
		\leq \|f\|_\infty^2 \sup_{x \in \M} \E_\P K \left( \frac{\|X-x\|_{\R^D}}{\e} \right)
		\leq \sigma^2 := C_1 \|f\|_\infty^2 \| \rho \|_\infty \e^d,
		\label{unif:sto-res1-eq2}
	\end{equation}
	where $\e_1$ and $C_1$ are $\M$-dependent.
	
	Lastly, since $K\left(\| \cdot - x \|_{\R^D} \right)$ is $1$-Lipschitz, we have for $x,x',u \in \M$,
	\begin{align*}
		\left| g_x(u) - g_{x'} (u) \right| 
		\leq \frac{\|f\|_\infty}{\e} \|x - x'\|_{\R^D}
		\leq \frac{\|f\|_\infty}{\e} \dist_\M (x,x').
	\end{align*}
	So, for any probability measure $\Q$ on $\M$, the $L^2(\Q)$ distance between $g_x$ and $g_{x'}$ is upper bounded by
	\begin{align*}
		\left\| g_x - g_{x'} \right\|_{L^2 (\Q)}
		&= \sqrt{\E_\Q \left| g_x (X) - g_{x'} (X) \right|^2 }
		\leq \frac{\|f\|_\infty}{\e} \dist_\M (x,x').
	\end{align*}
	Therefore, by using Lemma \ref{smoothing:aux:Mcovering}, if 
	$\eta < \frac{\|f\|_\infty}{\e} \eta_\M$,
	the covering number $\mathscr{N} \left( \msF_\e, L^2 (\Q), \eta \right)$ is upper bounded by
	\begin{equation}
		\sup_{\Q} \mathscr{N} \big( \msF_\e, L^2(\Q),\eta\big)	\leq \mathscr{N} \left( \M, \dist_\M, \frac{\e}{\|f\|_\infty} \eta \right) 
		\leq \left(\frac{A_\M \frac{\|f\|_\infty}{\e}}{\eta} \right)^d.
		\label{unif:sto-res1-eq3}
	\end{equation}
	
	By combining \eqref{unif:sto-res1-eq1} and \eqref{unif:sto-res1-eq3}, it follows that $\msF_\e$ is a VC-class with triple $(\nu,A,B)$, where $\nu = d$, $A = \frac{A_\M}{\e}$ and $B = \|f\|_\infty$, provided $\e \le \eta_\M.$ 
	
	Putting everything together, it follows from \eqref{unif:sto-res1-eq2} and Proposition~\ref{smoothing:aux:VC} that if $\e < \e_0:= \min(\e_1,\eta_\M)$, then  with probability at least $1-\delta$, we have
	\begin{align*}
		&\quad \sup_{x \in \M} \left| \T_{n,\e} [f] (x) - \T_{\e} [f] (x)  \right|
		\\&\leq 
		\frac{C}{\e^d} \Bigg(
		\sqrt{\frac{C_1 \|f\|_\infty \|\rho\|_\infty \e^d}{n} \log \left( \frac{2 A_\M \frac{\|f\|_\infty}{\e}  }{\sqrt{C_1 \|\rho\|_\infty \e^d}} \right)}
		+ \sqrt{\frac{C_1 \|f\|_\infty \|\rho\|_\infty\e^d \log (1/\delta)}{n}}\\
		&\hspace{16em}+ \frac{d \|f\|_\infty}{n} \log \left( \frac{2 A_\M \frac{\|f\|_\infty}{\e}   }{\sqrt{C_1 \|\rho\|_\infty \e^d}} \right)
		+ \frac{\|f\|_\infty \log(1/\delta)}{n}
		\Bigg)
		\\&=  
		C_0 \Bigg(
		\sqrt{\frac{\log (1/\e)}{n\e^d} }
		+ \sqrt{\frac{\log (1/\delta)}{n\e^d}}
		+ \frac{\log (1/\e)}{n \e^d}
		+ \frac{\log(1/\delta)}{n\e^d}
		\Bigg)
	\end{align*}
	for an appropriate choice of $C_0$. This implies \eqref{unif:sto-res1}.
\end{proof}

Before proving \eqref{unif:sto-res2}, we present some heuristic arguments to see where the faster rate of convergence in the normalized case originates.

From $f(u) = f(x) + \e \cdot \frac{f(u) - f(x)}{\e}$, we can express $\T_{n,\e}[f](x)$ as
\begin{equation*}
	\T_{n,\e}[f](x) = f(x) \, \T_{n,\e} [1] (x) + \e \, \bfS_{n,\e}[f](x),
\end{equation*}
where
\begin{equation}
	\bfS_{n,\e}[f](x) = \frac{1}{n\e^d} \sum_{i=1}^n K \left( \frac{\|X_i - x\|_{\R^D}}{\e} \right) \frac{f(X_i)-f(x)}{\e}.
\end{equation}
The first term does not involve the variability of the function $f$, but rather that of the density $\rho$ through $\T_{n,\e} [1] (x)$.
For the second term, since the kernel localizes the sum about $x$ within the bandwidth $\e$, the increment $f(X_i)-f(x)$ scaled by $\e$ behaves like the gradient of $f$ at $x$.
Consequently, $\bfS_{n,\e}[f](x)$ acts like the unnormalized kernel smoothing of a gradient-like quantity and therefore it has the same variance order as the unnormalized kernel smoothing.
Taking this together with the multiplicative factor $\e$ in front, the variability of the second term is expected to be smaller than that of the first.

On the other hand, for $\TT_{n,\e}[f](x)$, using the above decomposition yields
\begin{align*}
	\TT_{n,\e}[f](x) = \frac{\T_{n,\e}[f](x)}{\T_{n,\e}[1](x)} = f(x) + \frac{\e \, \bfS_{n,\e}[f](x)}{\T_{n,\e}[1](x)}.
\end{align*}
We see that the first term is deterministic, since the stochastic term $\T_{n,\e} [1] (x)$ is removed by normalization.
Hence, the variability of $\TT_{n,\e}[f](x)$ arises solely from the second term, which has a smaller variance order than that of $\T_{n,\e}[f](x)$, when ignoring the denominator $\T_{n,\e}[1] (x)$.

We now make this heuristic precise. To this end, we first prove two lemmas; the first provides a uniform bound for the denominator $\T_{n,\e}[1] (x)$, and the second establishes the uniform convergence of the second moment of $\bfS_{n,\e}[f](x)$.

\begin{lemma} \label{unif:sto:denom}
	Under Assumptions \ref{smoothing:assume:manifold}, \ref{smoothing:assume:density}, \ref{smoothing:assume:kernel},
	there exists positive constants $\e_0^*$ and $C_0^*$ such that if
	\begin{equation}
		C_0^*\left(\frac{\log n \vee \log(1/\delta)}{n}\right)^{1/d} < \e < \e_0^*,
		\label{unif:sto:denom-eq2}
	\end{equation}
	we have with probability at least $1-\delta$, 
	\begin{equation}
		\inf_{x \in \M} \T_{n,\e} [1] (x) \ge  \frac{1}{2} \inf_{x\in\M}\rho(x),
	\end{equation}
	where $\e_0^*$ and $C_0^*$ are $\M$-dependent, and they also depend on $\rho_\tmin$ and the $C^2$-norm of $\rho$.
\end{lemma}

\begin{proof}[Proof of Lemma \ref{unif:sto:denom}]
	From \eqref{unif:bias-res1} and \eqref{unif:sto-res1},
	there exist $\e_1, C_1 > 0$ such that whenever $\e < \e_1$, we have with probability at least $1-\delta$,
	\begin{equation}
		\sup_{x \in \M} \left| \T_{n, \e} [1] (x) - \rho(x) \right|
		\leq 
		C_1 \left( \e^2 + 
		\sqrt{ \frac{\log(1/(\e\wedge\delta))}{n\e^{d}} }  
		+ \frac{\log(1/(\e\wedge \delta))}{n \e^{d  }} \right).
		\label{unif:sto:denom-eq1}
	\end{equation}
	So if $n,\e$ and $\delta$ are such that the right-hand side is smaller than $\inf_{x\in\M} \rho(x)/2$, then we can conclude that $\inf_{x\in\M} \T_{n,\e}[1](x) \ge \inf_{x\in\M} \rho(x)/2$ for otherwise there would be a contradiction.
	Thus, all we need is to assure that under our assumption, the right-hand side is small enough. 
	
	To see this, first observe that for $\e_0^*$ small enough, we have $C_1\e^2 < C_1(\e_0^*)^2 \le \inf_{x\in\M} \rho(x)/4.$ As for the second and third terms on the right-hand side of \eqref{unif:sto:denom-eq1}, notice that if $\e,\delta > 0$ are such that $ C_0^* \left(\frac{ \log (1/(\e\wedge\delta))}{n}\right)^{1/d} < \e < \e_0^*$, then
	$$ C_1 \left( \sqrt{\frac{\log (1/(\e\wedge\delta))}{n \e^{d}}} + \frac{\log (1/(\e\wedge\delta)) }{n \e^{d}} \right)
	\leq C_1 \left( \sqrt{ \Big(\frac{1}{C_0^*}\Big)^d } + \Big(\frac{1}{C_0^*}\Big)^d \right),$$
	and by appropriate choice of $C_0^*$, we can make the right-hand side smaller than $\inf_{x\in\M} \rho(x)/4$. It remains to observe that assumption \eqref{unif:sto:denom-eq2} implies that $\e \ge \frac{1}{n^{1/d}}$,  so that $\log \frac{1}{\e} \le \frac{1}{d} \log n \le \log n.$ 
	In other words, assumption \eqref{unif:sto:denom-eq2} implies the condition  $C_0^*\left(\frac{ \log (1/(\e\wedge\delta)}{n}\right)^{1/d} < \e$. This completes the proof.
\end{proof}

\begin{lemma} \label{unif:sto:var}
	Suppose Assumptions \ref{smoothing:assume:manifold}, \ref{smoothing:assume:density}, \ref{smoothing:assume:kernel} hold
	and $f \in C^2(\M)$.
	Then, there exist $\e_0, C_0$ such that $\e < \e_0$ implies
	\begin{align}
		\sup_{x \in \M} \left| \frac{1}{\e^d} \int_\M K^2 \left( \frac{\|u-x\|_{\R^D}}{\e} \right) \left(\frac{f(u)-f(x)}{\e} \right)^2 d\P(u) - \frac{\rho(x) \|\nabla_\M f (x)\|^2 }{2(4\pi)^{d/2}}\right|
		& \leq C_0 \e.
	\end{align}
	Here, the constants $\e_0$ and $C_0$ are $\M$-dependent, and they also depend on the $C^2$-norm $f$ and $\rho$.
\end{lemma}

\begin{proof}[Proof of Lemma \ref{unif:sto:var}]
	Note that $K^2$ is still a Gaussian kernel.
	Thus, by using similar arguments to those in Theorem~\ref{unif:bias}, for all sufficiently small $\e$, we have
	\begin{align*}
		&\quad \frac{1}{\e^d} \int_\M K^2 \left( \frac{\|u-x\|_{\R^D}}{\e} \right) \left(\frac{f(u)-f(x)}{\e} \right)^2 d\P(u)
		\\&= \int_{\R^d} K^2(\|z\|_{\R^d}) \Big( z^\top \nabla_{\R^d} \tilde{f}_x (0) + \e O(\|z\|_{\R^d}^2) \Big)^2 \Big( \tilde{\rho}_x (0) + \e z^\top \nabla \tilde{\rho}_x (0) + \e^2 O(\|z\|_{\R^d}^2) \Big) dz + O(\e^2)
		\\&= \rho(x) \int_{\R^d} K^2(\|z\|_{\R^d}) \left( z^\top \nabla_{\R^d} \tilde{f}_x (0) \right)^2 dz + O(\e).
	\end{align*}
	Here, by Lemma \ref{smoothing:aux:quadratic}, the integral in the leading term can be simplified as
	\begin{align*}
		\int_{\R^d} K^2(\|z\|_{\R^d}) \left( z^\top \nabla_{\R^d} \tilde{f}_x (0) \right)^2 dz
		&= \int_{\R^d} K^2(\|z\|_{\R^d}) z^\top \left( \nabla_{\R^d} \tilde{f}_x (0) \nabla \tilde{f}_x (0)^\top \right) z  dz
		\\&= \frac{\tr \left(\nabla_{\R^d} \tilde{f}_x (0) \nabla \tilde{f}_x (0)^\top\right)}{2(4\pi)^{d/2}}
		= \frac{\|\nabla_{\M} f(x) \|^2}{2(4\pi)^{d/2}}.
	\end{align*}
	Therefore, we complete the proof by choosing $\e_0$ and $C_0$ appropriately.
\end{proof}

Now we proceed to the proof of \eqref{unif:sto-res2}.

\begin{proof}[Proof of \eqref{unif:sto-res2} in Theorem \ref{unif:sto}]
	By setting
	\begin{equation*}
		\bfS_\e [f](x) = \frac{1}{\e^d} \int_\M K \left( \frac{\|X - x\|_{\R^D}}{\e} \right) \frac{f(X)-f(x)}{\e},
	\end{equation*}
	we can write
	\begin{equation*}
		\TT_{n,\e} [f] (x)- \TT_{\e} [f] (x) = \cA_1 \cA_2 - \cA_3 \cA_4,
	\end{equation*}
	where
	\begin{alignat*}{4}
		\cA_1 &= \frac{\e}{\T_{n,\e}[1](x)}, &
		\cA_2 &= \bfS_{n,\e}[f](x) - \bfS_{\e}[f](x), \\
		\cA_3 &= \frac{\TT_{\e}[f](x) - f(x)}{\T_{n,\e}[1](x)}, \quad &
		\cA_4 &= \T_{n,\e}[1](x) - \T_{\e}[1](x). 
	\end{alignat*}
	We can bound $\cA_1, \cA_3, \cA_4$ by using \eqref{unif:sto-res1}, Lemma~\ref{unif:sto:denom} and Theorem~\ref{unif:bias}. Hence, it suffices to check $\cA_2$.
	
	The details of this derivation are quite similar to the proof of (\ref{unif:sto-res1}). In order to avoid repetitious arguments, we only outline the main steps of the proof.
	
	With
	$ \msG_\e = \left\{ g_x(\cdot) = K \left( \frac{\|\cdot - x\|_{\R^D}}{\e} \right) \frac{f(\cdot) - f(x)}{\e} : x \in \M \right\},$ we can write
	$$ \sup_{x \in \M} \left| \cA_2 \right| = \frac{1}{\e^d} \sup_{g \in \msG_\e} \left| \frac{1}{n} \sum_{i=1}^{n} g(X_i) - \E_\P g(X) \right|. $$
	Then, we first have
	$$ \sup_{g \in \msG_\e} \| g\|_\infty \leq B \coloneqq \frac{2 \|f\|_\infty }{\e}. $$
	Next, using Lemma \ref{unif:sto:var},
	it can be shown that there exist $\e'_2, C'_2 >0$ such that
	$$ \sup_{g \in \msG_\e} \E_\P g^2(X) \leq \sigma^2 \coloneqq C'_2 \e^d $$
	whenever $\e < \e'_2$.
	We further see that for all sufficiently small $\e$,
	\begin{align*}
		|g_x (u) - g_{x'} (u)|
		\leq \left( \frac{\|\nabla_\M f\|_\infty}{\e} + \frac{2\|f\|_\infty}{\e^2} \right) \|x - x'\|_{\R^D}
		\leq \frac{4\|f\|_\infty}{\e^2}  \dist_\M (x,x').
	\end{align*}
	So, there exists $\e_2''>0$ such that whenever $\e < \e''_2$ and $\eta < B$, we have
	\begin{align*}
		\sup_\Q \mathscr{N} \left( \msG_\e, L^2(\Q), \eta \right)
		\leq & \left(\frac{A_\M \cdot \frac{4\|f\|_\infty}{\e^2} }{\eta} \right)^d
	\end{align*}
	Therefore, we can find constants $\e_2$ and $C_2$  such that
	if $\e <\e_2$, with probability at least $1-\delta/4$, we have
	\begin{equation*}
		\sup_{x \in \M} \left| \cA_2 \right|
		\leq C_2 \left( \sqrt{ \frac{\log(1/(\e\wedge\delta))}{n\e^{d}} } 
		+ \frac{\log(1/(\e\wedge \delta))}{n \e^{d+1}} \right).
	\end{equation*}
	By putting everything together, we conclude the proof.
\end{proof}

\subsection{Proof of Theorem \ref{Berry:unnor}}

We first present some supporting lemmas.

The following lemma will be used to show the asymptotic independence of kernel smoothing at two distinct points in $\M$. 

\begin{lemma} \label{Berry:unnor:aux1}
	Suppose Assumptions \ref{smoothing:assume:manifold}, \ref{smoothing:assume:density}, \ref{smoothing:assume:kernel} hold.
	Then, for distinct points $x_1, x_2 \in \M$ and $\kappa>0,$ there exists positive constants $\e_0$ and $C_0$ such that for all $\e < \e_0$,
	\begin{equation}
		\frac{1}{\e^d} \int_{\M} K\left(\frac{\|u-x_1\|_{\R^D}}{\e}\right) K\left(\frac{\|u-x_2\|_{\R^D}}{\e}\right) \dvol (u) < C \e^\kappa.
	\end{equation}
\end{lemma}

\begin{proof}[Proof of Lemma \ref{Berry:unnor:aux1}]
	Let $0<a<1$ and set $\e_{a} = (1 \wedge \|x_1 - x_2\|_{\R^D})^{1/a}$.
	Then, for $\e <\e_a,$ we have 
	$$ \{u \in \M :\|u-x_1\|_{\R^D} > \e^a\} \cup \{u \in \M:\|u-x_2\|_{\R^D} > \e^a\} = \M.$$
	Using this, we see that for $\e < \e_a$,
	\begin{align*}
		& \frac{1}{\e^d} \int_{\M} K\left(\frac{\|u-x_1\|_{\R^D}}{\e}\right) K\left(\frac{\|u-x_2\|_{\R^D}}{\e}\right) \dvol (u)
		\\& \leq \frac{1}{\e^d} \int\limits_{\M \cap \set{u: \|u-x_1\|\geq\e^a} } K\left(\frac{\|u-x_1\|_{\R^D}}{\e}\right) \dvol (u)
		+ \frac{1}{\e^d} \int\limits_{\M \cap \set{u: \|u-x_2\|\geq\e^a} } K\left(\frac{\|u-x_2\|_{\R^D}}{\e}\right) \dvol (u)
		\\& \leq \frac{2 \mathrm{vol} (\M) e^{-\frac{1}{2}\e^{2(1-a)}}}{(2 \pi)^{d/2}\e^{d}}.
	\end{align*}
	The remainder of the proof is standard.
\end{proof}

We present three lemmas for controlling Berry-Esseen type bounds.

\begin{lemma}[Lemma 8 of \cite{panov2015finite}] \label{Berry:unnor:compare}
	For $\cN_m(\mu',\Sigma')$ and $\cN_m(\mu,\Sigma)$, suppose
	$$ \|\Sigma^{-1/2} \Sigma' \Sigma^{-1/2} - I_m \|_\op \leq \zeta \leq 1/2, \quad 
	\|\Sigma^{-1/2} \Sigma' \Sigma^{-1/2} - I_m \|_\F \leq \delta, $$
	where $\|\cdot\|_\op$ denotes the operator norm.
	Then, for any measurable set $A \subset \mathbb{R}^m$, it holds that
	\begin{equation*}
		|\Phi_{m,\mu',\Sigma'}(A) - \Phi_{m,\mu,\Sigma} (A)| \leq \frac{1}{2} \sqrt{\delta^2 + (1+\zeta)(\mu'-\mu)^\top \Sigma (\mu'-\mu)}.
	\end{equation*}	
\end{lemma}

\begin{lemma} \label{Berry:unnor:matconv}
	For two strictly positive definite matrices $\Sigma, \Sigma' \in \R^{m \times m}$, it holds that
	\begin{equation}
		\|\Sigma^{-1/2} \Sigma' \Sigma^{-1/2} - I_m\|_\op \leq \|\Sigma^{-1}\|_\op \|\Sigma' - \Sigma\|_\op
	\end{equation}
	Moreover, if $\|\Sigma^{-1/2} \Sigma' \Sigma^{-1/2} - I_m\|_\op \leq \zeta \leq 1/2$, then
	\begin{equation}
		\| (\Sigma')^{-1/2} \|_\op \leq \|\Sigma^{-1/2}\|_\op \sqrt{1+2 \zeta}.
	\end{equation}
\end{lemma}

\begin{proof}[Proof of Lemma \eqref{Berry:unnor:matconv}]
	From the identity
	$\Sigma^{-1/2} (\Sigma' - \Sigma) \Sigma^{-1/2} = \Sigma^{-1/2} \Sigma' \Sigma^{-1/2} - I_m,$ 
and the sub-multiplicativity of the operator norm, we have
\begin{align*}
	\|\Sigma^{-1/2} \Sigma' \Sigma^{-1/2} - I_m\|_\op 
	&\leq \|\Sigma^{-1/2}\|_\op \| \Sigma' - \Sigma \|_\op \|\Sigma^{-1/2}\|_\op
	\\&= \|\Sigma^{-1}\|_\op \| \Sigma' - \Sigma \|_\op.
\end{align*}
This implies the first assertion. Next, similarly, from the identity
$\Sigma' = \Sigma^{1/2} [ I_m + (\Sigma^{-1/2} \Sigma' \Sigma^{-1/2} - I_m) ] \Sigma^{1/2},$
we obtain
\begin{equation}
	\|(\Sigma')^{-1}\|_\op \leq \| \Sigma^{-1}\|_\op \left\| [I_m + (\Sigma^{-1/2} \Sigma' \Sigma^{-1/2} - I_m)]^{-1} \right\|_\op 
	\label{Berry:unnor:matconv-eq1}
\end{equation}

Note that from the inequality $ |(1+u)^{-1} - 1| \leq 2|u| $ with $|u| < 1/2$, applying the matrix function argument (see, e.g., Theorem 6.2.8 of \cite{horn1991topics}) gives
\begin{align*}
	\| (I_m + E)^{-1} - I_m \|_\op \leq 2\|E\|_\op, \quad\text{and}\quad \| (I_m + E)^{-1} \|_\op \leq 1+ 2\|E\|_\op,
\end{align*}
for any symmetric matrix $E$ with $\|E\|_\op \leq 1/2$.

Thus, by substituting $E = \Sigma^{-1/2} \Sigma' \Sigma^{-1/2} - I_m$ and taking square root on both sides of \eqref{Berry:unnor:matconv-eq1}, the proof is complete.
\end{proof}

\begin{lemma} \label{Berry:unnor:error}
	For random vectors $Z, W \in \R^m$, a probability measure $\Q$ on $\R^m$, and $\delta >0$, it holds that
	\begin{multline}
		\sup_{A \in \cvx(\R^m)} |\Q( Z+W \in A) - \Phi_{m,\mu,\Sigma} (A)|
		\leq \sup_{A \in \cvx(\R^m)} |\Q( Z \in A) - \Phi_{m,\mu,\Sigma}(A)|
		\\+ C_m \delta \sqrt{\|\Sigma^{-1}\|_\F} + \Q(\|W\|_{\R^m} > \delta).
	\end{multline}	
	where $C_m$ is a constant depending only on $m$.
\end{lemma}
	
\begin{proof}[Proof of Lemma \ref{Berry:unnor:error}]
	For $A \in \cvx(\R^m)$, set
	\begin{align*}
		A_{+\delta} = \{ x \in \R^m: \dist(x,A) \leq \delta \}, \quad
		A_{-\delta} = \{ x \in \R^m: B_{m,\delta} (x) \subset A \},
		\end{align*}
	where $\dist(x,A) = \inf_{y \in A} \|x-y\|_{\R^m}$ and $B_{m,\delta} = \{y \in \R^m: \|x-y\|_{\R^m} \leq \delta\} $.
	Then, it can be easily shown that
	\begin{align*}
		\Q(Z \in A_{-\delta}) &\leq \Q(Z + W \in A) + \Q( \|W\|_{\R^m} > \delta), \\
		\Q(Z+W \in A) &\leq \Q(Z \in A_{+\delta}) + \Q( \|W\|_{\R^m} > \delta).
		\end{align*}
	Hence, we obtain
	\begin{align}
		|\Q( Z+W \in A) - \Phi_{m,\mu,\Sigma}(A)|
		&\leq |\Q(Z \in A_{-\delta}) - \Phi_{m,\mu,\Sigma}(A)| \vee |\Q(Z \in A_{+\delta}) 
		- \Phi_{m,\mu,\Sigma}(A)|
		\notag\\&\quad + \Q( \|W\|_{\R^m} > \delta).
		\label{Berry:unnor:error-eq1}
	\end{align}
	Using the identities
	\begin{align*}
		\Phi_{m,\mu,\Sigma}(A) = \Phi_{m,\mu,\Sigma}(A_{-\delta}) + \Phi_{m,\mu,\Sigma}(A \backslash A_{-\delta}) 
		= \Phi_{m,\mu,\Sigma}(A_{+\delta}) - \Phi_{m,\mu,\Sigma}(A_{+\delta} \backslash A),
	\end{align*}
	and the fact that both $A_{+\delta}$ and $A_{-\delta}$ are convex, 
	the first term on the right-hand side of \eqref{Berry:unnor:error-eq1} can be bounded by
	\begin{equation}
	\sup_{A' \in \cvx(\R^m)} |\Q( Z \in A') - \Phi_{m,\mu,\Sigma}(A')| + \Phi_{m,\mu,\Sigma}(A \backslash A_{-\delta}) \vee \Phi_{m,\mu,\Sigma}(A_{+\delta} \backslash A).
	\end{equation}
	According to Lemma 2.6 of \cite{bentkus2003dependence}, there exists a constant $C_m$ depending only on $m$ such that
	\begin{equation}
	\Phi_{m,\mu,\Sigma} (A \backslash A_{-\delta}) \vee \Phi_{m,\mu,\Sigma} (A_{+\delta} \backslash A) \leq C_m \delta \sqrt{\|\Sigma^{-1}\|_\F}.
	\end{equation}
	Therefore, combining these bounds completes the proof.
\end{proof}

Now we turn to the proof of Theorem \ref{Berry:unnor}. 

\begin{proof}[Proof of \eqref{Berry:unnor-res1}]
	For $i=1,\dots,n$ and $j=1,\ldots,m$, define
	\begin{align*}
		Z_i &= (Z_{i1}, \ldots, Z_{im})^\top, \quad S_n = \frac{1}{\sqrt{n}} \sum_{i=1}^n Z_i, 
		\quad \Sigma_{f,\e} = \var_\P (Z_1), \\
		Z_{ij} &= \sqrt{\e^d} \left( \frac{1}{\e^d} K \left( \frac{\|X_i - x_j\|_{\R^D}}{\e} \right) f(X_i)
		- \E_\P \left[ \frac{1}{\e^d} K \left( \frac{\|X - x_j\|_{\R^D}}{\e} \right) f(X) \right] \right).
	\end{align*}
	Note that $S_n = \mathbf{Z}_{n,\e,f}$ is our quantity of interest, and that $Z_i$, $i=1,\dots,n$, are mean zero, i.i.d. random vectors in $\R^m$. For $A \in \cvx(\R^m)$, write
	\begin{equation}
		\P( S_n \in A) - \Phi_{m,0,\Sigma_f} (A) = \underbrace{\P( S_n \in A) - \Phi_{m,0,\Sigma_{f,\e}} (A)}_{\cE_1} + \underbrace{\Phi_{m,0,\Sigma_{f,\e}} (A) - \Phi_{m,0,\Sigma_{f}} (A)}_{\cE_2}.
		\label{Berry:unnor-eq4}
	\end{equation}
	
	\begin{myenumerate}
		\myitem{Bounding $\cE_1$}
		Applying the Berry-Esseen theorem gives
		\begin{equation*}
			|\cE_1| \leq \frac{C \, \E\| \Sigma^{-1/2}_{f,\e} Z_1\|^3_{\R^m}}{\sqrt{n}},
		\end{equation*}
		where $C$ is a constant depending only on $m$. 
		To bound the right-hand, using Lemmas~\ref{smoothing:aux:quadratic} and \ref{Berry:unnor:aux1}, and similar arguments as in the proof of Theorem~\ref{unif:bias}, it is straightforward to show that for all sufficiently small $\e$,
		\begin{align*}
			\E_\P Z_{ij}^2 &= \frac{\rho(x_j) f^2(x_j)}{(4\pi)^{d/2}} + O(\e^{2\wedge d}), \\
			\E_\P Z_{ij} Z_{ij'} &= O(\e^d) \quad\text{if $j\ne j'$}, \\
			\E_\P \left| Z_{ij} \right|^3 &\le \frac{8 \Theta_d \rho(x_j) \left|f(x_j)\right|^3
			 + O(\e^2) }{\sqrt{\e^d}},
		\end{align*}
		where $\Theta_d = \int_{\R^d} K^3 (\|z\|_{\R^d}) dz $.
		These yield
		\begin{align*}
			\|\Sigma_{f,\e} - \Sigma_f \|_\F = O(\e^{2\wedge d}), \quad
			\E \|Z_1\|^3_{\R^m} = O \Big(\frac{1}{\sqrt{\e^d}}\Big).
		\end{align*}
		Since the Frobenius norm dominates the operator norm and 
		$$ \| \Sigma_f^{-1} \|_\op = \max_{j=1,\ldots,m} \frac{(4 \pi)^{d/2}}{\rho(x_j) f^2(x_j)} >0, $$
		applying Lemma~\ref{Berry:unnor:matconv} yields
		\begin{equation}
			\| \Sigma_f^{-1/2} \Sigma_{f,\e} \Sigma_f^{-1/2} - I_m \|_\op = O(\e^{2\wedge d}),
			\quad \|\Sigma_{f,\e}^{-1/2}\|_\op = O(1).
			\label{Berry:unnor-eq1}
		\end{equation}
		
		Therefore, there exist $\e_1, C_1>0$ such that $\e <\e_1$ implies
		\begin{equation}
			|\cE_1| \leq \frac{C_1}{\sqrt{n\e^d}}.
			\label{Berry:unnor-eq2}
		\end{equation}
		
		\myitem{Bounding $\cE_2$}
		Applying Lemma~\ref{Berry:unnor:compare} gives
		\begin{equation*}
			|\cE_2| \leq \frac{1}{2} \|\Sigma_f^{-1/2} \Sigma_{f,\e} \Sigma_f^{-1/2}-I_m \|_\F
			\leq \frac{\sqrt{m}}{2} \|\Sigma_f^{-1/2} \Sigma_{f,\e} \Sigma_f^{-1/2}-I_m \|_\op.
		\end{equation*}		
		provided that $\|\Sigma_f^{-1/2} \Sigma_{f,\e} \Sigma_f^{-1/2}-I_m\|_\op$ is sufficiently small.
		
		Therefore, using \eqref{Berry:unnor-eq1}, we can find $\e_2, C_2>0$ such that $\e <\e_2$ implies
		\begin{equation}
			|\cE_2| \leq C_2 (\e^{2 \wedge d}).
			\label{Berry:unnor-eq3}
		\end{equation}
	\end{myenumerate}
	
	Combining \eqref{Berry:unnor-eq4}, \eqref{Berry:unnor-eq2} and \eqref{Berry:unnor-eq3}, completes the proof of (\ref{Berry:unnor-res1}).
	
	{\em Proof of \eqref{Berry:unnor-res2}.} By letting $\bfB_{n,\e,f} = \sqrt{n\e^d} \big(\T_\e[f](x_i) - \rho(x_i)f(x_i)\big)_{i=1}^m$ be the bias vector, we have
	$ \tbfZ_{n,\e,f} = \bfZ_{n,\e,f} + \bfB_{n,\e,f}. $
	From Theorem~\ref{unif:bias}, there exist $\e_0, C_0$ such that $\e < \e_0$ implies
	$$ \|\bfB_{n,\e,f}\|_{\R^m} < \sqrt{n\e^d} \cdot C_0 \e^2 = C_0 \sqrt{n\e^{d+4}}. $$
	Thus, by applying Lemma~\ref{Berry:unnor:error} with $Z = \bfZ_{n,\e,f}$, $W = \bfB_{n,\e,f}$ and $\delta = C_0 \sqrt{n\e^{d+4}}$ and using \eqref{Berry:unnor-res1}, we complete the proof.
\end{proof}

\subsection{Proof of Theorem \ref{Berry:nor}}

Since \eqref{Berry:nor-res2} can be derived similarly to the proof of \eqref{Berry:unnor-res2} in Theorem \ref{Berry:unnor}, we only prove \eqref{Berry:nor-res1}.

\begin{proof}[Proof of \eqref{Berry:nor-res1}]
	Using the expansion of $\TT_{n,\e} [f] (x) - \TT_{\e} [f](x)$ in the proof of \eqref{unif:sto-res2} in Theorem~\ref{unif:sto}, we can write
	\begin{equation*}
	\overline{\bfZ}_{n,\e,f} = \sqrt{n \e^{d-2}} 
	\big( \TT_{n,\e} [f] (x_j) - \TT_{\e} [f] (x_j) \big)_{j=1}^m
	= S_n + R_n S_n,
	\end{equation*}
	where for $i=1,\ldots,n$, and $j=1,\ldots,m$,
	\begin{align*}
	Z_i &= (Z_{i1}, \dots, Z_{im})^\top, \quad S_n = \frac{1}{\sqrt{n}} \sum_{i=1}^n Z_i, \quad \overline{\Sigma}_{f,\e} = \var_\P (Z_1), \\
	Z_{ij} &= \frac{1}{\T_\e [1](x_j)} \left( V_{ij} - \frac{\TT_\e[f](x_j)-f(x_j)}{\e} U_{ij} \right), \\
	U_{ij} &= \sqrt{\e^d} \left( \frac{1}{\e^d} K \left( \frac{\|X_i-x_j\|_{\R^D}}{\e} \right)
	- \E_\P \left[ \frac{1}{\e^d} K \left( \frac{\|X - x_j\|_{\R^D}}{\e} \right) \right] \right), \\
	V_{ij} &= \sqrt{\e^d} \left( \frac{1}{\e^d} K \left( \frac{\|X_i - x_j\|_{\R^D}}{\e} \right) \frac{f(X_i)-f(x_j)}{\e}
	- \E_\P \left[ \frac{1}{\e^d} K \left( \frac{\|X - x_j\|_{\R^D}}{\e} \right) \frac{f(X)-f(x_j)}{\e} \right] \right), \\
	R_n &= \diag \left( \frac{1}{\T_{n,\e}[1](x_j)} - \frac{1}{\T_{\e}[1](x_j)} \right)_{j=1}^m. 
	\end{align*}
	
	Applying Lemma~\ref{Berry:unnor:error} gives for $\delta>0$,
	\begin{align}
	\sup_{A \in \cvx(\R^m)} \Big| \P( \bfZ_{n,\e,f} \in A) - \Phi_{m,0,\overline{\Sigma}_f} (A) \Big|
	&\leq \sup_{A \in \cvx(\R^m)} \Big| \P( S_n \in A) - \Phi_{m,0,\overline{\Sigma}_f} (A) \Big|
	\notag\\ &\quad + C \delta \sqrt{\|\overline{\Sigma}_f^{-1}\|_\F} + \P(\|R_n S_n\|_{\R^m} >\delta).
	\label{Berry:nor:eq1}
	\end{align}
	We refer to the first term on the right-hand side as the main term and to the sum of the latter two terms as the remainder term, and bound each separately.
	
	\begin{myenumerate}
		\myitem{Bounding the main term}
		Since the derivation is quite similar to the proof \eqref{Berry:unnor-res1} in Theorem~\ref{Berry:unnor}, we only present the key steps.
		Using Theorem~\ref{unif:bias}, Lemma~\ref{unif:sto:var} and their proofs, it can be shown that for all sufficiently small $\e$,
		\begin{gather*}
			\E_\P U_{ij} = 0, \quad 
			\E_\P V_{ij} = 0, \\	
			\frac{1}{\T_\e [1](x_j)} = \frac{1}{\rho(x_j)} + O(\e^2), \quad
			\frac{\TT_\e[f](x_j)-f(x_j)}{\e} = O(\e), \\
			\E_\P U_{ij}^2 = \frac{\rho(x_j)}{(4\pi)^{d/2}} + O(\e), \quad
			\E_\P V_{ij}^2 = \frac{\rho(x_j) \| \nabla_\M f(x_j)\|^2}{2(4\pi)^{d/2}} + O(\e^{2\wedge d}), \\
			\E_\P U_{i1} U_{i2} = O(\e^d), \quad
			\E_\P V_{i1} V_{i2} = O(\e^d), \\
			\E_\P U_{i1} V_{i1} = O \left( \e \right), \quad
			\E_\P U_{i1} V_{i2} = O(\e^d), \\
			\E_\P \left| U_{ij} \right|^3 \leq \frac{8 \Theta_d \rho(x_j) + O(\e^{2\wedge d})}{\sqrt{\e^d}}, \quad
			\E_\P \left| V_{ij} \right|^3 \leq \frac{8 \Theta_d' \rho(x_j) \left\|\nabla_\M f(x_j)\right\|^3 + O(\e)}{\sqrt{\e^d}},			
		\end{gather*}
		where $\Theta_d = \int_{\R^d} K^3 (\|z\|_{\R^d}) dz $ and $\Theta_d' = \int_{\R^d} K^3(\|z\|_{\R^d})\|z\|_{\R^d}^3 dz$.	
		These yield
		\begin{align*}
			\|\Sigma_{f,\e} - \overline{\Sigma}_f \|_\F = O(\e), \quad
			\E \|Z_1\|^3_{\R^m} = O \Big(\frac{1}{\sqrt{\e^d}}\Big).
		\end{align*}
		
		Thus, there exist $\e_0, C_1, C_2 >0$ such that $\e < \e_0$ implies
		\begin{equation}
			\sup_{A \in \cvx(\R^m)} \Big| \P( S_n \in A) - \Phi_{m,0,\overline{\Sigma}_f} (A) \Big| \leq \frac{C_1}{\sqrt{n\e^d}} + C_2 \e.
			\label{Berry:nor:eq2}
		\end{equation}
		
		\myitem{Bounding the remainder term}
		The goal of this part is to bound 
		\begin{equation}
			C \delta \sqrt{\|\overline{\Sigma}_f^{-1}\|_\F} + \P(\|R_n S_n\|_{\R^m} >\delta)
			\label{Berry:nor:eq10}
		\end{equation}
		by choosing appropriate $\delta>0$ to balance the two terms. To this end, first, using \eqref{Berry:nor:eq2} and the tail bound for the chi-squared distribution, it is straightforward that
		\begin{equation}
			\P( \|S_n\|_{\R^m} > t) \leq \frac{C_1}{\sqrt{n\e^d}} + C_2 \e + \exp\left( -\frac{t^2}{8 \|\overline{\Sigma}_f\|_\op} \right),
			\label{Berry:nor:eq4}
		\end{equation}
		provided $\e < \e_0$ and $t \ge 2 \sqrt{m \|\overline{\Sigma}_f\|_\op}$. 
		Next, for each $j=1,\ldots,m$, the $j$th diagonal component of $R_n$ can be written as
		\begin{equation}
			-\frac{\T_{n,\e}[1](x_j) - \T_\e[1](x_j)}{\T_{n,\e}[1](x_j) \T_\e[1](x_j)}
			\label{Berry:nor:eq5}
		\end{equation}
		
		On the other hand, $\T_{n,\e}[1](x_j)-\T_{\e}[1](x_j)$ can be expressed by $U_{ij}$ as
		\begin{equation}
			\sqrt{n\e^d} ( \T_{n,\e}[1](x_j)-\T_{\e}[1](x_j)) = \frac{1}{\sqrt{n}} \sum_{i=1}^n U_{ij}.
			\label{Berry:nor:eq6}
		\end{equation}
		Then, using the arguments in bounding the main term and Theorem~\ref{unif:bias}, we have
		$$ \var_\P \T_{n,\e}[1](x_j) = O\Big(\frac{1}{n\e^d}\Big), \quad \T_{\e}[1](x_j) = \rho(x_j) + O(\e^2). $$
		These imply the denominator of \eqref{Berry:nor:eq5} can be properly bounded with high probability uniformly over $j=1,\dots,m$.
		Specifically, using Chebyshev's inequality, we have for all sufficiently small $\e$ and sufficiently large $n\e^d$ that
		\begin{equation}
			\P \left( \max_{j=1,\dots,m} \left| \frac{1}{\T_{n,\e}[1](x_j) \T_\e[1](x_j)} \right| > C_3 \right) \leq \frac{C_4}{n\e^d},
			\label{Berry:nor:eq7}
		\end{equation}
		for some $C_1',C_2'>0$.
		Therefore, under the same conditions, by combining \eqref{Berry:nor:eq5}, \eqref{Berry:nor:eq6}, and \eqref{Berry:nor:eq7}, we can bound $R_n$ as follows:
		\begin{equation}
			\P( \|R_n\|_\op > t ) \leq \P\left( \frac{C_3 \|S_n'\|_{\R^m}}{\sqrt{n\e^d}} > t\right) + \frac{C_4}{n\e^d}, \quad \text{for all $t>0$,}
			\label{Berry:nor:eq8}
		\end{equation}
		where $S_n' = \frac{1}{\sqrt{n}} \sum_{i=1}^n (U_{ij})_{j=1}^m.$
		
		Similar to $S_n$, we can apply the Berry-Esseen theorem to $S_n'$, and this yields
		\begin{equation}
			\P( \|S_n'\|_{\R^m} > t) \leq \frac{C_5}{\sqrt{n\e^d}} + C_6 \e + \exp\left( -\frac{t^2}{8 \|\Sigma_\rho\|_\op} \right), \quad\text{where } \Sigma_\rho = \diag\left(\frac{\rho(x_j)}{(4\pi)^{d/2}}\right)_{j=1}^m,
			\label{Berry:nor:eq9}
		\end{equation}
		provided $\e < \e_1$ and $t \ge 2 \sqrt{m \|\Sigma_\rho\|_\op}$, for some $\e_1,C_5,C_6>0$. 
		Now using the inequality 
		\begin{align*}
			\sqrt{n\e^d} \|R_n S_n\|_{\R^m} \leq \frac{n\e^d \|R_n\|^2_\op}{2} + \frac{\|S_n\|^2_{R^m}}{2},
		\end{align*}
		and combining \eqref{Berry:nor:eq4}, \eqref{Berry:nor:eq8} and \eqref{Berry:nor:eq9}, we obtain
		\begin{align}
			\P( \|R_n S_n\|_{\R^m} > \delta) 
			&= \P( \sqrt{n\e^d} \|R_n S_n\|_{\R^m} > \sqrt{n\e^d} \delta)
			\notag\\&\leq \P ( n\e^d \|R_n\|_\op^2 > \sqrt{n\e^d} \delta) + \P ( \|S_n\|_\op^2 > \sqrt{n\e^d} \delta)
			\notag\\&\leq \frac{C_7}{\sqrt{n\e^d}} + C_8 \e + \exp(-C_9 \sqrt{n\e^d} \delta),
		\end{align}
		provided $\e<\e_2$ and $\sqrt{n\e^d} \delta >C_{10}$, for some $\e_2, C_7,C_8,C_9,C_{10}>0$.
		
		Thus, under the same conditions, \eqref{Berry:nor:eq10} is bounded by
		$$ C \delta \sqrt{\|\overline{\Sigma}_f^{-1}\|_\F} + \frac{C_7}{\sqrt{n\e^d}} + C_8 \e + \exp(-C_9 \sqrt{n\e^d} \delta).$$
		By choosing optimal $\delta>0$, we have for $\delta = \log(n\e^d)/\sqrt{n\e^d}$,
		\begin{equation}
			C \delta \sqrt{\|\overline{\Sigma}_f^{-1}\|_\F} + \P(\|R_n S_n\|_{\R^m} >\delta)
			\leq \frac{C_{11} \log(n\e^d)}{\sqrt{n\e^d}} + C_8 \e.
			\label{Berry:nor:eq3}
		\end{equation}
	\end{myenumerate}
	
	Finally, by combining \eqref{Berry:nor:eq1}, \eqref{Berry:nor:eq2} and \eqref{Berry:nor:eq3} and adjusting the constants appropriately, we complete the proof.
\end{proof}

\subsection{Proof of Theorem \ref{Berry:norcrit}}

The following lemma helps to derive the asymptotic variances at critical points.

\begin{lemma} \label{Berry:norcrit:var}
	Suppose Assumptions \ref{smoothing:assume:manifold}, \ref{smoothing:assume:density}, \ref{smoothing:assume:kernel} hold
	and $f \in C^3(\M)$. If $x\in\M$ is a critical point of $f$, then there exist $\e_0, C_0>0$ such that if $\e < \e_0$, we have
	\begin{gather}
	\left| \frac{\TT_\e[f](x)-f(x)}{\e^2} + \frac{1}{2} \Delta_\M f(x) \right| \leq C_0 \e, \label{Berry:norcrit:var-res1} \\
	\left| \frac{1}{\e^d} \int_\M K^2 \left( \frac{\|u-x\|_{\R^D}}{\e} \right) \frac{f(u)-f(x)}{\e^2} d\P(u) 
	+ \frac{\rho(x) \Delta_\M f(x) }{4(4\pi)^{d/2}} \right|
	\leq C_0 \e, \label{Berry:norcrit:var-res2} \\
	\left| \frac{1}{\e^d} \int_\M K^2 \left( \frac{\|u-x\|_{\R^D}}{\e} \right) \left(\frac{f(u)-f(x)}{\e^2} \right)^2 d\P(u) 
	- \frac{\rho(x) (2\| \nabla^2_{\M} f (x) \|^2 + \left|\Delta_\M f (x)\right|^2)}{16(4\pi)^{d/2}} \right|
	\leq C_0 \e. \label{Berry:norcrit:var-res3}
	\end{gather}
	Here, the constants $\e_0$ and $C_0$ are $\M$-dependent, and they also depend on the $C^3$-norm $\rho$ and $f$.
\end{lemma}

\begin{proof}[Proof of Lemma \ref{Berry:norcrit:var}]
	From Theorem~\ref{unif:bias} and its proof, for all sufficiently small $\e$, we have
	\begin{align*}
		\TT_\e[f](x)-f(x) &= -\frac{\e^2}{2} \Delta_{\M,2} f(x) + O(\e^3), \\
		\Delta_{\M,2} f (x) &= \frac{\rho^2(x) \Delta_\M f (x) - 2 \rho(x) \langle \nabla_\M \rho(x), \nabla_\M f(x) \rangle_{T_x \M} }{\rho^2 (x)}.
	\end{align*}
	Since $x$ is a critical point of $f$, we have $\Delta_{\M,2} f (x) = \Delta_\M f (x)$.
	Thus, we obtain \eqref{Berry:norcrit:var-res1}.
	
	To show \eqref{Berry:norcrit:var-res2}, similar to the proof Lemma~\ref{unif:sto:var}, for all sufficiently small $\e$, we have
	\begin{align*}
		&\quad \frac{1}{\e^d} \int_\M K^2 \left( \frac{\|u-x\|_{\R^D}}{\e} \right) \frac{f(u)-f(x)}{\e^2} d\P(u)
		\\&= \int_{\R^d} K^2(\|z\|_{\R^d}) \Big( \e^{-1} z^\top \nabla_{\R^d} \tilde{f}_x (0) + \frac{1}{2} z^\top \big[ \nabla_{\R^d}^2 \tilde{f}_x(0) \big] z + \e O(\|z\|_{\R^d}^3) \Big)
		\\&\quad \times \Big( \tilde{\rho}_x (0) + \e z^\top \nabla \tilde{\rho}_x (0) + \e^2 O(\|z\|_{\R^d}^2) \Big) dz + O(\e^2)
		\\&= \frac{\rho(x)}{2} \int_{\R^d} K^2(\|z\|_{\R^d}) \left( z^\top \big[\nabla_{\R^d}^2 \tilde{f}_x (0) \big] z\right) dz + O(\e).
	\end{align*}
	Here, we use that $x$ is a critical point, meaning that $\nabla_{\R^d} \tilde{f}_x (0)=0.$
	
	Now using Lemma \ref{smoothing:aux:quadratic}, the integral in the leading term can be simplified as
	\begin{align*}
		\int_{\R^d} K^2(\|z\|_{\R^d}) \left( z^\top \big[\nabla_{\R^d}^2 \tilde{f}_x (0) \big] z\right) dz
		= \frac{\tr \nabla_{\R^d}^2 \tilde{f}_x (0) }{2(4\pi)^{d/2}}
		= \frac{-\Delta_\M f(x) }{2(4\pi)^{d/2}}.
	\end{align*}
	Therefore, we obtain \eqref{Berry:norcrit:var-res2}.
	
	Similarly, we also have
	\begin{align*}
		\frac{1}{\e^d} \int_\M K^2 \left( \frac{\|u-x\|_{\R^D}}{\e} \right) \left(\frac{f(u)-f(x)}{\e^2} \right)^2 d\P(u)
		= \frac{\rho(x)}{4} \int_{\R^d} K^2(\|z\|_{\R^d}) \left( z^\top \big[\nabla_{\R^d}^2 \tilde{f}_x (0) \big] z\right)^2 dz + O(\e),
	\end{align*}
	and
	\begin{align*}
		\int_{\R^d} K^2(\|z\|_{\R^d}) \left( z^\top \big[\nabla_{\R^d}^2 \tilde{f}_x (0) \big] z\right)^2 dz
		&= \frac{2 \| \nabla_{\R^d}^2 \tilde{f}_x (0)\|_\F^2 + |\tr \nabla_{\R^d}^2 \tilde{f}_x (0)|^2 }{4(4\pi)^{d/2}}
		\\&= \frac{2 \|\nabla_\M^2 f (x)\|^2 + |\Delta_\M f(x)|^2 }{4(4\pi)^{d/2}}.
	\end{align*}
	These imply \eqref{Berry:norcrit:var-res3}, and we complete the proof.
\end{proof}

Since the proof of the corollary is analogous to that of Theorem~\ref{Berry:nor}, we only highlight the differences.

\begin{proof}[Proof of Theorem \ref{Berry:norcrit}]
	Assume $m=2$. Similar to the proof \eqref{Berry:nor-res1} in Theorem~\ref{Berry:nor}, we have
	\begin{equation*}
	\overline{\bfZ}_{n,\e,f} = \sqrt{n \e^{d-4}} 
	\begin{pmatrix} 
		\TT_{n,\e} [f] (x_1) - \TT_{\e} [f] (x_1) \\
		\TT_{n,\e} [f] (x_2) - \TT_{\e} [f] (x_2)
	\end{pmatrix} = S_n + E_n S_n,
	\end{equation*}
	where for $i=1,\ldots,n$ and $j=1,2$,
	\begin{align*}
	Z_i &= (Z_{i1}, Z_{i2})^\top, \quad S_n = \frac{1}{\sqrt{n}} \sum_{i=1}^n Z_i, \quad \Sigma'_{f,\e} = \var_\P (Z_1), \\
	Z_{ij} &= \frac{1}{\T_\e [1](x_j)} \left( V_{ij} - \frac{\TT_\e[f](x_j)-f(x_j)}{\e^2} U_{ij} \right), \\
	U_{ij} &= \sqrt{\e^d} \left( \frac{1}{\e^d} K \left( \frac{\|X_i-x_j\|_{\R^D}}{\e} \right)
	- \E_\P \left[ \frac{1}{\e^d} K \left( \frac{\|X - x_j\|_{\R^D}}{\e} \right) \right] \right), \\
	V_{ij} &= \sqrt{\e^d} \left( \frac{1}{\e^d} K \left( \frac{\|X_i - x_j\|_{\R^D}}{\e} \right) \frac{f(X_i)-f(x_j)}{\e^2}
	- \E_\P \left[ \frac{1}{\e^d} K \left( \frac{\|X - x_j\|_{\R^D}}{\e} \right) \frac{f(X)-f(x_j)}{\e^2} \right] \right), \\
	R_n &= \diag \left( \frac{1}{\T_{n,\e}[1](x_1)} - \frac{1}{\T_{\e}[1](x_1)}, \frac{1}{\T_{n,\e}[1](x_2)} - \frac{1}{\T_{\e}[1](x_2)} \right).
	\end{align*}
	
	Using Theorem~\ref{unif:bias}, Lemmas~\ref{unif:sto:var} and \ref{Berry:norcrit:var}, and their proofs, it can be shown that for all sufficiently small $\e$,
	\begin{gather*}
	\E_\P U_{ij} = 0, \quad 
	\E_\P V_{ij} = 0, \\	
	\frac{1}{\T_\e [1](x_j)} = \frac{1}{\rho(x_j)} + O(\e^2), \quad
	\frac{\TT_\e[f](x_j)-f(x_j)}{\e^2} = -\frac{1}{2} \Delta_{\M} f(x_j) + O(\e), \\
	\E_\P U_{ij}^2 = \frac{\rho(x_j)}{(4\pi)^{d/2}} + O(\e), \quad
	\E_\P V_{ij}^2 = \frac{\rho(x) (2\| \nabla^2_{\M} f (x_j) \|^2 + \left|\Delta_\M f (x_j)\right|^2)}{16(4\pi)^{d/2}} + O(\e^{2\wedge d}), \\
	\E_\P U_{i1} U_{i2} = O(\e^d), \quad
	\E_\P V_{i1} V_{i2} = O(\e^d), \\
	\E_\P U_{i1} V_{i1} = -\frac{\rho(x_j) \Delta_\M f(x_j) }{4(4\pi)^{d/2}} + O \left( \e \right), \quad
	\E_\P U_{i1} V_{i2} = O(\e^d), \\
	\E_\P \left| U_{ij} \right|^3 \leq \frac{8 \Theta_d \rho(x_j) + O(\e^{2\wedge d})}{\sqrt{\e^d}}, \quad
	\E_\P \left| V_{ij} \right|^3 \leq \frac{4 \Theta_d'' \rho(x_j) \left\|\nabla^2_\M f(x_j)\right\|^3 + O(\e)}{\sqrt{\e^d}},			
	\end{gather*}
	where $\Theta_d = \int_{\R^d} K^3 (\|z\|_{\R^d}) dz $ and $\Theta_d'' = \int_{\R^d} K^3(\|z\|_{\R^d})\|z\|_{\R^d}^6 dz$.
	
	We see that
	\begin{align*}
	\var_\P (Z_{ij}) 
	&= \frac{1}{\rho^2(x_j)} \bigg[ \frac{\rho(x_j) (2\| \nabla^2_{\M} f (x_j) \|^2 + \left|\Delta_\M f (x_j)\right|^2)}{16(4\pi)^{d/2}} 
	\\&\quad -2 \left(-\frac{1}{2} \Delta_{\M} f(x_j)\right) \left( -\frac{\rho(x_j)\Delta_\M f(x_j) }{4(4\pi)^{d/2}} \right)
	+ \left(-\frac{1}{2} \Delta_{\M} f(x_j)\right)^2 \frac{\rho(x_j)}{(4\pi)^{d/2}} \bigg] + O(\e)
	\\&= \frac{2\| \nabla_\M^2 f(x_j) \|_\F^2 + |\Delta_\M f(x_j)|^2}{16 \rho(x_j) (4\pi)^{d/2}} + O(\e).
	\end{align*}
	Therefore, by following the same argument as in the proof of Theorem~\ref{Berry:nor}, the proof is complete.
\end{proof}

\subsection{Proof of Theorem \ref{deriv:unnor}}
The proof is carried out separately for the gradient and the Hessian.

\subsubsection{Proof of Theorem \ref{deriv:unnor} for gradients}

We first outline the main idea of the proof.
By the definition of the gradient norm, we have
\begin{align*}
	\left\| \nabla_\M \T_{\e} [f] (x) - \nabla_\M (\rho f) (x) \right\|
	= \left\| \nabla_{\R^d} (\T_{\e} [f] \circ \Exp_x) (0) - \nabla_{\R^d} \tilde{(\rho f)}_x (0) \right\|_{\R^d}.
\end{align*}
To simplify $\nabla_{\R^d} (\T_{\e} [f] \circ \Exp_x) (0)$,
if we view $\T_{\e} [f]$ as a function on $\R^D$ and $\Exp_x$ as a map from $\R^d$ to $\R^D$, 
then the chain rule gives
$$ \nabla_{\R^d} (\T_{\e} [f] \circ \Exp_x) (0) = J(x)^\top \nabla_{\R^D} \T_{\e} [f] (x), $$
where $J(x)$ is as defined earlier.
Moreover, $\nabla_{\R^D} \T_{\e} [f] (x)$ can be expressed explicitly since $x$ appears in the kernel only. 
Specifically, we have
\begin{align*}
	\nabla_{\R^D} \T_{\e} [f]
	&= \frac{1}{\e^d} \int_\M \left[ \nabla_{\R^D} K \left( \frac{\|u-\cdot\|_{\R^D}}{\e} \right) \right] f(u) \dP (u)
	\\&= \frac{1}{\e^{d+1}} \int_\M K \left( \frac{\|u-\cdot\|_{\R^D}}{\e} \right) \frac{u-\cdot}{\e} f(u) \dP (u),
\end{align*}
and hence,
\begin{equation*}
	\nabla_{\R^D} \T_{\e} [f] (x) = \frac{1}{\e^{d+1}} \int_\M K \left( \frac{\|u-x\|_{\R^D}}{\e} \right) \frac{u-x}{\e} f(u) \dP (u),
\end{equation*}
where the integral is understood componentwise.
Therefore, we obtain 
\begin{align*}
	\nabla_{\R^d} (\T_{\e} [f] \circ \Exp_x) (0)
	&= \frac{1}{\e^{d+1}} \int_\M K \left( \frac{\|u-x\|_{\R^D}}{\e} \right) J(x)^\top \frac{u-x}{\e} f(u) \dP (u).
\end{align*}
Our strategy is to compare this expression with $\nabla_{\R^d} \tilde{(\rho f)}_x (0)$, which yields the desired convergence result.

The idea of viewing smoothed functions as defined on the ambient space will also be used for the gradients of other smoothed functions, as well as their Hessians.

\begin{proof}[Proof of part (a) of Theorem \ref{deriv:unnor} for gradients]
	Using the Taylor expansion of the exponential map stated in (c) of Proposition~\ref{smoothing:aux:expansion}, for $u= \Exp_x (\e z)$ with sufficiently small $\e$ and $z \in \R^d$, we have
	\begin{align*}
		J(x)^\top \Big(\frac{u-x}{\e}\Big) 
		&= J(x)^\top \frac{\e J(x) z + \frac{\e^2}{2} \sff_x(z,z) + \e^3 \alpha_{x,3} (z) + \e^4 O(\|z\|_{\R^d}^4)}{\e}
		\\&= J(x)^\top J(x) z + \frac{\e}{2} J(x)^\top \sff_x(z,z) + \e^2 J(x)^\top \alpha_{x,3} (z) + \e^3 O(\|z\|_{\R^d}^4).
		\end{align*}
	Noting that the columns of $J(x)$ are orthonormal and that they are perpendicular to $\sff_x(z,z)$, we simplify the above to
	\begin{equation*} \label{derivative:grad:unnormal-eq1}
		J(x)^\top \Big(\frac{u-x}{\e}\Big) = z + \e^2 J(x)^\top \alpha_{x,3} (z) + \e^3 O(\|z\|_{\R^d}^4).
	\end{equation*}
	
	Using this and arguments analogous to those in the proof of \eqref{unif:bias-res1} in Theorem~\ref{unif:bias}, for all sufficiently small $\e$, we obtain
	\begin{align*}
		\nabla_{\R^d} (\T_{\e} [f] \circ \Exp_x) (0)
		&= \frac{1}{\e^{d+1}} \int_\M K \left( \frac{\|u-x\|_{\R^D}}{\e} \right) J(x)^\top \Big(\frac{u-x}{\e}\Big) f(u) d\P(u)
		\\&= \frac{1}{\e} \int_{\R^d} K \left(\|z\|_{\R^d}\right) \cA_1 \cA_2 \cA_3 \cA_4 dz + O(\e^2),
	\end{align*}
	where
	\begin{align*}
		\cA_1 &= 1 + \frac{\e^2}{24} \sff_x (z,z)^\top \sff_x (z,z) + \e^3 O(\|z\|_{\R^d}^4), \\
		\cA_2 &= z + \e^2 R_3 (z) + \e^3 O(\|z\|_{\R^d}^4), \\
		\cA_3 &= (\widetilde{\rho f})_x (0) + \e z^\top \nabla_{\R^d} (\widetilde{\rho f})_x (0) 
		+ \frac{\e^2}{2} z^\top [\nabla_{\R^d}^2 (\widetilde{\rho f})_x (0)] z + \e^3 O(\|z\|_{\R^d}^3),\\
		\cA_4 &= 1 -\frac{\e^2}{6} z^\top \Ric_x z + \e^3 O(\|z\|_{\R^d}^3).
	\end{align*}
	By symmetry of the kernel, all odd-order terms in $z$ vanish after integration.
	Therefore, the above expression simplifies to
	\begin{align*}
		\nabla_{\R^d} (\T_{\e} [f] \circ \Exp_x) (0)
		&= \frac{1}{\e} \cdot \e \left( \int_{\R^d} K (\|z\|_{\R^d}) z z^\top dz \right) \nabla_{\R^d}
		(\widetilde{\rho f})_x (0) + O(\e^2)
		\\&= \nabla_{\R^d} (\widetilde{\rho f})_x (0) + O(\e^2).
	\end{align*}
	This completes the proof.
\end{proof}

\begin{proof}[Proof of part (b) of Theorem \ref{deriv:unnor} for gradients]
	From
	$$ \nabla_{\R^D} \T_{n,\e} [f] (x) = \frac{1}{\e^{d+1}} \sum_{i=1}^{n} K\left(\frac{\|X_i-x\|_{\R^D}}{\e}\right) \frac{X_i-x}{\e} f(X_i), $$
	we can write
	\begin{align*}
		\sup_{x \in \M} \left\| \nabla_\M \T_{n,\e} [f] (x) - \nabla_\M \T_{\e} [f] (x)  \right\|
		= \frac{1}{\e^{d+1}} \sup_{g \in \msF_\e} 
		\left\| \frac{1}{n} \sum_{i=1}^n g(X_i) - \E_\P g(X) \right\|_{\R^d}, 
	\end{align*}
	where 
	$\msF_\e = \set{ g_{x} (\cdot) = (g_{x,j}(\cdot))_{j=1}^d = K\left( \frac{\|\cdot - x\|_{\R^D}}{\e} \right) J(x)^\top \left(\frac{\cdot - x}{\e}\right) f(\cdot) \in \R^d : x \in \M }. $ 
	To complete the proof, we apply arguments analogous to those in the proof of Theorem~\ref{unif:sto}.
	
	First, we have
	\begin{equation*}
		\sup_{x \in \M} \max_{j=1,\dots,d} \|g_{x,j}\|_\infty
		\leq \sup_{g \in \msF_\e} \|g\|_\infty
		\leq B \coloneqq \frac{\|J\|_\infty \|f\|_\infty \diam(\M)}{\e},
	\end{equation*}
	where $\diam(\M)$ denotes the diameter of $\M$. Next, for all sufficiently small $\e$,
	\begin{align*}
		\E_\P g_{x,j}^2 (X)
		&= \E_\P \left[K^2\left(\frac{\|X-x\|_{\R^D}}{\e}\right) J_j (x)^\top \left(\frac{X-x}{\e}\right) \left(\frac{X-x}{\e}\right)^\top J_j(x) f^2(X)\right]	\\
		&\leq \e^d \|\rho\|_\infty \|f\|^2_\infty \left(\int_{\R^d} K(\|z\|_{\R^d}) z_j^2 dz + O(\e^2) \right) \\
		&\leq \e^d \|\rho\|_\infty \|f\|^2_\infty \left(1 + O(\e^2) \right).
	\end{align*}
	So, we can find $\e_2, C_2 >0$ such that we have for $\e <\e_2$,
	\begin{align*}
		\sup_{x \in \M} \max_{j=1,\dots,d} \E_\P g^2_{x,j}(X) \leq \sigma^2 \coloneqq C_2 \e^d.
	\end{align*}
	
	Lastly, for all sufficiently small $\e$, we obtain
	\begin{align*}
		\| g_x (u) - g_{x'} (u) \|_{\R^d}
		&\leq \left( \frac{\|J\|_\infty + \|\nabla_{\R^D} J\|_\infty \diam(\M)}{\e} + \frac{\|J\|_\infty \diam(\M)}{\e^2} \right) \|f\|_\infty \|x - x'\|_{\R^D}
		\\
		&\leq \frac{2 \|J\|_\infty \|f\|_\infty  \diam(\M)}{\e^2}  \dist_\M (x,x').
	\end{align*}
	Hence, there exists $\e_2''>0$ such that whenever $\e < \e''_2$ and $\eta < B$, we have
	\begin{align*}
		\sup_\Q \mathscr{N} \left( \msF_\e, L^2(\Q), \eta \right)
		\leq & \left( \frac{A_\M \frac{2 \|f\|_\infty \|J\|_\infty \diam(\M)}{\e^2} }{\eta} \right)^d .
	\end{align*}
	Therefore, we can find $\e_0$, $C_0 >0$ such that if $\e >\e_0$ we have
	\begin{align*}
		&\quad = \sup_{x \in \M} \left\| \nabla_{\R^d} (\T_{n,\e} [f] \circ \Exp_x) (0)
		- \nabla_{\R^d} (\T_{\e} [f] \circ \Exp_x)  (0) \right\|_{\R^d}
		\\ &\leq
		\frac{1}{\e^{d+1}}
		\cdot C_0 \left( \sqrt{  \frac{\e^d \log(1/\e)}{n} }  + \sqrt{ \frac{\e^d \log (1/\delta)}{n}}
		+ \frac{\log(1/\e)}{n \e} + \frac{\log (1/\delta) }{n \e} \right)
		\\ &=
		C_0 \left( \sqrt{ \frac{\log(1/\e)}{n\e^{d+2}} } + \sqrt{\frac{\log (1/\delta)}{n \e^{d+2}}}
		+ \frac{\log(1/\e)}{n \e^{d+2}} + \frac{\log (1/\delta) }{n \e^{d+2}} \right).
	\end{align*}
	This completes the proof.
\end{proof}

\subsubsection{Proof of Theorem \ref{deriv:unnor} for Hessians}

Hereafter, we denote by $[v]_j$ the $j$th component of a vector $v$, and by $[A]_{ij}$ the $(i,j)$-entry of a matrix $A$. By the definition of the Hessian norm, we have
\begin{align*}
	\left\| \nabla_\M^2 \T_{\e} [f] (x) - \nabla_\M^2 (\rho f) (x) \right\|
	= \left\| \nabla_{\R^d}^2 (\T_{\e} [f] \circ \Exp_x) (0) - \nabla_{\R^d}^2 \tilde{(\rho f)}_x (0) \right\|_{\R^d}.
\end{align*}
Similar to the case of gradients, if we view $\T_{\e} [f]$ as a function on $\R^D$ and $\Exp_x$ as a map from $\R^d$ to $\R^D$, 
then the chain rule gives
\begin{equation*}
	\nabla_{\R^d}^2 (\T_{\e} [f] \circ \Exp_x) (0) 
	= J(x)^\top \nabla_{\R^D}^2 \T_{\e} [f] J(x)
	+ \sum_{k=1}^D \Big[ \nabla_{\R^D} \T_\e[f] \Big]_k \nabla_{\R^d}^2 \left[ \Exp_x \right]_k (0).
\end{equation*}
Here, differentiating $\T_{\e} [f]$ with respect to $x$ yields
\begin{align*}
	\nabla_{\R^D}^2 \T_{\e} [f] (x)
	&= \frac{1}{\e^{d+2}} \int_{\M} K\left(\frac{\|u-x\|_{\R^D}}{\e}\right) \left[\left(\frac{u-x}{\e}\right)\left(\frac{u-x}{\e}\right)^\top - I_{D}\right] f(u) d\P(u), \\
	\nabla_{\R^D} \T_{\e} [f] (x)
	&= \frac{1}{\e^{d+1}} \int_{\M} K\left(\frac{\|u-x\|_{\R^D}}{\e}\right) \frac{u-x}{\e} f(u) d\P(u).
\end{align*}

Moreover, noting that $\sff_x$ is symmetric and bilinear, we can write
\begin{equation*}
	\sff_x (z,z) = \sum_{i,j=1}^d z_i z_j b_{ij}, \quad\text{where}\quad b_{ij} := \sff_x (e_i, e_j) \in (T_x \M)^\perp \subset \R^D,
\end{equation*}
with $\{e_1, \ldots, e_d\}$ denoting the standard basis of $\R^d$.
Using this and the property of the second derivative of the exponential map,
we have
$$ \left[ \nabla_{\R^d}^2 \left[ \Exp_x \right]_k (0) \right]_{ij} = [ b_{ij} ]_k. $$

\begin{proof}[Proof of part (a) of Theorem \ref{deriv:unnor} for Hessians]
	First, using arguments similar to those in the gradient case,
	for all sufficiently small $\e$, we can write
	\begin{align*}
		&\quad J(x)^\top \nabla^2_{\R^D} \T_{\e} [f] (x) J(x)
		\\&= \frac{1}{\e^{d+2}} \int_{\M} K\left(\frac{\|u-x\|_{\R^D}}{\e}\right) \left[ J(x)^\top\left(\frac{u-x}{\e}\right)\left(\frac{u-x}{\e}\right)^\top J(x) - I_{d}\right] f(u) d\P(u),
		\\&= \cA_1 + \cA_2 + \cA_3 + \cA_4 + \cA_5 + O(\e^2),
	\end{align*}
	where
	\begin{align*}
		\cA_1 &= \frac{\rho(x) f(x)}{\e^2} \int_{\R^d} K(\|z\|_{\R^d}) (zz^\top - I_d) dz, \\
		\cA_2 &= \frac{1}{2} \int_{\R^d} K(\|z\|_{\R^d}) (zz^\top-I_d) \left( z^\top \nabla^2_{\R^d}
		(\widetilde{\rho f})_x (0) z \right) dz, \\
		\cA_3 &= \frac{\rho(x) f(x)}{24}  \int_{\R^d} K(\|z\|_{\R^d}) \sff_x (z,z)^\top \sff_x (z,z)
		(z z^\top - I_d) dz, \\
		\cA_4 &= \frac{\rho(x) f(x)}{6} \int_{\R^d} K(\|z\|_{\R^d}) \left( J(x)^\top \alpha_{x,3} (z) z^\top
		+ z \alpha_{x,3} (z)^\top J(x) \right) dz, \\
		\cA_5 &= -\frac{\rho(x) f(x)}{6} \int_{\R^d} K(\|z\|_{\R^d}) (z z^\top - I_d) \left( z^\top \Ric_x z \right) z dz.
	\end{align*}
	Also, by an analogous expansion, we have
	\begin{align*}
		\sum_{k=1}^D \Big[ \nabla_{\R^D} \T_\e[f] \Big]_k \nabla_{\R^d}^2 \left[ \Exp_x \right]_k (0)
		=& \cA_6 + \cA_7 + O(\e^2)
	\end{align*}
	where
	\begin{align*}
		\cA_6 &= \frac{\rho(x) f(x)}{2} \sum_{\ell=1}^D  \left[ \int_{\R^d} K (\|z\|_{\R^d})  \sff_{x} (z,z) dz \right]_{\ell} \nabla_{\R^d}^2 \left[ \Exp_x \right]_\ell, \\
		\cA_7 &=  \sum_{\ell=1}^D \left[ J(x) \nabla_{\R^d} \widetilde{(\rho f)}_x (0) \right]_{\ell} \nabla_{\R^d}^2 \left[ \Exp_x \right]_\ell.
	\end{align*}
	Combining the above yields the expansion
	\begin{align*}
		\nabla_{\R^d}^2 (\T_\e [f] \circ \Exp_x) (0)
		&= \cA_1 + \cA_2 + \cA_3 + \cA_4 + \cA_5 + \cA_6 + \cA_7 + O(\e^2),
	\end{align*}
	and we proceed to simplify the right-hand side.
	
	\begin{myenumerate}
	\myitem{Evaluating $\cA_1$}
	Since $\int_{\R^d} K(\|z\|_{\R^d}) z z^\top dz = I_d$, we have $\cA_1 = 0.$
	
	\myitem{Evaluating $\cA_2$}
	Using Lemma~\ref{smoothing:aux:quadratic}, we obtain $\cA_2= \nabla_{\R^d}^2 (\widetilde{\rho f})_x (0).$
	
	\myitem{Evaluating $\cA_3$}
	Using $ \sff_x (z,z)^\top \sff_x (z,z) = \sum_{i,j,k,\ell=1}^d z_i z_j z_k z_\ell b_{ij}^\top b_{k\ell}, $ and the moments of $K$, it can be shown
	\begin{align*}
		\left[ \cA_3 \right]_{ij}
		&= \frac{\rho(x) f(x)}{24}  \left[ \int_{\R^d} K(\|z\|_{\R^d}) \sff_x (z,z)^\top \sff_x (z,z)  (z z^\top - I_d) dz \right]_{ij}
		\\&= \frac{ \rho(x) f(x) }{6} \sum_{k=1}^d (b_{ij}^\top b_{kk} + 2 b_{ik}^\top b_{jk}).
	\end{align*}
	
	\myitem{Evaluating $\cA_4$}
	By using (b) in Proposition \ref{smoothing:aux:expansion}, and the moments of $K$, we can have
	\begin{align*}
		\left[ \cA_4 \right]_{ij}
		&= -\frac{ \rho(x) f(x) }{6} \left[ \frac{\rho(x) f(x)}{6} \int_{\R^d} K(\|z\|_{\R^d}) \left( J(x)^\top \alpha_{x,3} (z) z^\top  + z \alpha_{x,3} (z)^\top J(x) \right) dz \right]_{ij}
		\\&= -\frac{ \rho(x) f(x)}{3}\sum_{k=1}^d (b_{ij}^\top b_{kk} + 2b_{ik}^\top b_{jk}).
	\end{align*}
	
	\myitem{Evaluating $\cA_5$}
	Directly from Lemma~\ref{smoothing:aux:quadratic}, we have
	$\cA_5 = -\frac{\rho(x)f(x)}{3} \Ric_x. $ 
	Now, in normal coordinates, the entries of the Ricci curvature tensor can be expressed as the sum of the Riemann curvature tensors:
	$$ \Ric_{x,ij} = \sum_{k=1}^d \mathrm{Rm}_{x,kijk},$$
	and since $\M$ is embedded in $\R^D$, the Gauss equation (see Theorem 8.5 of \cite{lee2018introduction}) gives
	$ \mathrm{Rm}_{x,ijk \ell} = b_{i\ell} b_{jk} - b_{ik}^\top b_{j\ell}.$ 
	Combining these, we have
	\begin{equation*}
		\left[ \cA_5 \right]_{ij}
		= -\frac{\rho(x)f(x)}{3} [\Ric_{x}]_{ij}
		= -\frac{\rho(x)f(x)}{3} \sum_{k=1}^{d} \left( b_{ij}^\top b_{kk} - b_{ik}^\top b_{jk} \right). 
	\end{equation*}
	
	\myitem{Evaluating $\cA_6$}
	Since
	$\int_{\R^d} K(\|z\|_{\R^d}) \sff_x (z,z) dz = \int_{\R^d} K(\|z\|_{\R^d}) \sum_{i,j=1}^{d} z_i z_j b_{ij} dz
	=  \sum_{k=1}^d b_{kk}, $
	we have
	\begin{align*}
		\left[ \cA_6 \right]_{ij}
		= \frac{ \rho(x) f(x)}{2} \sum_{\ell=1}^{D} \sum_{k=1}^{d} \left[ b_{kk} \right]_\ell \left[ b_{ij} \right]_\ell
		= \frac{ \rho(x) f(x)}{2} \sum_{k=1}^{d} b_{ij}^\top b_{kk}.
	\end{align*}
	
	\myitem{Evaluating $\cA_7$}
	Note that $J(x) \nabla_{\R^d} (\widetilde{\rho f})_x (0) \in T_x \M$. 
	Therefore, for all $i,j=1,\dots d$, we have
	$$ \sum_{\ell=1}^{D} \left[ J(x) \nabla_{\R^d} (\widetilde{\rho f})_x (0) \right]_{\ell} \left[ b_{ij} \right]_{\ell}
	=  b_{ij}^\top J(x) \nabla_{\R^d} (\widetilde{\rho f})_x (0) = 0,$$
	and this implies $\cA_7 = 0$.
	\end{myenumerate}
	
	From the above results, for all $i,j=1,\dots,d$, we see that
	$ \left[ \cA_3 + \cA_4 + \cA_5 +\cA_6 \right]_{ij}=0 .$
	Therefore, we have
	$$ \nabla_{\R^d}^2 (\T_\e [f] \circ \Exp_x) (0) =  \nabla_{\R^d}^2 (\widetilde{\rho f})_x (0)  + O(\e^2), $$
	and this completes the proof.
\end{proof}

\begin{proof}[Proof of part (b) of Theorem \ref{deriv:unnor} for Hessians]
	First, express the following difference as
	\begin{align*}
		\nabla_{\R^d}^2 (\T_{n,\e} [f] \circ \Exp_x) (0) - \nabla_{\R^d}^2 (\T_{\e} [f] \circ \Exp_x ) (0)
		= \cA_1 + \cA_2,
	\end{align*}
	where
	\begin{align*}
		\cA_1 &= J(x)^\top \nabla_{\R^D}^2 \T_{n,\e} [f] (x) J(x) - J(x)^\top \nabla_{\R^D}^2 \T_{\e} [f] (x) J(x), \\
		\cA_2 &=\sum_{\ell=1}^{D} \left[ \nabla_{\R^D} \T_{n,\e} [f] (x) - \nabla_{\R^D} \T_{\e} [f] (x) \right]_\ell \nabla_{\R^d}^2 [\Exp_x]_{\ell} (0).
	\end{align*}
	
	\begin{myenumerate}
		\myitem{Bounding $\cA_1$}
		From
		$$ \nabla^2_{\R^D} \T_{n,\e} [f] (x) = \frac{1}{\e^{d+2}} \sum_{i=1}^{n} K\Big(\frac{\|X_i-x\|_{\R^D}}{\e}\Big) \Big\{ \Big(\frac{X_i-x}{\e}\Big)\Big(\frac{X_i-x}{\e}\Big)^\top - I_D \Big\} f(X_i), $$
		write
		$$ \sup_{x \in M} \left\| \cA_1 \right\|_\F 
		= \frac{1}{\e^{d+2}} \sup_{g \in \msF_\e}
		\left\| \frac{1}{n} \sum_{i=1}^n g(X_i) - \E_\P g(X) \right\|_\F $$
		where 
		$$ \msF_\e = \Big\{ g_{x} (\cdot) = (g_{x,ij}(\cdot))_{i,j=1,\dots,d}
		= K\Big( \frac{\|\cdot - x\|_{\R^D}}{\e} \Big) \Big\{ J(x)^\top \Big(\frac{\cdot -x}{\e}\Big)\Big(\frac{\cdot - x}{\e}\Big)^\top J(x) - I_d \Big\} f(\cdot) : x \in \M \Big\}.
		$$        
		Then, we have
		$$ \sup_{x \in \M} \max_{i,j=1,\dots,d} \|g_{x,ij}\|_\infty
		\leq \sup_{g \in \msF_\e} \|g\|_\infty
		\leq B \coloneqq \left( \frac{\|J\|^2_\infty \diam^2(\M)}{\e^2} + d \right) \|f\|_\infty. $$
		Also, for all sufficiently small $\e$,
		\begin{align*}
			\E_\P g^2_{x,ij} (X) 
			&=  \E_\P \Big[K^2\Big(\frac{\|X-x\|_{\R^D}}{\e}\Big) \Big\{ J_{i} (x)^\top \Big(\frac{X-x}{\e}\Big)
			\Big(\frac{X-x}{\e}\Big)^\top J_{j} (x) \Big\}^2 f^2(X)\Big]
			\\ &\leq \e^d \|\rho\|_\infty \|f\|^2_\infty \left(\int_{\R^d} K(\|z\|_{\R^d}) z_i^2 z_j^2 dz + O(\e^2) \right)
			\\ &= \e^d \|\rho\|_\infty \|f\|^2_\infty \left( 3 + O(\e^2) \right)
		\end{align*}
		So, we can find $\e_1', C_1' >0$ such that whenever $\e <\e_1'$ we have
		\begin{align*}
			\sup_{x \in \M} \max_{i,j=1,\dots,d} \E_{P} g^2_{x,ij}(X)
			\leq \sigma^2
			\coloneqq C_1' \e^d
		\end{align*}
		Also, we see that for all sufficiently small $\e$,
		\begin{align*}
			\| g_x (u) - g_{x'} (u) \|_\F
			\leq \frac{2 \|J\|_\infty \|f\|_\infty \diam^2(\M) }{\e^3} \dist_\M (x,x').
		\end{align*}
		Thus, there exists $\e''_2 > 0$ such that whenever $\e < \e_2''$ and $\eta < B$, 
		\begin{align*}
			\sup_\Q \mathscr{N} \left( \msF_\e, L^2(\Q), \eta \right)
			\leq&
			\left( \frac{A_{\M} \frac{2 \|J\|_\infty \|f\|_\infty \diam^2(\M) }{\e^3} }{\eta} \right)^{d/2}
		\end{align*}
		Therefore, by applying Proposition \ref{smoothing:aux:VC} component-wisely,
		we can find $\e_1$, $C_1>0$ such that if $\e < \e_1$, we have
		\begin{align*}
			\sup_{x \in \M} \left\| \cA_1 \right\|_\F
			&\leq \frac{1}{\e^{d+2}}
			\cdot C_1 \left( \sqrt{  \frac{\e^d \log(1/\e)}{n} } + \sqrt{ \frac{\e^d \log (1/\delta)}{n}}
			+ \frac{\log(1/\e)}{n \e^2}	+ \frac{\log (1/\delta) }{n \e^2} \right)
			\\&= C_1 \left( \sqrt{ \frac{\log(1/\e)}{n\e^{d+4}} } + \sqrt{\frac{\log (1/\delta)}{n \e^{d+4}}}
			+ \frac{\log(1/\e)}{n \e^{d+4}} + \frac{\log (1/\delta) }{n \e^{d+4}} \right)
		\end{align*}
		
		\myitem{Bounding $\cA_2$}
		For $i,j=1,\dots,d$, write
		\begin{align*}
			\left[ \cA_2 \right]_{ij}
			= \sum_{\ell=1}^{D} \left[ \nabla_{\R^D} \T_{n,\e} [f] (x) - \nabla_{\R^D} \T_{\e} [f] (x) \right]_\ell
			\left[b_{ij}\right]_\ell
			= b_{ij}^\top \left( \nabla_{\R^D} \T_{n,\e} [f] (x) - \nabla_{\R^D} \T_{\e} [f] (x) \right)
		\end{align*}
		to express
		\begin{align*}
			\sup_{x \in \M} \left| \left[ \cA_2 \right]_{ij} \right|
			&= \sup_{g \in \msF_{\e,{ij}}} \frac{1}{\e^{d+1}} \left| \frac{1}{n} \sum_{k=1}^{n} g(X_k)
			- \E_\P g (X) \right|,
		\end{align*}
		where 
		$\msF_{\e,{ij}} = \Big\{ g_x(\cdot) = K\Big(\frac{\|\cdot - x\|_{\R^D}}{\e}\Big) b_{ij}^\top \Big(\frac{\cdot - x}{\e}\Big) f(\cdot) : x \in \M\Big\}.$
		Similar to the proof of part (b) in Theorem \ref{deriv:unnor}, for all sufficiently small $\e$,
		we have
		\begin{align*}
			\sup_{g \in \msF_{\e,{ij}}} \|g\|_\infty 
			&\;\leq\;  \frac{C_{\sff} \|f\|_\infty \mathrm{diam}(\M)}{\e}, \\
			\sup_{g \in \msF_{\e,{ij}}} \E_\P g^2(X)
			&\;\leq\;  \e^d C_{\sff}^2 \|\rho\|_\infty \|f\|^2_\infty \left( 1 + O(\e^2) \right), \\
			\|g_x - g_{x'}\|_\infty
			&\;\leq\;  \frac{C_\sff \|f\|_\infty \diam(\M)}{\e^2} \dist_\M (x,x'), \\
			\sup_Q \mathscr{N} \left( \msF_{\e,{ij}}, L^2(Q), \eta \right)
			& \;\leq\;  \left( \frac{ A_\M \frac{C_\sff \|f\|_\infty \diam(\M)}{\e^2} }{\eta} \right)^d,
		\end{align*}
		where $ C_{\sff} = \max_{i,j=1,\cdots,d} \| b_{ij} \|_\F $.        
		So, we can find $\e_2$, $C_2 >0$ such that if $\e < \e_2$ we have
		\begin{align*}
			\max_{i,j=1,\dots,d} \sup_{x \in \M} \left| \left[\cA_2\right]_{ij} \right|
			&\leq C_2 \left( \sqrt{ \frac{\log(1/(\e\wedge\delta))}{n\e^{d+2}} }  
			+ \frac{\log(1/(\e\wedge\delta))}{n \e^{d+2}} \right)
		\end{align*}
	\end{myenumerate}
	
	Therefore, combining the above two upper bounds completes the proof.
\end{proof}

\subsection{Proof of Theorem \ref{deriv:nor}}

\subsubsection{Proof of Theorem \ref{deriv:nor} for gradients}

\begin{proof}[Proof of part (a) of Theorem \ref{deriv:nor} for gradients]
	By the product rule for gradients, and Theorems~\ref{unif:bias} and \ref{deriv:unnor},
	we have for all sufficiently small $\e$,
	\begin{align*}
		\nabla_{\M} \TT_{\e} [f] (x)
		&= \frac{1}{\T_{\e} [1] (x)} \nabla_{\M} \T_{\e} [f] (x) - \frac{\TT_{\e} [f] (x)}{\T_{\e} [1] (x)} \nabla_{\M} \T_{\e} [1] (x)
		\\&= \left( \frac{1}{ \rho(x)} + O(\e^2) \right) 
		\left(  \nabla_\M (\rho f) (x)  +  O(\e^2) \right)
		\\& \hspace*{2cm}- \left( \frac{1}{ \rho(x)} + O(\e^2) \right)
		\left( f (x) + O(\e^2) \right)
		\left(  \nabla_\M \rho (x)  +  O(\e^2) \right)
		\\&= \nabla_{\M} f(x) + \frac{f(x)}{\rho(x)} \nabla_{\M} \rho(x)
		- \frac{f(x)}{\rho(x)} \nabla_{\M} \rho(x)+  O(\e^2)
		= \nabla_{\M} f(x) + O(\e^2).
	\end{align*}
	This concludes the proof.
\end{proof}
	
\begin{proof}[Proof of part (b) of Theorem \ref{deriv:nor} for gradients]
	First, define the following:
	\begin{align*}
		\rS^{(1)}_{n,\e} (x) &= \frac{1}{n\e^d} \sum_{i=1}^{n} K\left( \frac{\|X_i-x\|_{\R^D}}{\e}\right) J(x)^\top \frac{X_i-x}{\e} \cdot \frac{f(X_i)-f(x)}{\e}, \\
		\rH^{(1)}_{n,\e} (x) &= \frac{1}{n\e^{d+1}} \sum_{i=1}^{n} K\left( \frac{\|X_i-x\|_{\R^D}}{\e}\right) J(x)^\top \frac{X_i-x}{\e}, \\
		\rS^{(1)}_{\e} (x) &= \E_\P \rS^{(1)}_{n,\e} (x) = \frac{1}{\e^d} \int_{\M} K\left( \frac{\|u-x\|_{\R^D}}{\e}\right) J(x)^\top \frac{u-x}{\e} \cdot \frac{f(u)-f(x)}{\e} d \P(u) , \\
		\rH^{(1)}_{\e} (x) &= \E_\P \rH^{(1)}_{n,\e} (x) = \frac{1}{\e^{d+1}} \int_{\M} K\left( \frac{\|u-x\|_{\R^D}}{\e}\right) J(x)^\top \frac{u-x}{\e}, \\
		\rSS^{(1)}_{n,\e} (x) &= \frac{\rS^{(1)}_{n,\e} (x)}{\T_{n,\e} [1] (x)} ,\quad
		\rHH^{(1)}_{n,\e} (x) = \frac{\rH^{(1)}_{n,\e} (x)}{\T_{n,\e} [1] (x)}, \\
		\rSS^{(1)}_{\e} (x) &= \frac{\rS^{(1)}_{\e} (x)}{\T_{\e} [1] (x)} ,\quad
		\rHH^{(1)}_{\e} (x) = \frac{\rH^{(1)}_{\e} (x)}{\T_{\e} [1] (x)}.
	\end{align*}    
	From
	\begin{align*}
		\nabla_{\R^D} \TT_{n, \e} [f] (x) - \nabla_{\R^D} \TT_{\e} [f] (x)
		=&\left\{ \frac{1}{\T_{n,\e} [1] (x)} \nabla_{\R^D} \T_{n,\e}[f](x) - \frac{\TT_{n,\e}[f](x)}{\T_{n,\e}[1] (x)} \nabla_{\R^D} \T_{n,\e} [1] (x) \right\}
		\\&- \left\{ \frac{1}{\T_{\e} [1] (x)} \nabla_{\R^D} \T_{\e} [f] (x) - \frac{\TT_{\e} [f] (x)}{\T_{\e} [1] (x)} \nabla_{\R^D} \T_{\e} [1] (x) \right\},
	\end{align*} 
	we can write
	\begin{align*}
		\nabla_{\R^d} (\TT_{n, \e} [f] \circ \Exp_x) (0) - \nabla_{\R^d} (\TT_{n, \e} [f] \circ \Exp_x) (0)
		= \cA_1 + \cA_2 + \cA_3,
	\end{align*}
	where
	\begin{align*}
		\cA_1 &= \rSS^{(1)}_{n,\e} (x) - \rSS^{(1)}_{\e} (x), \\
		\cA_2 &= -\left( \TT_{\e} [f] (x) - f(x) \right) \left( \rHH^{(1)}_{n,\e} (x)
		- \rHH^{(1)}_{\e} (x) \right), \\
		\cA_3 &= - \big( \TT_{n,\e} [f] (x) - \TT_{\e} [f] (x)\big) \rHH^{(1)}_{n,\e} (x).
	\end{align*}
	Using the similar arguments as in the proof of Theorem~\ref{deriv:unnor}, we can bound each term properly.
	This completes the proof.
\end{proof}

\subsubsection{Proof of Theorem \ref{deriv:nor} for Hessians}

\begin{proof}[Proof of part (a) of Theorem \ref{deriv:nor} for Hessians]
	Let $\T_{\e}[f] \circ \Exp_x = p_x$ and  $\T_{\e} [1] \circ \Exp_x = q_x$.
	Then, by the product rule of Hessians, we have
	\begin{align*}
	& \nabla_{\R^d}^2 (\TT_{\e} [f] \circ \Exp_x) (0)
	\\=& \frac{1}{q_x (0)} \nabla^2_{\R^d} p_x (0)
	-\frac{1}{q_x^2 (0)} \nabla_{\R^d} p_x(0) \nabla_{\R^d} q_x(0)^\top
	-\frac{1}{q_x^2 (0)} \nabla_{\R^d} q_x(0) \nabla_{\R^d} p_x(0)^\top
	\\& +\frac{2 p_x(0)}{q_x^3 (0)} \nabla_{\R^d} q_x(0) \nabla_{\R^d} q_x(0)^\top
	-\frac{p_x(0)}{q_x^2 (0)} \nabla^2_{\R^d} q_x(0).
	\end{align*}
	Now, using Theorems~\ref{unif:sto} and \ref{deriv:unnor}, we can check the convergence of each term.
	This completes the proof.
\end{proof}

\begin{proof}[Proof of part (b) of Theorem \ref{deriv:nor} for Hessians]
	Write
	\begin{align*}
		\nabla_{\R^d}^2 (\TT_{n,\e} [f] \circ \Exp_x) (0) - \nabla_{\R^d}^2 (\TT_{\e} [f] \circ \Exp_x) (0)
		= \cA_1 + \cA_1,
	\end{align*}
	where
	\begin{align*}
		\cA_1 &= J(x)^\top \left(\nabla_{\R^D}^2\TT_{n,\e}[f](x) - \nabla_{\R^D}^2 \TT_{\e} [f] (x) \right) J(x), \\
		\cA_2 &= \sum_{\ell=1}^{D} \left[ \nabla_{\R^D} \TT_{n, \e} [f] (x) - \nabla_{\R^D} \TT_{\e} [f] (x)
		\right]_{\ell} \nabla_{\R^d}^2 \left[ \Exp_x \right]_\ell (0).
	\end{align*}
	\begin{myenumerate}
		\myitem{Bounding $\cA_1$}
		Define
		\begin{align*}
			\rS_{n,\e}^{(2)} (x) =& \frac{1}{n\e^{d+1}} \sum_{i=1}^{n} K\left(\frac{\|X_i-x\|_{\R^D}}{\e}\right) \left\{ J(x)^\top \left(\frac{X_i-x}{\e}\right)  
			\left(\frac{X_i-x}{\e}\right)^\top J(x) - I_d \right\} \cdot \frac{f(X_i)-f(x)}{\e}, \\
			\rH_{n,\e}^{(2)} (x) =& \frac{1}{n\e^{d+2}} \sum_{i=1}^{n} K\left(\frac{\|X_i-x\|_{\R^D}}{\e}\right) \left\{ J(x)^\top \left(\frac{X_i-x}{\e}\right)  
			\left(\frac{X_i-x}{\e}\right)^\top J(x) - I_d \right\}, \\
			\rH_{n,\e}^{(1)} (x) =& \frac{1}{n\e^{d+1}} \sum_{i=1}^{n} K\left(\frac{\|X_i-x\|_{\R^D}}{\e}\right) 
			J(x)^\top \left(\frac{X_i-x}{\e}\right), \\
			\rS^{(2)}_{\e} (x) =& \E_\P \rS^{(2)}_{n,\e} (x),\quad
			\rH^{(2)}_{\e} (x) = \E_\P \rH^{(2)}_{n,\e} (x), \quad
			\rH^{(1)}_{\e} (x) = \E_\P \rH^{(1)}_{n,\e} (x), \\
			\rSS^{(2)}_{n,\e} (x) =& \frac{\rS^{(2)}_{n,\e} (x)}{\T_{n,\e} [1] (x)},\quad
			\rHH^{(2)}_{n,\e} (x) = \frac{\rH^{(2)}_{n,\e} (x)}{\T_{n,\e} [1] (x)}, \quad
			\rHH^{(1)}_{n,\e} (x) = \frac{\rH^{(1)}_{n,\e} (x)}{\T_{n,\e} [1] (x)}, \\
			\rSS^{(2)}_{\e} (x) =& \frac{\rS^{(2)}_{\e} (x)}{\T_{\e} [1] (x)},\quad
			\rHH^{(2)}_{\e} (x) = \frac{\rH^{(2)}_{\e} (x)}{\T_{\e} [1] (x)}, \quad
			\rHH^{(1)}_{\e} (x) = \frac{\rH^{(1)}_{\e} (x)}{\T_{\e} [1] (x)}.
		\end{align*}        
		From
		\begin{align*}
			\nabla_{\R^D}^2 \TT_{\e} [f] (x)
			=& \frac{1}{\T_{\e} [1] (x)} \nabla^2_{\R^D} \T_{\e} [f] (x)
			-\frac{1}{\left\{\T_{\e} [1] (x)\right\}^2} \nabla_{\R^D} \T_{\e} [f] (x) \nabla_{\R^D} \T_{\e} [1] (x)^\top
			\\& -\frac{1}{\left\{\T_{\e} [1] (x)\right\}^2} \nabla_{\R^D} \T_{\e} [1] (x) \nabla_{\R^D} \T_{\e} [f] (x)^\top
			+\frac{2 \T_{\e} [f] (x)}{\left\{\T_{\e} [1] (x)\right\}^3 } \nabla_{\R^D} \T_{\e} [1] (x)\nabla_{\R^D} \T_{\e} [1] (x)^\top
			\\& -\frac{\T_{\e} [f] (x)}{\left\{\T_{\e} [1] (x)\right\}^2 } \nabla^2_{\R^D} \T_{\e} [1] (x),
		\end{align*}
		we can rewrite
		$ \cA_1 = \mathcal{B}_1 + \mathcal{B}_2 + \mathcal{B}_3 + \mathcal{B}_4 + \mathcal{B}_4^\top, $
		where
		\begin{align*}
			\mathcal{B}_1 =& \rSS^{(2)}_{n,\e} (x) - \rSS^{(2)}_{\e} (x), \\
			\mathcal{B}_2 =& - \left\{ \TT_{\e}[f](x) - f(x) \right\} \left\{ \rHH^{(2)}_{n,\e} (x) - \rHH^{(2)}_{\e} (x) \right\}, \\
			\mathcal{B}_3 =& - \left\{ \TT_{n,\e} [f] (x) - \TT_{\e} [f] (x) \right\} \rHH^{(2)}_{n,\e} (x), \\
			\mathcal{B}_4 =& J(x) \nabla_{\R^D} \TT_{n,\e} [f] (x) \rHH^{(1)}_{n,\e} (x) - J(x) \nabla_{\R^D} \TT_{\e} [f] (x) \rHH^{(1)}_{\e} (x).
		\end{align*}
		Then, similar to the proofs of Theorem~\ref{deriv:unnor}, 
		we can find $\e_1, C_1, C_1^*>0$ such that if $ \e^*_{n,\delta}(C_1^*) < \e < \e_1,$    
		with probability at least $1-\delta/2$, we have
		\begin{equation*}
			\sup_{x \in \M} \|\cA_1\|_\F \leq 
			C_1 \left( 
			\sqrt{ \frac{\log(1/(\e\wedge\delta))}{n\e^{d+2}} }  
			+ \frac{\log(1/(\e\wedge\delta)}{n \e^{d+4}} \right)
		\end{equation*}
		
		\myitem{Bounding $\cA_2$}
		Consider
		$
		\left[ \cA_2 \right]_{ij} = b_{ij}^\top \left( \nabla_{\R^D} \TT_{n, \e} [f] (x) - \nabla_{\R^D} \TT_{\e} [f] (x) \right)
		$
		for $i,j = 1, \dots, d$.
		Then, we can write
		\begin{align*}
			b_{ij}^\top \left( \nabla_{\R^D} \TT_{n, \e} [f] (x) - \nabla_{\R^D} \TT_{n, \e} [f] (x) \right)
			= \mathcal{B}_{5,ij} + \mathcal{B}_{6,ij} + \mathcal{B}_{7,ij},
		\end{align*}
		where
		\begin{align*}
			\mathcal{B}_{5,ij} &=  \rSS^{(1,ij)}_{n,\e} (x) - \rSS^{(1,ij)}_{\e} (x), \\
			\mathcal{B}_{6,ij} &= -\left( \TT_{\e} [f] (x) - f(x) \right) \left( \rHH^{(1,ij)}_{n,\e} (x) - \rHH^{(1,ij)}_{\e} (x) \right), \\
			\mathcal{B}_{7,ij} &= -\left( \TT_{n,\e} [f] (x) - \TT_{\e} [f] (x)\right) \rHH^{(1,ij)}_{n,\e} (x), \\
			\rS^{(1,ij)}_{n,\e} (x) &= \frac{1}{n\e^d} \sum_{i=1}^{n} K\left( \frac{\|X_i-x\|_{\R^D}}{\e}\right) b_{ij}^\top \frac{X_i-x}{\e} \cdot \frac{f(X_i)-f(x)}{\e}, \\
			\rH^{(1,ij)}_{n,\e} (x) &= \frac{1}{n\e^{d+1}} \sum_{i=1}^{n} K\left( \frac{\|X_i-x\|_{\R^D}}{\e}\right) b_{ij}^\top \frac{X_i-x}{\e}, \\
			\rS^{(1,ij)}_{\e} (x) &= \E_\P \rS^{(1,ij)}_{n,\e} (x) = \frac{1}{\e^d} \int_{\M} K\left( \frac{\|u-x\|_{\R^D}}{\e}\right) b_{ij}^\top \frac{u-x}{\e} \cdot \frac{f(u)-f(x)}{\e} d \P(u), \\
			\rH^{(1,ij)}_{\e} (x) &= \E_\P \rH^{(1,ij)}_{n,\e} (x) = \frac{1}{\e^{d+1}} \int_{\M} K\left( \frac{\|u-x\|_{\R^D}}{\e}\right) b_{ij}^\top \frac{u-x}{\e} \cdot \frac{f(u)-f(x)}{\e} d \P(u), \\
			\rSS^{(1,ij)}_{n,\e} (x) &= \frac{\rS^{(1,ij)}_{n,\e} (x)}{\T_{n,\e} [1] (x)} ,\quad
			\rHH^{(1,ij)}_{n,\e} (x) = \frac{\rH^{(1,ij)}_{n,\e} (x)}{\T_{n,\e} [1] (x)}, \\
			\rSS^{(1,ij)}_{\e} (x) &= \frac{\rS^{(1,ij)}_{\e} (x)}{\T_{\e} [1] (x)} ,\quad
			\rHH^{(1,ij)}_{\e} (x) = \frac{\rH^{(1,ij)}_{\e} (x)}{\T_{\e} [1] (x)}.
		\end{align*}
		Then, similar to the proof of the gradient case, we can bound each $\mathcal{B}_{5,ij}$, $\mathcal{B}_{6,ij}$, $\mathcal{B}_{7,ij}$ similar to $\cA_1$.
	\end{myenumerate}
	Combining the above results completes the proof.
\end{proof}

\subsection{Proof of Corollary \ref{reg:unif}}

\begin{proof}[Proof of part (a) of Corollary \ref{reg:unif}]
	Note that
	\begin{align*}
		\rho_{X,\e} (x):= \int_{\R \times \M} K_\e \left( \|u-x\|_{\R^D} \right) d\P_{X,Y} (u,y)
		&= \int_{\M} K_\e \left( \|u-x\|_{\R^D} \right) d \P_X (x),
	\end{align*}
	and
	\begin{align*}
		\int_{\R \times \M} y K_\e (\|u-x\|_{\R^D}) d \P_{X,Y} (u,y)
		&= \int_{\M} K_\e \left( \|u-x\|_{\R^D} \right) \left( \int_\R y\, d \P_{Y|X}(y|u) \right) d \P_X (u)
		\\&= \int_{\M} K_\e \left( \|u-x\|_{\R^D} \right) \varphi(u) d \P_X (u).
	\end{align*}
	These yield
	\begin{align*}
		\varphi_\e (x) 
		&= \frac{\int_{\R \times \M} y K_\e (\|u-x\|_{\R^D}) d \P_{X,Y} (u,y)}{\int_{\R \times \M} K_\e \left( \|u-x\|_{\R^D} \right) d\P_{X,Y} (u,y)}
		= \frac{\int_{\M} K_\e \left( \|u-x\|_{\R^D} \right) \varphi (u) d \P_X (u)}{\int_{\M} K_\e \left( \|u-x\|_{\R^D} \right) d \P_X (u)}
		= \TT_\e [\varphi] (x).
	\end{align*}	
	Therefore, the result follows directly from Theorems~\ref{unif:bias} and \ref{deriv:nor}.
\end{proof}

\begin{proof}[Proof of part (b) of Corollary \ref{reg:unif}]
	From the proof of (a), we have
	\begin{align*}
		\E_{\P_{X,Y}} [Y K_\e( \|X-x\|_{\R^D})] = \rho_{X,\e}(x) \varphi_\e (x),
	\end{align*}
	and this gives
	\begin{align*}
		\frac{1}{n} \sum_{i=1}^n Y_i K_\e (\|X_i-x\|_{\R^D})
		= \frac{1}{n} \sum_{i=1}^n Y_i K_\e (\|X_i-x\|_{\R^D}) - \E_{\P_{X,Y}} [Y K_\e( \|X-x\|_{\R^D})] + \rho_{X,\e} (x) \varphi_\e (x).
	\end{align*}
	Thus, we can write the difference $\varphi_{n,\e} (x) - \varphi_\e(x)$ as
	\begin{align*}
		\varphi_{n,\e} (x) - \varphi_\e(x) =& \frac{1}{\rho_{X,n,\e} (x)} \left( \cA_1 - \varphi_\e(x) \cA_2 \right),
	\end{align*}
	where
	\begin{align*}
		\rho_{X,n,\e} (x) &= \frac{1}{n} \sum_{i=1}^n K_\e (\|X_i-x\|_{\R^D}), 
		\\ \cA_1 &= \sum_{i=1}^n Y_i K_\e (\|X_i-x\|_{\R^D}) - \E_{\P_{X,Y}} [Y K_\e( \|X-x\|_{\R^D})], 
		\\ \cA_2 &= \rho_{X,n,\e} (x) - \rho_{X,\e} (x).
	\end{align*}
	
	We can bound $1/\rho_{X,n,\e} (x)$, $\varphi_\e(x)$ and $\cA_2$ using (a), Theorem~\ref{unif:sto} and its proof.
	Thus, it suffices to bound $\cA_1$.
	
	Now, define a function class
	\begin{equation*}
		\msF_\e = \left\{ g_x : g_x(u,y) = y \cdot \mathds{1} (|y| \leq C_Y) \cdot K \left( \frac{\|u-x\|_{\R^D}}{\e} \right), x \in \M, \right\}.
	\end{equation*}	
	Then, similar to the proof of Theorem~\ref{unif:sto}, we can show        
	\begin{gather*}
		\sup_{g \in \msF_\e} \|g \|_\infty \leq C_1, \\
		\sup_{g \in \msF_\e} \E_{\P_{X,Y}} g^2(X,Y) \leq C_2 \e^d, \\
		\sup_\Q \mathscr{N} \left( \msF_\e, L^2(\Q), \eta \right)
		\leq \left( \frac{A_{\M} \frac{C_1}{\e} }{\eta} \right)^{d},
	\end{gather*}
	for some $C_1, C_2 > 0$ and for all sufficiently small $\e>0$.
	Therefore, we can obtain
	\begin{align*}
		\sup_{x \in \M} \left| \cA_1 \right| \leq
		C_0 \bigg( \sqrt{ \frac{\log(1/(\e\vee\delta))}{n\e^{d}} } + \frac{\log(1/(\e\vee\delta))}{n \e^{d}} \bigg),
	\end{align*}
	which completes the proof for $k=0$.
	The result for $k=1,2$ can be shown by similar arguments and the proof of Theorem~\ref{deriv:unnor}.
\end{proof}

\subsection{Proof of Corollary \ref{reg:Berry}}

The following lemma will be used for the asymptotic variance of kernel regression.
\begin{lemma} \label{reg:Berry:var}
	Under the conditions in Corollary~\ref{reg:Berry}, there exist $\e_0, C_0 >0$ such that $\e < \e_0$ implies
	\begin{gather}
		\sup_{x \in \M} \left| \frac{1}{\e^d} \int_{\M \times \R} y K^2 \left( \frac{\|u-x\|_{\R^D}}{\e} \right) \dP_{X,Y} (u,y)
		- \frac{\rho_X(x) \varphi (x)}{(4\pi)^{d/2}} \right| \leq C_0 \e, \\
		\sup_{x \in \M} \left| \frac{1}{\e^d} \int_{\M \times \R} y^2 K^2 \left( \frac{\|u-x\|_{\R^D}}{\e} \right) \dP_{X,Y} (u,y)
		- \frac{\rho_X(x) \psi(x) }{(4\pi)^{d/2}} \right| \leq C_0 \e,
	\end{gather}
	where $\psi(x)= \E[Y^2|X=x]$, and the constants $\e_0$ and $C_0$ are $\M$-dependent, and they also depend on the $C^2$-norms of $\varphi$, $\psi$ and $\rho_X$.
\end{lemma}

\begin{proof}[Proof of Lemma \ref{reg:Berry:var}]
	We see that
	\begin{align*}
		\frac{1}{\e^d} \int_{\M \times \R} y K^2 \left( \frac{\|u-x\|_{\R^D}}{\e} \right) \dP_{X,Y} (u,y)
		&= \frac{1}{\e^d} \int_\M K^2 \left( \frac{\|u-x\|_{\R^D}}{\e} \right) \left( \int_\R y d \P_{Y|X} (y|u) \right) \dP_{X} (u)
		\\&= \frac{1}{\e^d} \int_\M K^2 \left( \frac{\|u-x\|_{\R^D}}{\e} \right) \varphi(u) \dP_{X} (u),
	\end{align*}
	and similarly,	
	\begin{align*}
		\frac{1}{\e^d} \int_{\M \times \R} y^2 K^2 \left( \frac{\|u-x\|_{\R^D}}{\e} \right) \dP_{X,Y} (u,y)
		&= \frac{1}{\e^d} \int_\M K^2 \left( \frac{\|u-x\|_{\R^D}}{\e} \right) \psi(u) \dP_{X} (u).
	\end{align*}
	Thus, applying Lemma~\ref{unif:sto:var} and using its proof, we can complete the proof.
\end{proof}

\begin{proof}[Proof of Corollary \ref{reg:Berry}]
For simplicity, assume $m=2$. Using the notation in the proof of Corollary~\ref{reg:unif}, we have
\begin{equation*}
	\bfZ_{n,\e,f} = \sqrt{n \e^d} 
	\begin{pmatrix}
	\varphi_{n,\e} (x_1) - \varphi_{\e} (x_1) \\
	\varphi_{n,\e} (x_2) - \varphi_{\e} (x_2)
	\end{pmatrix}
	= S_n + R_n S_n
\end{equation*}
where for $i=1,\ldots,n$ and $j=1,2$,
\begin{align*}
	Z_i &= (Z_{i1}, Z_{i2})^\top, \quad S_n = \frac{1}{\sqrt{n}} \sum_{i=1}^n Z_i, \quad \Sigma_{f,\e} = \var_\P (Z_1), \\
	Z_{ij} &= \frac{1}{\rho_{X,\e}(x_j)} \left( V_{ij} - \varphi_\e (x_j) U_{ij} \right), \\
	U_{ij} &= \sqrt{\e^d} \left( \frac{1}{\e^d} K \left( \frac{\|X_i-x_j\|_{\R^D}}{\e} \right)
	- \E_{\P_{X,Y}} \left[ \frac{1}{\e^d} K \left( \frac{\|X - x_j\|_{\R^D}}{\e} \right) \right] \right), \\
	V_{ij} &= \sqrt{\e^d} \left( \frac{1}{\e^d} Y_i \, K \left( \frac{\|X_i - x_j\|_{\R^D}}{\e} \right)
	- \E_{\P_{X,Y}} \left[ \frac{1}{\e^d} Y \, K \left( \frac{\|X - x_j\|_{\R^D}}{\e} \right) \right] \right), \\
	R_n &= \diag \left( \frac{1}{\rho_{X,n,\e}(x_1)} - \frac{1}{\rho_{X,\e}(x_1)},\;\; \frac{1}{\rho_{X,n,\e}(x_2)} - \frac{1}{\rho_{X,\e}(x_2)}  \right).
\end{align*}
Using Theorem~\ref{unif:bias} and Lemma~\ref{reg:Berry:var}, it is straightforward that $\E_\P U_{ij} = 0,$ $\E_\P V_{ij} = 0, $
\begin{gather*}
	\frac{1}{\rho_{X,\e} (x_j)} = \frac{1}{\rho_X(x_j)} + O(\e^2), \quad
	\varphi_\e (x_j) = \varphi (x_j) + O(\e^2), \\
	\E_\P U_{ij}^2 = \frac{\rho_X(x_j)}{(4\pi)^{d/2}} + O(\e), \quad
	\E_\P V_{ij}^2 = \frac{\rho_X(x_j) \psi(x_j)}{(4\pi)^{d/2}} + O(\e), \\
	\E_\P U_{i1} U_{i2} = O(\e^d), \quad
	\E_\P V_{i1} V_{i2} = O(\e^d), \\
	\E_\P U_{i1} V_{i1} = \frac{\rho_X(x_j)\varphi(x_j)}{(4\pi)^{d/2}} + O \left( \e \right), \quad
	\E_\P U_{i1} V_{i2} = O(\e^d), \\
	\E_\P \left| U_{ij} \right|^3 \leq \frac{8 \Theta_d \rho(x_j) + O(\e^{2\wedge d})}{\sqrt{\e^d}}, \quad
	\E_\P \left| V_{ij} \right|^3 \leq \frac{8 \Theta_d' \rho(x_j) \left\|\nabla_\M f(x_j)\right\|^3 + O(\e)}{\sqrt{\e^d}},
\end{gather*}

We see that using $\var(Y|X=x_j) = \psi(x_j) - \varphi^2(x_j)$,
\begin{align*}
	\var_{\P_{X,Y}} (Z_{ij})
	&= \frac{1}{\rho^2_X (x_j)} \left( \frac{\rho_X(x_j) \psi(x_j)}{(4\pi)^{d/2}} -2 \varphi (x_j) \cdot \frac{\rho_X(x_j) \varphi(x_j)}{(4\pi)^{d/2}} + \frac{\rho_X (x_j) \varphi^2(x_j)}{(4\pi)^{d/2}} \right) 
	\\&= \frac{\var(Y|X=x_j)}{(4\pi)^{d/2} \rho_X (x_j)}.
\end{align*}

Therefore, by following the same arguments in the proof of Theorems~\ref{Berry:unnor} and \ref{Berry:nor},
we can completes the proof.
\end{proof}

\subsection{Proof of Corollaries~\ref{HKS:unif} and \ref{HKS:Berry}}

Observe that
\begin{align*}
	\widehat{H}_{n,N,\eta,\tau,\e} (x) - H_\tau (x) 
	&= \TT_{n,\e} \left[ \tilde{H}_{n,N,\eta,\tau} \right] (x) - H_\tau (x)
	\\&= \TT_{n,\e} \big[ \tilde{H}_{n,N,\eta,\tau}  - H_\tau \big] (x)
	+ \left( \TT_{n,\e} [H_\tau ] (x) - H_\tau (x)  \right).
\end{align*}
Since the second term is simply the normalized kernel smoothing of $H_\tau$, it suffices to control the first term.
Therefore, to complete the proof, it is enough to apply the the following lemma:
\begin{lemma} \label{kernel:lemma:hks_diff}
	Suppose Assumption \ref{smoothing:assume:manifold}, \ref{smoothing:assume:density}, \ref{smoothing:assume:kernel} hold.
	Then, for any function $f: \M \rightarrow \R$ and $k=0,1,2$, we have
	\begin{gather*}
		\sup_{x \in \M} \left\| \nabla_\M^k \TT_{n, \e}[f] (x) \right\| \leq \frac{C_0 \|f\|_{\cX_n}}{\e^{2k}},
	\end{gather*}
	where $\|f\|_{\cX_n} = \max\limits_{i=1,\dots,n} |f(X_i)| $ and $C_0>0$ is $\M$-dependent. 
	In particular, for $k=0$, we take $C_0=1$.
\end{lemma}

\begin{proof}[Proof of Lemma \ref{kernel:lemma:hks_diff}]
	First, by the definition of normalized kernel smoothing, it is clear that
	\begin{align*}
		\left| \TT_{n, \e}[f] (x) \right|
		&\leq \frac{\sum_{i=1}^{n} K_\e ( \|X_i-x\|_{\R^D} ) \left| f(X_i) \right|}{\sum_{i=1}^{n} K_\e ( \|X_i-x\|_{\R^D} )}
		\leq \|f\|_{\cX_n}.
	\end{align*}
	
	Next, for the difference of the gradients, using the arguments in the proof of Theorem \ref{deriv:nor}, we have
	\begin{align*}
		\left\| \nabla_\M \TT_{n, \e}[f] (x) \right\|
		&\leq \frac{C_1}{\T_{n,\e} [1] (x)} \left\|  \nabla_{\R^D} \T_{n,\e} [f] (x) \right\|_{\R^D}
		\\&\quad +
		\frac{C_1 \left|\TT_{n,\e} [f] (x)\right|}{\T_{n,\e} [1] (x)} \left\| \nabla_{\R^D} \T_{n,\e} [1] (x) \right\|_{\R^D}
		\\&\leq \frac{C_2 \|f\|_{\cX_n}}{\e^2}
	\end{align*}
	for some $C_1, C_2 >0$. Similarly, the same argument applies for $k=2$, which completes the proof.
\end{proof}

\end{document}

%% file: smoothing_preamble.tex
\usepackage[utf8]{inputenc}
\usepackage[T1]{fontenc}
\usepackage{amsmath, amsfonts, amssymb, amsthm, dsfont, mathtools, mathrsfs}
\usepackage[hypertexnames=false]{hyperref}
    \hypersetup{
      colorlinks   = true, 
      urlcolor     = blue, 
      linkcolor    = blue, 
      citecolor    = red 
    }
\usepackage[nameinlink]{cleveref}
\usepackage{enumitem}
\usepackage{xcolor,colortbl} 
\usepackage{graphicx}
\usepackage{caption}
\usepackage{subcaption}
\usepackage{authblk}
\usepackage[thinlines]{easytable}
\usepackage{tikz}
\usetikzlibrary{arrows}
\usepackage{pdflscape} 
\usepackage{nccmath} 
\usepackage{thmtools}
\usepackage{thm-restate}
\usepackage{calc}
\usepackage{accents}
\usepackage{multirow}
\usepackage{todonotes}

\usepackage[sort]{natbib}
\usepackage{usebib}
\newtheorem{theorem}{Theorem}[section]
\newtheorem{lemma}[theorem]{Lemma}
\newtheorem{proposition}[theorem]{Proposition}
\newtheorem{corollary}[theorem]{Corollary}
\newtheorem{assumption}[theorem]{Assumption}

\theoremstyle{definition}
    \newtheorem{definition}[theorem]{Definition}
    
    \newtheorem{remark}[theorem]{Remark}

\usepackage[margin=1in]{geometry}
\allowdisplaybreaks

\providecommand{\keywords}[1]
{
	\small	
	\textbf{Keywords---} #1
}

%% file: smoothing_commands.tex
\newcommand{\Romannum}[1]{\textup{\uppercase\expandafter{\romannumeral#1}}}
\newcommand{\romannum}[1]{\textup{\lowercase\expandafter{\romannumeral#1}}}

\newcommand{\tmin}{{\mathrm{min}}}
\newcommand{\set}[1]{\left\{ #1 \right\}}
\newcommand{\dist}{{\operatorname{dist}}}
\newcommand{\op}{{\operatorname{op}}}
\newcommand{\F}{{\operatorname{F}}}

\newcommand{\diag}{{\operatorname{diag}}}

\newcommand{\tr}{{\operatorname{tr}}}
\newcommand{\E}{\mathbb{E}}
\newcommand{\var}{\mathrm{Var}}

\renewcommand{\div}{\mathrm{div}}

\newcommand{\dP}{d\mathbb{P}}

\newcommand{\Exp}{\mathrm{Exp}}

\newcommand{\Ric}{\mathrm{Ric}}
\newcommand{\cvx}{\mathrm{cvx}}

\newcommand{\e}{\varepsilon}

\newcommand{\M}{\mathbb{M}}

\renewcommand{\P}{\mathbb{P}}
\newcommand{\Q}{\mathbb{Q}}
\newcommand{\R}{\mathbb{R}}

\newcommand{\rH}{\mathrm{H}}
\newcommand{\rHH}{\overline{\rH}}
\newcommand{\rS}{\mathrm{S}}
\newcommand{\rSS}{\overline{\rS}}
\newcommand{\T}{\mathbf{T}}
\newcommand{\TT}{\overline{\T}}

\newcommand{\rn}[1]{\textup{\lowercase\expandafter{\romannumeral#1}}}
\newcommand{\dvol}{d \mathrm{vol}_{\M}}
\newcommand{\sff}{\mathrm{I\!I}}
\newcommand{\diam}{\operatorname{diam}}

\renewcommand{\tilde}{\widetilde}

\newcommand{\cA}{\mathcal{A}}
\newcommand{\cE}{\mathcal{E}}
\newcommand{\cN}{\mathcal{N}}
\newcommand{\cX}{\mathcal{X}}

\newcommand{\msF}{\mathscr{F}}
\newcommand{\msG}{\mathscr{G}}

\newcommand{\bfB}{\mathbf{B}}

\newcommand{\bfS}{\mathbf{S}}

\newcommand{\bfZ}{\mathbf{Z}}

\newcommand{\tbfZ}{\tilde{\mathbf{Z}}}

\newlist{myenumerate}{enumerate}{1}
\setlist[myenumerate,1]{
	label={},
	leftmargin=0pt
}
\newcommand{\myitem}[1]{\item \underline{{\em#1:}}\hspace{0.5em}}

\renewcommand{\cite}[1]{\citet{#1}}

\newcommand{\distcvx}{\mathsf{d}} 
\renewcommand{\citep}[1]{(\cite{#1})}

%% file: smoothing_bib.bib
@article{sun2009concise,
	title = {A {{Concise}} and {{Provably Informative Multi-Scale Signature Based}} on {{Heat Diffusion}}},
	author = {Sun, Jian and Ovsjanikov, Maks and Guibas, Leonidas},
	year = 2009,
	journal = {Computer Graphics Forum},
	volume = {28},
	number = {5},
	pages = {1383--1392},
	langid = {english}
}

@article{berenfeld_density_2021,
	title = {Density estimation on an unknown submanifold},
	volume = {15},
	number = {1},
	urldate = {2025-07-21},
	journal = {Electronic Journal of Statistics},
	author = {Berenfeld, Clément and Hoffmann, Marc},
	year = {2021},
	pages = {2179--2223}
}

@article{DiMarzio2011torus,
    title={Kernel density estimation on the torus},
    author={Di Marzio, Marco and Panzera, Angnese and Taylor, Charles C.},
    journal={Journal of Multivariate Analysis},
    year={2011},
    pages={2156-2173}
}

@article{dunson2021spectral,
	title = {Spectral convergence of graph {Laplacian} and heat kernel reconstruction in ${L}^{\infty}$ from random samples},
	volume = {55},
	journal = {Applied and Computational Harmonic Analysis},
	author = {Dunson, David B. and Wu, Hau-Tieng and Wu, Nan},
	year = {2021},
	pages = {282--336},
}

@article{chazal2017robust,
  title={Robust topological inference: Distance to a measure and kernel distance},
  author={Chazal, Fr{\'e}d{\'e}ric and Fasy, Brittany and Lecci, Fabrizio and Michel, Bertrand and Rinaldo, Alessandro and Rinaldo, Alessandro and Wasserman, Larry},
  journal={The Journal of Machine Learning Research},
  volume={18},
  number={1},
  pages={5845--5884},
  year={2017},
  publisher={JMLR. org}
}

@article{monera2014taylor,
  title={The {T}aylor expansion of the exponential map and geometric applications},
  author={Monera, Maria G and Montesinos-Amilibia, A and Sanabria-Codesal, Esther},
  journal={Revista de la Real Academia de Ciencias Exactas, Fisicas y Naturales. Serie A. Matematicas},
  volume={108},
  number={2},
  pages={881--906},
  year={2014},
  publisher={Springer}
}

@article{coifman2006diffusion,
  title={Diffusion maps},
  author={Coifman, Ronald R and Lafon, St{\'e}phane},
  journal={Applied and computational harmonic analysis},
  volume={21},
  number={1},
  pages={5--30},
  year={2006},
  publisher={Elsevier}
}

@article{berry2017density,
  title={Density estimation on manifolds with boundary},
  author={Berry, Tyrus and Sauer, Timothy},
  journal={Computational Statistics \& Data Analysis},
  volume={107},
  pages={1--17},
  year={2017},
  publisher={Elsevier}
}

@inproceedings{kim2019uniform,
  title={Uniform convergence rate of the kernel density estimator adaptive to intrinsic volume dimension},
  author={Kim, Jisu and Shin, Jaehyeok and Rinaldo, Alessandro and Wasserman, Larry},
  booktitle={International Conference on Machine Learning},
  pages={3398--3407},
  year={2019},
  organization={PMLR}
}

@inproceedings{jiang2017uniform,
  title={Uniform convergence rates for kernel density estimation},
  author={Jiang, Heinrich},
  booktitle={International Conference on Machine Learning},
  pages={1694--1703},
  year={2017},
  organization={PMLR}
}

@article{loubes2008kernel,
  title={A kernel-based classifier on a Riemannian manifold},
  author={Loubes, Jean-Michel and Pelletier, Bruno},
  journal={Statistics \& Decisions},
  volume={26},
  number={1},
  pages={35--51},
  year={2008},
  publisher={De Gruyter (A)}
}

@book{lee2018introduction,
  title={Introduction to Riemannian manifolds},
  author={Lee, John M},
  edition={2},
  year={2018},
  publisher={Springer}
}

@article{nishiyama2011ImpossibilityWeak,
	title = {Impossibility of weak convergence of kernel density estimators to a non-degenerate law in {${L}^2({\mathbb R}^d)$}},
	volume = {23},
	number = {1},
	urldate = {2024-01-08},
	journal = {Journal of Nonparametric Statistics},
	author = {Nishiyama, Yoichi},
	month = mar,
	year = {2011},
	pages = {129--135},
}

@article{henry2009KernelDensity,
	title = {Kernel Density Estimation on {R}iemannian Manifolds: Asymptotic Results},
	volume = {34},
	number = {3},
	journal = {Journal of Mathematical Imaging and Vision},
	author = {Henry, Guillermo and Rodriguez, Daniela},
	month = jul,
	year = {2009},
	pages = {235--239},
}

@article{khardani2022NonparametricRecursive,
	title = {Nonparametric recursive regression estimation on {R}iemannian Manifolds},
	volume = {182},
	journal = {Statistics \& Probability Letters},
	author = {Khardani, Salah and Yao, Anne Françoise},
	month = mar,
	year = {2022},
	pages = {109274}
}

@article{zhang2021KernelSmoothing,
	title = {Kernel Smoothing, Mean Shift, and Their Learning Theory with Directional Data},
	volume = {22},
	number = {154},
	urldate = {2024-05-21},
	journal = {Journal of Machine Learning Research},
	author = {Zhang, Yikun and Chen, Yen-Chi},
	year = {2021},
	pages = {1--92}
}

@article{wu2022strong,
  title={Strong uniform consistency with rates for kernel density estimators with general kernels on manifolds},
  author={Wu, Hau-Tieng and Wu, Nan},
  journal={Information and Inference: A Journal of the IMA},
  volume={11},
  number={2},
  pages={781--799},
  year={2022},
  publisher={Oxford University Press}
}

@article{bouzebda2023RatesStrong,
	title = {Rates of the {strong} {uniform} {consistency} for the {kernel}-{type} {regression} {function} {estimators} with {general} {kernels} on {manifolds}},
	volume = {32},
	language = {en},
	number = {1},
	urldate = {2024-05-22},
	journal = {Mathematical Methods of Statistics},
	author = {Bouzebda, Salim and Taachouche, Nourelhouda},
	year = {2023},
	pages = {27--80}
}

@article{Bai1988Directional,
    title = {Kernel estimators of density function of directional data},
    journal = {Journal of Multivariate Analysis},
    volume = {27},
    number = {1},
    pages = {24-39},
    year = {1988},
    author = {Z.D. Bai and C.Radhakrishna Rao and L.C. Zhao},
}

@article{Hall1987Spherical,
    author = {Peter Hall and G. S. Watson and Javier Cabrera},
    journal = {Biometrika},
    number = {4},
    pages = {751--762},
    publisher = {[Oxford University Press, Biometrika Trust]},
    title = {Kernel Density Estimation with Spherical Data},
    urldate = {2024-12-18},
    volume = {74},
    year = {1987}
}

@article{cheng2013local,
	title={Local linear regression on manifolds and its geometric interpretation},
	author={Cheng, Ming-Yen and Wu, Hau-Tieng},
	journal={Journal of the American Statistical Association},
	volume={108},
	number={504},
	pages={1421--1434},
	year={2013},
	publisher={Taylor \& Francis}
}

@incollection{gine2006empirical,
	title = {Empirical Graph {{Laplacian}} Approximation of {{Laplace}}--{{Beltrami}} Operators: {{Large}} Sample Results},
	booktitle = {High {{Dimensional Probability}}},
	author = {Gin{\'e}, Evarist and Koltchinskii, Vladimir},
	year = 2006,
	month = jan,
	volume = {51},
	pages = {238--260},
	publisher = {Institute of Mathematical Statistics}
}

@article{panov2015finite,
	title = {Finite {{Sample Bernstein}} -- von {{Mises Theorem}} for {{Semiparametric Problems}}},
	author = {Panov, Maxim and Spokoiny, Vladimir},
	year = 2015,
	month = sep,
	journal = {Bayesian Analysis},
	volume = {10},
	number = {3},
	pages = {665--710},
	publisher = {International Society for Bayesian Analysis}
}

@article{bentkus2003dependence,
	title = {On the Dependence of the {{Berry}}--{{Esseen}} Bound on Dimension},
	author = {Bentkus, V.},
	year = 2003,
	month = may,
	journal = {Journal of Statistical Planning and Inference},
	volume = {113},
	number = {2},
	pages = {385--402}
}

@book{horn1991topics,
	title = {Topics in {{Matrix Analysis}}},
	author = {Horn, Roger A. and Johnson, Charles R.},
	year = 1991,
	publisher = {Cambridge University Press},
	address = {Cambridge},
	isbn = {978-0-521-46713-1}
}

@article{hein2007graph,
	title={Graph Laplacians and their convergence on Random Neighborhood Graphs.},
	author={Hein, Matthias and Audibert, Jean-Yves and Luxburg, Ulrike von},
	journal={Journal of Machine Learning Research},
	volume={8},
	number={6},
	year={2007}
}

@article{block2022intrinsic,
	title = {Intrinsic {{Dimension Estimation Using Wasserstein Distance}}},
	author = {Block, Adam and Jia, Zeyu and Polyanskiy, Yury and Rakhlin, Alexander},
	year = 2022,
	journal = {Journal of Machine Learning Research},
	volume = {23},
	number = {313},
	pages = {1--37}
}

@inproceedings{farahmand2007manifoldadaptive,
	title = {Manifold-Adaptive Dimension Estimation},
	booktitle = {Proceedings of the 24th International Conference on {{Machine}} Learning},
	author = {massoud Farahmand, Amir and Szepesv{\'a}ri, Csaba and Audibert, Jean-Yves},
	year = 2007,
	month = jun,
	series = {{{ICML}} '07},
	pages = {265--272},
	publisher = {Association for Computing Machinery},
	address = {New York, NY, USA},
	isbn = {978-1-59593-793-3}
}
